\documentclass[a4paper,11pt]{amsart}
\usepackage{preamble}

\newcommand{\daniel}[1]{\textcolor{black}{#1}}
\newcommand{\shu}[1]{\textcolor{black}{#1}}
\newcommand{\mbN}{\mathbb{N}}
\newcommand{\mbZ}{\mathbb{Z}}

\author{Daniel Skodlerack}
\address{Daniel Skodlerack, Institute of Mathematical Sciences, ShanghaiTech University,
201210, Pudong New District, Shanghai, China}
\email{dskodlerack@shanghaitech.edu.cn}
\author{Shuyang Ye}
\address{Shuyang Ye, Institute of Mathematical Sciences, ShanghaiTech University,
201210, Pudong New District, Shanghai, China}
\email{ye.shuyang@outlook.com}

\title[Compatible beta-extensions]{Semisimple types for quaternionic forms of p-adic classical groups and compatible beta-extensions}


\begin{document}
\maketitle

\begin{abstract}
Let~$G$ be a quaternionic form of a~$p$-adic classical group ($p$ odd). We construct a Bushnell--Kutzko-Stevens type for every Bernstein block of the category of smooth complex representations of~$G$. Further we construct a system of compatible $\beta$--extensions, i.e.\,a family of~$\beta$-extensions parametrised by the points of a chamber of the Bruhat--Tits building of the centralizer~$G_\beta$ which are related via transfer. 
\end{abstract}

\setcounter{tocdepth}{1}
\tableofcontents

\section{Introduction}

Let $F$ be a non-archimedean local field and let $G = \GG(F)$ be the $F$-rational point of a connected reductive $F$-group $\GG$.
The category $\mathcal{R}(G)$ of irreducible smooth (complex) representations of $G$ has been intensively studied over the last fifty years, motivated in large part by the local Langlands programme.
As a first step, Bernstein~\cite{bernstein} showed that $\mathcal{R}(G)$ can be decomposed as a product of indecomposable full subcategories:
\[
	\mathcal{R}(G) = \prod\limits_{\mathfrak{s}} \mathcal{R}^{\mathfrak{s}}(G),
\] 
indexed by the equivalence classes $\mathfrak{s}=[M, \tau]_G$, the \emph{Bernstein components}, with $M$ a Levi subgroup of $G$ and  $\tau$ an irreducible cuspidal representation of $M$.
The subcategory $\mathcal{R}^{\mathfrak{s}}(G)$ of $\mathcal{R}(G)$ is often referred to as the \emph{Bernstein block} attached to $\mathfrak{s}$.
A representation $\pi$ of $G$ belongs to $\mathcal{R}^{\mathfrak{s}}(G)$ if and only if every irreducible subquotient $\pi'$ of $\pi$ has inertial support $\mathfrak{s}$, that is, there exists a representative $(M',\tau')$ of $\mathfrak{s}$
and a parabolic subgroup $P'$ of  $G$ with Levi factor  $M'$ such that  $\pi'$ is a subrepresentation of $\operatorname{Ind}_{M'}^{P'} \tau'$.
The Bernstein decomposition reduces the study of $\mathcal{R}(G)$ to following two fundamental problems:
\begin{itemize}
	\item exhaustive construction of all cuspidal representations of the Levi subgroups of $G$ (including $G$ itself);
	\item  explicit description of the Bernstein blocks. 
\end{itemize}

The theory of \emph{types}, initiated by Bushnell and Kutzko, provides a unified framework to address these problems for general linear groups $\GL_N$~\cite{BK-book,BK-types}.
First, starting from explicit arithmetic data, namely a simple stratum and a simple character, they construct a pair $(\JJ,\lambda)$ consisting of a compact open subgroup $\JJ$ of $G$ and an irreducible representation $\lambda$ of $\JJ$, called a \emph{simple type}.
Among these, the \emph{maximal simple types} (or \emph{cuspidal types}) yield irreducible cuspidal representations of $\GL_N$ via compact induction, and every irreducible cuspidal representation arises in this way.
Second, they associate to each Bernstein component $\mathfrak{s}=[M,\tau]_G$ a \emph{semisimple type} (or $\mathfrak{s}$-\emph{type}) $(\JJ,\lambda)$, where  $\JJ$ is an open compact subgroup of $G$ and $\lambda$ is an irreducible representation of $\JJ$, characterized by the property that irreducible representation  $\pi$ of $G$ is an object of  $\mathcal{R}^{\mathfrak{s}}(G)$ if and only if $\Hom_{\JJ}(\lambda, \pi|_{\JJ}) \neq \{0\}$. 

This programme has since been extended to special linear groups $\operatorname{SL}_N$~\cite{BK-sl,goldberg-type,goldberg-hecke}, to inner forms of $\GL_N$~\cite{repIV,repVI}, and, when $p\neq 2$, to classical groups (symplectic, special orthogonal, or unitary groups over $F$)~\cite{stevens-super,miyauchi}.

For the remainder of the introduction, we assume $p\neq 2$. 
We fix a skewfield $D$ of index $2$ over $F$ together with an anti-involution $({\bar{~}})$ on $D$, and an $\epsilon$-hermitian form
\[
	h \colon V \times V \longto D,
\]
where $V$ is a right $D$-vector space and $\epsilon\in \{-1,1\}$. 
Let $G$ be the group of isometries of $h$, called a quaternionic form of classical groups.
The first author has previously resolved the first problem by constructing and classifying the cuspidal representations of $G$~\cite{daniel-3}.
The purpose of the present paper is to address the second problem:  the description of Bernstein blocks of $G$ by constructing semisimple types, following the blueprint of~\cite{miyauchi} for classical groups.

Our first main result of the present article is the following.

\begin{thm}\label{T:main}
Let ${M}$ be a Levi subgroup of ${G}$, let $\tau$ be a cuspidal irreducible representation of ${M}$, and put $\mathfrak{s}=[\mathrm{M}, \tau]_{G}$. 
There is an $\mathfrak{s}$-type $(\mathrm{J}, \lambda)$ which is, moreover, a cover of the $\mathfrak{s}_{M}$-type $(\mathrm{J} \cap {M}, \lambda |_{\mathrm{J} \cap {M}})$.
\end{thm}
We sketch the construction of the type.
The first step is the construction of a self-dual semisimple stratum $\Delta = [\Lambda, n, 0,\beta]$, a quadruple encoded with following data:
\begin{itemize}
	\item An element $\beta$ of the Lie algebra of $G$ which generates over $F$ a product $E$ of fields in  $A \coloneqq \End_D(V)$. We denote by $G_\beta$ the centralizer of  $\beta$ in  $G$.
	\item A self-dual $\mathfrak{o}_E$-$\mathfrak{o}_D$-lattice sequence $\Lambda$ which can be interpreted as a point in the Bruhat-Tits building $\mathscr B(G)$ of $G$, and as a point as the image of a point $\Lambda_\beta$ of the Bruhat-Tits building $\mathscr B(G_\beta)$ of $G_\beta$, under a canonical embedding (see~\cite{daniel-centralizer})
		\[
		j_\beta \colon \mathscr{B}(G_\beta) \hookrightarrow \mathscr{B}(G).
		\]
	\item An integer $n>0$ which is related to the depth of the stratum;
	\item Compact open subgroups $\HH^1(\beta,\Lambda) \subseteq \JJ^1(\beta,\Lambda) \subseteq \JJ(\beta,\Lambda)$, abbreviated by $\HH^1, \JJ^1 = \JJ^1_{\Lambda}$ and $\JJ =\JJ_{\Lambda}$.
	\item A set $\CCC(\Delta)$ of characters of  $\HH^1$, whose elements are called semisimple characters.
\end{itemize}

Fix a Levi subgroup $M$ stabilizing a self-dual decomposition $V = \bigoplus\limits_{j =-m} ^m W_j$ of $V$, together with a cuspidal representation $\tau$ of $M$.
We then have an isomorphism $M \cong (\Aut_D(W_0)\cap G) \times \prod\limits_{j=1}^m \Aut_D(W_j)$ and  a decomposition $\tau = \tau_0 \boxtimes \bigboxtimes\limits_{j=1}^m \tau_j$.
By the existence of cuspidal types, \cite{daniel-3} and~\cite{repIV},
we obtain a skew semisimple stratum $\Delta_0$ and a self-dual semisimple character $\theta_0$, and for each  $1\leq j \leq m$, a simple stratum  $\Delta_j$ and a simple character $\theta_j$.
In Section~\ref{S:types}, we glue these data into a self-dual seimsimple stratum $\Delta=[\Lambda, n, 0, \beta]$ and a self-dual seimsimple character $\theta$, see Theorem~\ref{T:sdssc}.

We then choose a $\beta$-extension $\kappa$ of $\theta$ to the group  $\JJ(\beta,\Lambda)$.
Chose an arbitrary parabolic subgroup $P$ of $G$ with Levi factor $M$, we form a representation
$\kappa_P$ on $\JJ_P$ canonically obtained from $(\JJ,\kappa)$.
We choose an irreducible representation $\rho_M$ of $\JJ_P$ inflated from a representation of the finite quotient  $\JJ_P/\JJ^1_P \cong \PP(\Lambda_\beta)/ \PP_1(\Lambda_\beta)$, and form
 \[
\lambda_P \coloneqq \kappa_P \otimes \rho_M.
\] 
We prove that $(\JJ_P,\lambda_P)$ is a $[M,\tau]_G$-type by establishing the following two facts:
\begin{itemize}
  \item $(\JJ_P\cap M,\lambda_P|_{\JJ_P\cap M})$ is a $[M,\tau]_M$-type;
  \item $(\JJ_P,\lambda_P)$ is a \emph{cover} of $(\JJ_P\cap M,\lambda_P|_{\JJ_P\cap M})$.
\end{itemize}

Our construction does not attach a type merely to a Bernstein component $[M,\tau]_G$.
Indeed, we impose an ordering on the blocks $\Aut_D(W_j)$ so that endo-equivalent blocks occur in consecutive packets:
more precisely, if $\tilde\theta_{j}$ and $\tilde\theta_{j+r}$ are endo-equivalent, then
$\tilde\theta_j,\tilde\theta_{j+1},\dots,\tilde\theta_{j+r}$ are pairwise endo-equivalent.
This ordering determines a self-dual flag of $V$, hence a parabolic subgroup $P$ with Levi factor $M$.
We then merge the pairwise endo-equivalent $M$-blocks into a single block of the semisimple stratum $\Delta$.
Let $L$ be the Levi subgroup stabilizing the resulting self-dual decomposition of $V$
(equivalently, the $L$-block decomposition attached to $\Delta$).
The parabolic $P$ induces a parabolic subgroup $P\cap L$ of $L$.
Consequently, the actual input of the construction is the triple $(M,\tau,P\cap L)$.

This explains why inertial conjugacy of $[M,\tau]_G$ is too coarse for comparing two types produced by our procedure:
when several $M$-blocks are merged into a single $L$-block, inertial conjugacy allows permutations among those $M$-blocks,
and such permutations generally change the parabolic used in the construction of the type.
We therefore introduce the following refinement.

We say that two triples $(M,\tau,P\cap L)$ and $(M',\tau',P'\cap L')$ are \emph{endo-inertially equivalent}
if there exists $y\in G$ such that
\begin{itemize}
  \item[(i)] $M={\leftidx{^{y}}{M'}{}}$;
  \item[(ii)] $\tau\cong \leftidx{^{y}}{\tau'}{} \otimes \chi$, where $\chi$ is an unramified character of $M$;
  \item[(iii)] $P \cap L={\leftidx{^{y}}{(P'\cap L')}{}}$.
\end{itemize}
Conditions (i) and (ii) express the usual inertial conjugacy.
Condition (iii) records the additional parabolic datum implicit in our construction: it allows permutations of the simple $L$-blocks but forbids permutations of the $M$-blocks
within a single simple $L$-block.
For example, when $L=G$, condition (iii) reduces to $P={\leftidx{^{y}}{P'}{}}$.
Without imposing (iii), one may start from the same Bernstein component $[M,\tau]_G$ but choose different orderings among endo-equivalent blocks, hence different parabolics $P$ and $P'$, possibly even opposite.
Since the representation $\kappa_P$ (and hence $\lambda_P$) is defined using $\JJ^1\cap U$, where $U$ is the unipotent radical of $P$, there is no reason for the resulting types to be conjugate.

We say that two types $(\JJ_P,\lambda_P)$ and $(\JJ'_{P'},\lambda'_{P'})$ are \emph{essentially conjugate}
if there exist a bijection $\zeta\colon I\to I'$ and an element $g\in G$ such that
\begin{itemize}
  \item[(i)] $gV^{\zeta(i)} \cong V^i$ for all $i\in I$;
  \item[(ii)] $\lambda_L \cong \leftidx{^g}{\!\lambda'_{L'}}{}$.
\end{itemize}
In Theorem~\ref{T:ess-conj}, we prove that, when $D=F$, endo-inertial equivalence implies essential conjugacy.

The second part of this article constructs, for a fixed $\beta$, a compatible family of $\beta$-extensions.
Fix a chamber $\mathcal C$ of $\mathscr B(G_\beta)$.
For each $x\in \overline{\mathcal{C}}$, we denote by $\Lambda_x$ the self-dual $\mathfrak{o}_E$-$\mathfrak{o}_D$-lattice sequence corresponding to $j_\beta(x)$; we will sometimes say $\Lambda_x$ corresponds to $x$.
We fix a system $\{\theta_x \mid x\in \overline{\mathcal{C}}\}$ of semisimple characters such that the  $\theta_x$ are related by transfer.
We denote by  $\eta_x$  the Heisenberg extension of $\theta_x$.
In~\cite[\S4]{stevens-super} and~\cite[\S6]{daniel-3}, to any self-dual $\mathfrak{o}_E$-$\mathfrak{o}_D$-lattice sequence $\Lambda_{\mathsf{M}}$ corresponding to a vertex $x_{\mathsf{M}}$ of $\overline{\mathcal C}$ such that $\tilde{\mathfrak{b}}(\Lambda_x)\subseteq \tilde{\mathfrak{b}}(\Lambda_{\mathsf{M}})$, we associate a set
$\beta\text{-}\ext_{\Lambda_{\mathsf{M}}}(\Lambda_x)$ consisting of the isomorphism class of certain extensions of the Heisenberg representation $\eta_{\Lambda_x}$.
The elements of $\beta\text{-}\ext_{\Lambda_{\mathsf{M}}}(\Lambda_x)$ are called \emph{$\beta$-extensions} on  $\JJ_{\Lambda_x}$ relative to $\Lambda_{\mathsf{M}}$.
\emph{A priori} this set depends on the choice of the auxiliary vertex $x_{\mathsf{M}}$, and it is natural to ask to what extent this dependence can be removed.

In the present article, we remove this dependence after passing to the subgroup $\JJ^\circ_{\Lambda_x}$ of $\JJ_{\Lambda_x}$.
More precisely, the finite quotient $\JJ_{\Lambda_x}/ \JJ^1_{\Lambda_x} = \mathds{P}(\Lambda_{x,\beta})(k_F)$ is the $k_F$-points of a $k_F$-reductive group $\mathds{P}(\Lambda_{x,\beta})$, where $k_F$ is the residue field of $F$.
We write $\mathds{P}(\Lambda_{x,\beta})^\circ$ for the connected component of  $\mathds{P}(\Lambda_{x,\beta})$ and denote by $\JJ^\circ_{\Lambda_x}$ the pre-image of $\mathds{P}(\Lambda_{x,\beta})^\circ(k_F)$ in $\JJ_{\Lambda_x}$.
Likewise, we have that $\PP(\Lambda_x)/\PP_1(\Lambda_x) = \mathds{P}(\Lambda_x)(k_F)$ for a $k_F$-reductive group $\mathds{P}(\Lambda_x)$.
We write  $\mathds{P}(\Lambda_x)^\circ$ for the connected component of $\mathds{P}(\Lambda_x)$ and denote by $\PP^\circ(\Lambda_x)$ the pre-image of $\mathds{P}(\Lambda_x)^\circ(k_F)$ in $\PP(\Lambda_x)$. 

As above, let $\Lambda_x$ (resp.\,$\Lambda_{\mathsf{M}}$) correspond to a point (resp.\,a vertex) of $\overline{\mathcal{C}}$ with $\tilde{\mathfrak{b}}(\Lambda_x) \subseteq \tilde{\mathfrak{b}}(\Lambda_{\mathsf{M}})$  (for convenience, we sometimes say that $\Lambda_{\mathsf{M}}$ is a \emph{vertex above $x$}).
We consider the set
$$
\beta\text{-}\ext^\circ_{\Lambda_{\mathsf{M}}}(\Lambda_x) \coloneqq \Res^{\JJ_{\Lambda_x}}_{\JJ^\circ_{\Lambda_x}} \left(\beta\text{-}\ext_{\Lambda_{\mathsf{M}}}(\Lambda_x) \right)
$$
of $\beta$-extensions on $\JJ^\circ_{\Lambda_x}$ relative to $\Lambda_{\textsf{M}}$.
Our aim is to single out, for each $x\in \overline{\mathcal C}$, a distinguished representation
$\hat\kappa_x^\circ$ of  $\JJ^\circ_{\Lambda_x}$: although its construction involves the choice of $\Lambda_{\mathsf{M}}$, the resulting isomorphism class is independent of that choice.
Its construction proceeds in two steps, see Construction~\ref{C:chi}.
First, we take the extension $\kappa_x$ of $\eta_x$ to $\JJ_{\Lambda_x}$ whose determinant has $p$-power order.
Next, we introduce a quadratic character
$$
\chi_x^\circ=\chi_{x,1}^\circ \cdot \chi_{x,2}^\circ \cdot \chi_{x,3}^\circ
$$
of $\JJ^\circ_{\Lambda_x}$.
Here, the factor $\chi_{x,1}^\circ$ is the signature character of the conjugation action on the finite set $\JJ^1_{\Lambda_x}\backslash \PP^1(\Lambda_x)$.
The factors $\chi_{x,2}^\circ$ and $\chi_{x,3}^\circ$ are quadratic characters on $\JJ^\circ_{\Lambda_x}/ \JJ^1_{\Lambda_x}$ and $\PP^\circ(\Lambda_x)/ \PP_1(\Lambda_x)$, respectively,
each characterized by the signature character of the conjugation action on the unipotent radical of one (equivalently, any) Borel subgroup, see Proposition~\ref{P:sign-extension}.
We then set
\[
\hat\kappa_x^\circ
\coloneqq \left(\Res^{\JJ_{\Lambda_x}}_{\JJ^\circ_{\Lambda_x}}\kappa_x \right) \otimes \chi_x^\circ.
\]

Before we state compatibility, let us introduce one further notation.
When $\Lambda_{\sM}$ is a vertex above $x$, Lemma~\ref{L:Psi-circ} provides a connected transfer map
\[ 
\Psi^\circ_{\Lambda_x,\Lambda_{\mathsf{M}},\Lambda_x} \colon \Ext^\circ(\Lambda_x, \Lambda_{\mathsf{M}}) \longto \Ext^\circ(\Lambda_x), 
\] 
which is compatible with the (non-connected) transfer map of~\cite[Theorem 6.2]{KS2020}, \cite[\S6.4]{daniel-3} and~\cite[\S4.2]{stevens-super}.
Here, $\Ext^\circ(\Lambda_x, \Lambda_{\mathsf{M}})$ (resp.\,$\Ext^\circ(\Lambda_x)$) denotes the set of isomorphism classes of extensions of $\eta_{\Lambda_x,\Lambda_{\mathsf{M}}}$ to $\JJ^\circ_{\Lambda_x,\Lambda_{\mathsf{M}}}$ (resp.\,isomorphism classes of extensions of $\eta_{\Lambda_x}$ to $\JJ^\circ_{\Lambda_x}$).

The second main result of the present article is the following: 
\begin{thm}\label{T:main2}
[{Theorem~\ref{T:compatible-beta-extension}}]
	The family $\{ \hat{\kappa}^\circ_x \mid x\in\overline{\mathcal{C}}\}$ is a \emph{compatible family of $\beta$-extensions} in the following sense (see Definition~\ref{D:compatible-family}). 
\begin{itemize}
	\item[(i)] For any $x\in \overline{\mathcal{C}}$, there exists a vertex $\Lambda_{\sM}$ above $x$ such that $\hat{\kappa}^\circ_x\in \beta$-$\ext^\circ_{\Lambda_{\mathsf{M}}}(\Lambda_x)$.
\item[(ii)] 
For any vertices $\Lambda_1, \Lambda_2$ above $x$, we have
\[
	\hat{\kappa}^\circ_x \cong 
	\Psi^\circ_{\Lambda_x,\Lambda_1, \Lambda_x} \left(\Res^{\JJ^\circ_{\Lambda_1}}_{\JJ^\circ_{\Lambda_x, \Lambda_1}}\hat{\kappa}^\circ_1 \right) \cong \Psi^\circ_{\Lambda_x,\Lambda_2, \Lambda_x} \left(\Res^{\JJ^\circ_{\Lambda_2}}_{\JJ^\circ_{\Lambda_x, \Lambda_2}}\hat{\kappa}^\circ_2 \right).
\] 
\end{itemize}

\end{thm}
Assertion (i) states that each $\hat{\kappa}^\circ_x$ is indeed a  $\beta$-extension, while Assertion (ii) states that these representations do not depend on the auxiliary vertices used to define them.
Both follow readily from the more general Proposition~\ref{P:compatible-beta-extension} which 
compares the family $\{ \hat{\kappa}^\circ_x \mid x\in\overline{\mathcal{C}}\}$ under all of the connected transfer maps given by Lemma~\ref{L:Psi-circ}.
The proof of Proposition~\ref{P:compatible-beta-extension} reduces to comparing the discrepancy quadratic character $\nu^\circ_{x,y}$, induced by the connected transfer map $\Psi^\circ_{\Lambda,\Lambda_x,\Lambda_y}$, with the twisting characters $\chi^\circ_{x}$ and $\chi^\circ_{y}$, for all self-dual $\mathfrak{o}_E$-$\mathfrak{o}_D$-lattice sequences $\Lambda, \Lambda_x$ and $\Lambda_y$ satisfying 
$\tilde{\mathfrak{b}}(\Lambda) \subseteq \tilde{\mathfrak{b}}(\Lambda_x) \cap \tilde{\mathfrak{b}}(\Lambda_y)$. 
The comparison is in turn reduced to finite reductive groups of Lie type.
Since such a group is generated by the conjugates of a Borel subgroup, we may carry out the comparison on a Borel subgroup. 
There, $\chi^\circ_{x,2}$ and  $\chi^\circ_{x,3}$ are characterized by the signature character of the conjugation action on its unipotent radical.
This technique is not available for disconnected reductive groups of Lie type, and we only obtain the compatibility on the groups $\JJ^\circ_{\Lambda_x}$.

For a general linear group $\tilde{G}$, we construct a compatible family of $\beta$-extensions in Theorem~\ref{T:compatible-beta-extension-gl} by a more elementary approach.

This compatibility result has many applications. One important application is the reduction to depth zero to the centralizer for the study of certain factors of the category of smooth representations of~$G$. 

\subsection*{Structure of the paper}
In Section~\ref{S:preliminaries}, we collect some necessary notions and recall some basic facts on semisimple types.
In Section~\ref{S:types}, we construct semisimple types attached to Bernstein components, proving Theorem~\ref{T:main}.
In Section~\ref{S:conjugacy}, we prove, for $D=F$, an \emph{endo-inertial equivalence} implies \emph{essential conjugacy} result.
In Section~\ref{S:compatibility}, we establish the existence of a compatible family of $\beta$-extensions, both for  $G$ (Theorem~\ref{T:compatible-beta-extension}, stated above as Theorem~\ref{T:main2}) and for general linear groups (Theorem~\ref{T:compatible-beta-extension-gl}).

\section{Preliminaries}\label{S:preliminaries}

\subsection{First notations}
We use the notation in~\cite[\S 2]{daniel-3}.
Let $H$ be a locally pro-finite group.
Let $K$ be a subgroup of $H$ and $\rho$ a representation of $K$.
For any $h\in H$, we define the representation $\leftidx{^h}{\rho}{}$ of the group $\leftidx{^h}{K}{} \coloneqq h K h^{-1}$  by $\leftidx{^h}{\rho}{}(hkh^{-1}) \coloneqq \rho(k)$ for all $k\in K$.
We write $\mathcal{R}(H)$ for the category of smooth representations of $H$ with complex coefficients.
\shu{Let $K_1$ and $K_2$ be subgroups of $H$ and let $(V_i,\rho_i)$ be smooth
representations of $K_i$, for $i =1,2$. 
For $h\in H$, we set
\[
	\operatorname{I}_h(\rho_1,\rho_2) \coloneqq \Hom_{K_1 \cap \leftidx{^{h}}{K_2}{}}
	\left( \rho_1,\leftidx{^h}{\rho_2}{} \right),
\]
and call
\[
  \operatorname{I}_H(\rho_1,\rho_2) \coloneqq \{ h\in H \mid I_h(\rho_1,\rho_2)\neq 0 \}.
\]
the intertwining of $\rho_1$ with  $\rho_2$ in  $H$.
}

Given two smooth representations $(V_1,\rho_1)$ and  $(V_2,\rho_2)$ of two locally profinite groups $H_1$ and $H_2$, respectively, we denote by  $\rho_1 \boxtimes \rho_2$ the smooth representation of $H_1\times H_2$ on $V_1 \otimes V_2$ defined via  $(\rho_1 \boxtimes \rho_2)(h_1,h_2) (v_1\otimes v_2) = \rho_1(h_1)(v_1) \otimes \rho_2(h_2)(v_2)$.


Let $F$ be a non-archimedean local field of odd residue characteristic.
We fix a skewfield $D$ of index $2$ over $F$, together with an anti-involution $({\bar{~}})$ on $D$, and an $\epsilon$-hermitian form
\[
	h \colon V \times V \longto D,
\]
where $V$ is a right $D$-vector space and $\epsilon\in \{-1,1\}$.
The valuation on $F$ extends uniquely to a non-archimedean valuation $w_D$ on $D$.
We fix a uniformizer $\varpi_D$.
We denote by  $\mathfrak{o}_D$ the valuation ring, and by $\mathfrak{p}_D$ its maximal ideal.

We define $\tilde G=\Aut_D(V)$. We write $\sigma_{h}$ for the adjoint anti-involution induced by $h$ and $\sigma$ for the involution on $\tilde G$ given by $g \mapsto \sigma_{h}(g)^{-1}$.
We put $$G= U(V,h)=\{ g\in \tilde G \mid h(gv,gw)=h(v,w), \forall v,w \in V \}.$$
The group $\Sigma= \langle 1, \sigma \rangle$ acts on $\tilde G$ and we have $G=\tilde G ^{\Sigma}$.
Note that $G=\GG(F)$ for some connected reductive group $\GG$ defined over $F$, see~\cite[Proposition 2.9]{daniel-2} for the connectedness of~$\GG$.
We write $A$ (resp.\  $A_{-}$) for the Lie algebra of  $\tilde G$ (resp.\  $G$).

For any $x\in \QQ$, we denote by  $\lfloor x \rfloor$ the largest integer which is smaller or equal to $x$, and by $\lceil x \rceil$ the smallest integer which is greater or equal to $x$.

An $\mathfrak{o}_D$-\emph{lattice sequence} on a right $D$-vector space $V$ is a function $\Lambda$ from $\ZZ$ to the set of $\mathfrak{o}_D$-lattices in $V$,
which is decreasing with respect to\ inclusions, such that there exists a unique positive integer  $e=e(\Lambda| D)$ satisfying  $\Lambda(z+e)=\Lambda(z)\mathfrak{p}_D$ for all $z\in \ZZ$. 
We call $e$ the  $\mathfrak{o}_D$-\emph{period} of $\Lambda$.
We say that a lattice sequence has a \emph{jump} at an integer~$z$ if~$\Lambda(z)\supsetneq\Lambda(z+1)$.
We put $$\tilde{\mathfrak{a}}_s = \tilde{\mathfrak{a}}_s(\Lambda) = \{a\in A \mid a\Lambda(z) \subseteq \Lambda(z+s),~\forall z\in \ZZ \}$$ for all $s\in \ZZ$.

An \emph{affine translation} of $\Lambda$ is a lattice sequence $a\Lambda+b$ on $V$ defined by $$(a\Lambda+b)(z) \coloneqq \Lambda(\lfloor(z-b) / a\rfloor)$$ for some $a,b\in  \ZZ$ with  $a > 0$. 
Two lattice sequences are said to be in the same affine class if they have a common affine translation, which gives an equivalence relation on the set of lattice sequences.
We have  $e(a\Lambda+b)=a e(\Lambda)$ and  $\tilde{\mathfrak{a}}_z(a\Lambda+b)= \mathfrak{a}_{\lfloor z/a \rfloor}(\Lambda)$ for all $z\in \ZZ$.


Let $\Lambda$ be an $\mathfrak{o}_D$-lattice sequence, we define the \emph{dual sequence} $\Lambda^\#$ of $\Lambda$ via
\[
\Lambda^\#(z) \coloneqq  \left(\Lambda(1-z) \right)^\#,  \quad\quad\quad   \forall z\in \ZZ.
\] 
We say $\Lambda$ is \emph{self-dual} if there exists $d\in \ZZ$ such that $\Lambda^\#=\Lambda+d$.
\daniel{\begin{lem}\label{lemSelfdualAffineTranslate}
Let~$\Lambda$ be a lattice sequence, $a\in\mbN$ and $b\in\mbZ$. Then 
\begin{enumerate}
\item $(a\Lambda + b)^\#=a(\Lambda^\#)-(2(a-1)+b)$
\item Suppose~$\Lambda$ has a jump at~$z$. Then 
$(a\Lambda + b)$ has a jump at~$az+a-1+b$. 
\end{enumerate}
\end{lem}
\begin{proof}
The first assertion follows from the two cases $a=1$ and $b=1-a$.
For $a=1$ we obtain
\[
(\Lambda+b)^\#(z) = (\Lambda+b)(1-z) = \Lambda(1-z-b) 
= \Lambda^\#(z+b) = (\Lambda^\#-b)(z)
\]
and for~$b=1-a$
\[
(a\Lambda+1-a)^\#(z) = \Lambda(\lfloor \frac{a-z}{a}\rfloor) = \Lambda^\#(-\lceil \frac{-z}{a}\rceil) 
= (a\Lambda^\#+1-a)(z).
\]
Here we used~$-\lfloor x\rfloor=\lceil -x\rceil$ and~$\lceil x\rceil =\lfloor 1+x-\frac{1}{a} \rfloor$ for~$x\in\frac{1}{a}\mbZ$. 
The second assertion is left to the reader. 
\end{proof}
}

Given a decomposition $V= \bigoplus\limits_{j\in S} W_j$, we denote by $\mathbf{e}_j$ the idempotent of $\End_D(V)$ with kernel  $\bigoplus\limits_{i\neq j} W_j$ and image $W_j$.
The decomposition is said to be a \emph{splitting} of $\Delta$ if
\begin{itemize}
	\item[(i)]  it splits $\Lambda$, i.e., $\Lambda(z)= \bigoplus\limits_{j\in S} \Lambda(z)\cap W_j$;
	\item[(ii)] $\beta= \sum\limits_{j\in S} \beta_j$ where  $\beta_j= \mathbf{e}_j\beta \mathbf{e}_j$ for all $j\in S$.
\end{itemize}

In this article, we shall consider a decomposition $V= \bigoplus\limits_{j\in S} W_j$ given by a Levi subgroup, and for each $j\in S$, we are given a lattice sequence $\Lambda_j$ with $\Lambda_{-j}= \Lambda_j^\#$.
We shall construct the direct sum of the $\Lambda_j$ in Section~\ref{S:types}.
The lemma says that we can (and we will) modify the $\Lambda_j$ within their respective affine classes in a way adapted to our purpose, see Construction~\ref{C:central}.
Note that the Lie algebra element $\beta$ does not play a role in the lemma.

\daniel{
\begin{lem}\label{L:affine}
Suppose $\{\pm 1,\cdots, \pm m\}  \subseteq  S  \subseteq \{0, \pm 1, \cdots \pm m\}$ for some positive integer $m$.
Let $V= \bigoplus\limits_{j\in S} W_j$ be a decomposition and $e$ be a positive integer. 
For all $j\in S$, suppose that $\Lambda_j$ is an $\mathfrak{o}_D$-lattice sequence of period~$e$ on $W_j$ satisfying $\Lambda_{-j}= \Lambda_j^\#$,
then there exist a self-dual lattice sequences $\Lambda'_j$ on $W_j$ of period~$(2m+1)e$ subject to the following conditions: 
\begin{enumerate}
	\item $\Lambda'_j$ is in the same affine class as~$\Lambda_j$;
	\item $\Lambda'\coloneqq \bigoplus\limits_{j\in S} \Lambda'_j$ is self-dual;		
\item $\mathbf{e}_j+\tilde{\mathfrak{a}}_1(\Lambda')$  is a central idempotent in
	$\tilde{\mathfrak{a}}_0(\Lambda')/\tilde{\mathfrak{a}}_1(\Lambda')$;
\item $-j$ is a jump of $\Lambda'_j$.
\end{enumerate}		
\end{lem}
}
\begin{proof}
\daniel{Without loss of generality we can assume that for all non-zero~$j\in S$ the lattice sequence ~$\Lambda_j$ has a jump at~$0$. We put~$c=(2m+1)$ and define $\Lambda'_j=c\Lambda_j+1-c-j$. Now (1) is obvious and (2) and (4) follow from Lemma~\ref{lemSelfdualAffineTranslate}. 
Put $\Lambda' \coloneqq \bigoplus\limits_{j\in S} \Lambda'_j$ and  $$(\tilde{\mathfrak{a}}_k(\Lambda'))^{i,j}= \{ a\in \tilde{\mathfrak{a}}_k(\Lambda') \mid a {\Lambda'_j(z)} \subseteq {\Lambda'}_i(z), \forall z\in \ZZ\}$$
for any $i,j \in S$ and~$k\in\mbZ$: 
We next claim that each $\mathbf{e}_j$ is central in $\tilde{\mathfrak{a}}_0(\Lambda')/\tilde{\mathfrak{a}}_1(\Lambda')$.
It suffices to show that $(\tilde{\mathfrak{a}}_0(\Lambda'))^{i,j}=(\tilde{\mathfrak{a}}_1(\Lambda'))^{i,j}$ for all $i \neq j$.
Assume the contrary is true, then there exits $i \neq j$ in~$S$ and an element $a \in (\tilde{\mathfrak{a}}_0(\Lambda'))^{i,j} \setminus (\tilde{\mathfrak{a}}_1(\Lambda'))^{i,j}$.
We then have $a (\Lambda'_j(z)) \subseteq \Lambda'_i(z)$ for all  $z\in \ZZ$, and $a (\Lambda'_j(z_0)) \not\subseteq \Lambda'_i(z_0+1)$ for some  $z_0\in \ZZ$.
It follows that $\Lambda_i \left(\left\lfloor \frac{z_0+c+i}{c} \right\rfloor \right) \subsetneqq  \Lambda_i \left(\left\lfloor \frac{z_0+c-1+i}{c} \right\rfloor \right)$ and  $\Lambda_j \left(\left\lfloor \frac{z_0+c+j}{c} \right\rfloor \right) \subsetneqq  \Lambda_j \left(\left\lfloor \frac{z_0+c-1+j}{c} \right\rfloor \right)$.
We thus have that $c$ divides $z_0+c+i$ and $z_0+c+j$, and therefore it divides $i-j$, a contradiction.
}
\end{proof}

\begin{exmp}

	Let $V=W_{-1} \oplus W_{1}$ where  $W_{-1}=W_{1}= D \oplus D$, equipped with an Hermitian form $h$ given by $h \big((x_{-2},x_{-1},x_1,x_2), (y_{-2},y_{-1},y_1,y_2) \big)\coloneqq  x_{-2} y_2+ x_{-1} y_{1}-x_1 y_{-1}- x_{2} y_{-2}$.
	Let $\Lambda_{-1}$ and  $\Lambda_{1}$ be lattice sequences on  $W_{-1}$ and  $W_1$, respectively, defined by
	 \begin{equation*}
		 \left\{ \begin{array}{ll}
				 \Lambda_{-1}(2i)= \mathfrak{p}_D^{i} \oplus \mathfrak{p}_D^{i} \\
				 \Lambda_{-1}(2i+1)= \mathfrak{p}_D^{i} \oplus \mathfrak{p}_D^{i+1}
			 \end{array}\right.
			 \quad \text{and}\quad
			  \left\{ \begin{array}{ll}
				 \Lambda_{1}(2i)= \mathfrak{p}_D^{i} \oplus \mathfrak{p}_D^{i+1} \\
				 \Lambda_{1}(2i+1)= \mathfrak{p}_D^{i+1} \oplus \mathfrak{p}_D^{i+1},
			 \end{array}\right.
	\end{equation*}
respectively, for all $i\in \ZZ$.
Then $\Lambda_{-1}=(\Lambda_1)^\#$ and $\Lambda_1=(\Lambda_{-1})^\#$.
As in the proof of Lemma~\ref{L:affine}, we define ${\Lambda'}_{-1}=3 \Lambda_{-1}-1$ and ${\Lambda'}_1= 3 \Lambda_1-3$.
Then ${\Lambda'}_{-1}=({\Lambda'}_1)^\#$ and ${\Lambda'}_1=({\Lambda'}_{-1})^\#$.
Moreover, $\mathbf{e}_{-1}$ and $\mathbf{e}_1$ are central in
	${\tilde{\mathfrak{a}}_0 \left({\Lambda'}_{-1} \oplus {\Lambda'}_{1} \right)} \Big/{\tilde{\mathfrak{a}}_1 \left({\Lambda'}_{-1} \oplus {\Lambda'}_{1} \right)}$, as desired.
\end{exmp}

\subsection{Semisimple characters}

Let $\Delta=[\Lambda,n,r,\beta]$ be a semisimple stratum.
For any subgroup $H$ of $\tilde{G}$, we write $H_\beta$ for the centralizer of $\beta$ in $H$.
We write $B = B_\beta$ for the centralizer of $\beta$ in  $A$, i.e., $B=\End_{E\otimes_F D} (V)$ where $E \coloneqq F[\beta]$. For any $z\in \ZZ$, we put 
$$\tilde{\mathfrak{b}}_{z}=\tilde{\mathfrak{b}}_{z}(\beta,\Lambda)= \tilde{\mathfrak{a}}_z(\Lambda) \cap B.$$ 
The stratum~$\Delta$ comes equipped with an associated splitting
$$V=\bigoplus_{i\in I}V^i.$$ 
There are subgroups $\tilde{\HH}(\Delta),\tilde{\HH}^i(\beta,\Lambda)$ and  $\tilde{\JJ}^i(\beta,\Lambda)$ of $\tilde{G}$, which are defined and denoted as $\HH(\beta,\Lambda),{\HH}^i(\beta,\Lambda)$ and  $\JJ^i(\beta,\Lambda)$, respectively, in~\cite[\S 5.1.2]{daniel-1}. 
Further the element~$\beta$ generates a product of field extensions of~$F$ $$E=F[\beta]=\bigoplus_{i\in I}E_i,\ E_i=F[\beta_i].$$
We denote by $\tilde{\operatorname{C}}(\Delta)$ the set of \emph{semisimple characters} of $\Delta$, which are certain (complex) characters of  $\tilde{\HH}(\Delta)$.
We need more notations for later.
We put
$\tilde{\mathfrak{m}}_{z}(\beta, \Lambda)=\tilde{\mathfrak{n}}_{z+k_0}(\beta, \Lambda) \cap \tilde{\mathfrak{a}}_{z}(\Lambda)$, 
where $\tilde{\mathfrak{n}}_{z}(\beta, \Lambda)\coloneqq  \left\{a \in \tilde{\mathfrak{a}}_0 \mid \beta a-a \beta \in \tilde{\mathfrak{a}}_{z}\right\}$,~$z\in\ZZ$.
Further
 $$\tilde{\mathfrak{m}}(\Delta) = \tilde{\mathfrak{m}}_{-(r+k_0)}(\beta,\Lambda),$$ where $k_0$ denotes the \emph{critical exponent} of $\Delta$ as defined in~\cite[\S 4.1]{daniel-1},
and we denote
$$
\operatorname{S}(\Delta) =\operatorname{S}(\beta,\Lambda)= (1+\tilde{\mathfrak{m}}(\Delta)) \tilde{\JJ}^{\lfloor \frac{-k_0+1}{2} \rfloor}(\beta,\Lambda)
=1+ \tilde{\mathfrak{m}}(\Delta)+\tilde{\mathfrak{j}}^{\lfloor \frac{-k_0+1}{2} \rfloor}(\beta,\Lambda).
$$

Suppose that $\Delta$ is self-dual, in particular we have an action of~$\Sigma$ on~$I$ via the action of~$\sigma_h$ on the set of idempotents~$e_i$,~$i\in I$. We choose a partition~$I=I_+\cup I_-\cup I_0$ with the set~$I_0$ of~$\Sigma$-fixed points and a set~$I_+$ of representatives of orbits of length two, i.e.~$\sigma$ swaps~$I_+$ with~$I_-$. Note, we have~$\sigma_h(\beta_i)=-\beta_{\sigma(i)}$.
As usual, we write $\HH(\Delta)=\HH(\beta,\Lambda)=\tilde{\HH}(\beta,\Lambda)\cap G$, 
$$\HH^i(\beta,\Lambda)=\left\{\begin{array}{ll}\tilde{\HH}^i(\beta,\Lambda)\cap G, & i\in I_0\\
\tilde{\HH}^i(\beta,\Lambda), & i\in I_+\cup I_- \end{array}\right.,$$ and $\JJ(\Delta),\ \JJ(\beta,\Lambda)$ and~$\JJ^i(\beta,\Lambda)$ similarely.
We refer to~\cite[\S 6.1]{daniel-2} for the notion of self-dual semisimple characters.
We denote by $\operatorname{C}(\Delta)$ the restriction of $\tilde{\operatorname{C}}(\Delta)$ on $\HH(\Delta)$, the set of \emph{self-dual semisimple characters} of $\Delta$.
The restriction then gives a bijection between $\tilde{\operatorname{C}}^{\Sigma}(\Delta)$ and $\operatorname{C}(\Delta)$, by the \emph{Glauberman correspondence}.

\begin{lem}\label{L:central-idempotents}
	Let $\Delta$ and $\Delta'$ be semisimple strata with $\Lambda=\Lambda', n=n', r=r'=0$ such that  $\tilde{\CCC}(\Delta)\cap \tilde{\CCC}(\Delta') \neq \emptyset$.
For any idempotent in $\mathbf{e}\in \tilde{\mathfrak{b}}_0(\beta,\Lambda)$,
there exists an idempotent $\mathbf{e}'\in \tilde{\mathfrak{b}}_0(\beta',\Lambda)$ such that $\mathbf{e} \equiv \mathbf{e}' \pmod{S(\Delta)-1}$.
Moreover, if $\mathbf{e}+\tilde{\mathfrak{b}}_{1}(\beta,\Lambda)$ is central in $\tilde{\mathfrak{b}}_{0}(\beta,\Lambda)/ \tilde{\mathfrak{b}}_{1}(\beta,\Lambda)$, then  $\mathbf{e}'+\tilde{\mathfrak{b}}_{1}(\beta',\Lambda)$ is central in $\tilde{\mathfrak{b}}_{0}(\beta' ,\Lambda)/ \tilde{\mathfrak{b}}_{1}(\beta',\Lambda)$.
\end{lem}

\begin{proof}

Choose an arbitrary $\theta\in \tilde{\CCC}(\Delta)\cap \tilde{\CCC}(\Delta')$.
By~\cite[Proposition 5.15]{daniel-1}, we have that
\[
\operatorname{I}(\theta)=	S(\Delta) B_{\beta}^\times S(\Delta)=S(\Delta') B_{\beta'}^\times S(\Delta').
\] 
Denote by $\widehat{\operatorname{I}(\theta)}$ the $p$-adic closure of $\operatorname{I}(\theta)$, we then have
\begin{equation*}
\widehat{\operatorname{I}(\theta)}=	S(\Delta) B_{\beta} S(\Delta)=S(\Delta') B_{\beta'} S(\Delta').
\end{equation*}
Considering the additive group generated by $\widehat{\operatorname{I}(\theta)} \cap \tilde{\mathfrak{a}}_0$, we have 
\begin{equation}\label{E:s(delta)}
	\tilde{\mathfrak{b}}_{0}(\beta,\Lambda) + \tilde{\mathfrak{m}}(\Delta) + \tilde{\mathfrak{j}}^{\lfloor \frac{-k_0+1}{2}	 \rfloor}(\beta,\Lambda)
	= \tilde{\mathfrak{b}}_{0}(\beta', \Lambda) + \tilde{\mathfrak{m}}(\Delta' ) +\tilde{\mathfrak{j}}^{\lfloor \frac{-k_0+1}{2} \rfloor}(\beta',\Lambda)
\end{equation}
\daniel{and \cite[Lemma7.13]{daniel-int} provides the idempotent~$\mathbf{e}'$. }
Adding $\tilde{\mathfrak{a}}_1$ to both sides of~\eqref{E:s(delta)} we obtain $\tilde{\mathfrak{a}}_1+\tilde{\mathfrak{b}}_{0}(\beta,\Lambda)=\tilde{\mathfrak{a}}_1+\tilde{\mathfrak{b}}_{0}(\beta',\Lambda)$,
whence an isomorphism $$\tilde{\mathfrak{b}}_{0}(\beta,\Lambda)/ \tilde{\mathfrak{b}}_{1}(\beta,\Lambda)\cong \tilde{\mathfrak{b}}_{0}(\beta' ,\Lambda)/ \tilde{\mathfrak{b}}_{1}(\beta',\Lambda)$$ of $F$-algebras.
The lemma then follows as in~\cite[Proposition 9.9(iv)]{daniel-int} and~\cite[Corollary 5.17]{daniel-1}..

\end{proof}

\subsection{Centralizer}

We recall some notion from~\cite[\S 2.4]{daniel-3}.
We fix a self-dual semisimple stratum $\Delta=[\Lambda,r,0,\beta]$.
The Bruhat-Tits building $\mathcal{B}(G)$ of $G$ is identified with the space $\Latt^1_{h} V$ of self-dual $\mathfrak{o}_D$-lattice functions in $V$.


Let $$G^i=\left\{\begin{array}{ll}G\cap \Aut_D(V^i), & i\in I_0\\
\Aut_D(V^i),& i\in I_+\cup I_-\end{array}\right..$$ 
Then $\mathcal{B}(G_\beta)\cong \prod\limits_{i\in I_{0,+}} \mathcal{B}(G^i_{\beta})$, where $I_{0,+} \coloneqq I_0 \cup I_+$.
By~\cite[Theorem 7.2]{daniel-centralizer}, there exists a $G_\beta$-equivariant embedding
 \[
	 j_\beta \colon \mathcal{B}(G_\beta) \longhook  \mathcal{B}(G).
\] 
whose image contains $\Gamma_\Lambda$, the associated self-dual $\mathfrak{o}_E$-$\mathfrak{o}_D$-lattice function of $\Lambda$.

For any $i\in I_{0,+}$, there exists a \shu{ skewfield $D_{\beta}^i$ with centre $E_i$} and a right-$D_{\beta}^{i}$-vector space $V_{\beta}^{i}$ such that $\End_{{D}_{\beta}^{i}}\left(V_{\beta}^{i}\right)\cong \End_{\mathrm{E}_{i} \otimes_F {D}}\left(V^{i}\right)$ as  $\mathrm{E}_{i}$-algebras.
The construction of $j_{\beta}$ provides for each $i\in I_{0,+}$ an
$\mathfrak{o}_{{D}_{\beta}^{i}}$-lattice function $\Gamma_\beta^i$ in $V_{\beta}^{i}$ such that
\[
j_\beta \big( (\Gamma_\beta^i)_{i\in I_{0,+}} \big)=\Gamma_\Lambda.
\] 
Let $e$ be the  $\mathfrak{o}_F$-period of $\Lambda$. We define an $\mathfrak{o}_D$-lattice sequence $\Lambda_\beta^i$ via $\Lambda_\beta^i(z) \coloneqq \Gamma_\beta^i(\frac{z}{e})$ and define $\Lambda_\beta =(\Lambda_\beta^i)_{i\in I_{0,+}}$.

\subsection{$\beta$-extensions}

We fix a self-dual semisimple stratum $\Delta=[\Lambda,n,0,\beta]$ and a self-dual semisimple character $\theta\in \operatorname{C}(\Delta)$.

By~\cite[Proposition 4.3]{daniel-3}, there is up to isomorphism a unique irreducible representation $\eta$ of $\JJ^1$ which contains $\theta$,
the \emph{Heisenberg extension} of $\theta$.

Let $\JJ^{\circ}(\beta,\Lambda)= \JJ^1(\beta,\Lambda) \operatorname{P}^{\circ}(\Lambda_\beta)$.
We drop $\Lambda$ and  $\beta$ from the decorations if they are clear in the context.
We choose a self-dual $\mathfrak{o}_E$-$\mathfrak{o}_D$-lattice sequence $\Lambda_{\mM}$ which represents a weak vertex in $\mathcal{B}(G_\beta)$. 
In~\cite[\S 6]{daniel-3}, the author constructed the set $\beta$-$\ext_{\Lambda_{\mM}}(\Lambda)$ of $\beta$-extensions of $\eta$ to  $\JJ$ relative to  $\Lambda_{\mM}$, which consists of certain irreducible representations of  $\JJ$ extending $\eta$.
Restricting $\beta$-$\ext_{\Lambda_{\mM}}(\Lambda)$ to $\JJ^{\circ}$, we obtain the set $\beta$-$\ext^{\circ}_{\Lambda_{\mM}}(\Lambda)$
of $\beta$-extensions of $\eta$ to  $\JJ^{\circ}$ relative to $\Lambda_{\mM}$, which consists of certain irreducible representations of  $\JJ^{\circ}$ extending $\eta$.

\subsection{Partitions}
Let  $\Delta=[\Lambda, n, r, \beta]$ be a  semisimple stratum.
A finite tuple $(\mathbf{e}_j)_{j\in S}$ of idempotents in $\End_{E\otimes D} V$ is called an \emph{$E\otimes D$-partition} of $V$ if
$\mathbf{e}_i \mathbf{e}_j=\delta_{ij}$ for all $i,j\in S$ with $i \neq j$, and  $\sum_{j\in S} \mathbf{e}_j=1$.

A $E\otimes D$-partition $(\mathbf{e}_j)_{j\in S}$ of $V$ is called  \emph{self-dual} if the set $\{\mathbf{e}_j\}_{j\in S}$ is $\sigma_h$-invariant with at most one fixed point.
We may then assume  $\{\pm 1,\cdots, \pm m\}  \subseteq  S  \subseteq \{0, \pm 1, \cdots \pm m\}$ for some positive integer $m$, such that~$\sigma_h(\mathbf{e}_i)=\mathbf{e}_{-i}$,\ $i\in S$. We write~$S_{\pm}$ for~$S\setminus\{0\}$.

Let  $(\mathbf{e}_j)_{j\in S}$ be an $E\otimes D$-partition of $V$ and let $W_j= \mathbf{e}_j V$ for all $j\in S$.
Then $(\mathbf{e}_j)_{j\in S}$ is said to be 
\begin{itemize}
\item  \emph{subordinate} to $\Delta$ if $V= \bigoplus_{j\in S} W_j$ is a splitting of  $\Delta$;
\item \emph{properly subordinate} to $\Delta$ if it is subordinate to $\Delta$ and  $\mathbf{e}_j+\tilde{\mathfrak{b}}_{1}$ is a central idempotent in $\tilde{\mathfrak{b}}_0 / \tilde{\mathfrak{b}}_{1}$ for all $j\in S$;
\item \emph{exactly subordinate} if it is properly subordinate, and can not be refined by another properly subordinate $E\otimes D$-partition.
\end{itemize}

We now recall the respective notions for the case of classical groups. 
For the remainder of this subsection, we suppose that $\Delta$ is self-dual.

Following~\cite[Definition 8.2]{daniel-3}, a self-dual $E\otimes D$-partition $(\mathbf{e}_j)_{j\in S}$ of $V$  is called 
\begin{itemize}
\item \emph{self-dual subordinate} to~$\Delta$ if it is subordinate to~$\Delta$,
\item \emph{properly self-dual subordinate} to~$\Delta$ if it is properly subordinate to~$\Delta$,
\item \emph{exactly self-dual subordinate} to~$\Delta$ if it is properly subordinate to $\Delta$, and can not be refined by another properly self-dual subordinate  $E \otimes D$-partition.
\end{itemize}

Accordingly, a decomposition $V= \bigoplus_{j\in S} W_j$ of $V$ is said to be \emph{self-dual (resp.\, subordinate, properly subordinate, exactly subordinate)}
if the corresponding $E\otimes D$-partition $(\mathbf{e}_j)_{j\in S}$ is self-dual (resp.\, subordinate, properly subordinate, exactly subordinate).

Now suppose that $\Delta$ is a self-dual semisimple stratum with the associated splitting $V=\bigoplus_{i\in I} V^i$.
\daniel{
We might relable the indexes in~$I_{+-}=I_+\cup I_-$ via~$I_+=\{1,\cdots, l\}$ and~$I_-=\{-1,\cdots, -l\}$ such  that $-i =\sigma(i)$, for all $i\in I_+$.}
The decomposition $V=\bigoplus\limits_{i=-l}^l V^i$ is then a self-dual decomposition where $V^0 \coloneqq \bigoplus\limits_{i\in I_0} V^i$, and is called the \emph{associated self-dual splitting} of $\Delta$.

\begin{rmk}\label{R:ExSubordinate}
With the notation above, an $E\otimes D$-partition $(\mathbf{e}_j)_{j\in S}$ properly self-dual subordinate to~$\Delta$  is exactly self-dual subordinate to $\Delta$ if and only if 
\begin{itemize}
	\item for $j\in S,\ j\neq 0$, there exists exactly one $i_j \in I$ such that  $W_j \subseteq V^{i_j}$,
		and $\tilde{\mathfrak{a}}(\Lambda_j)\cap B_j$ is a maximal $\mathfrak{o}_E$-$\mathfrak{o}_D$ order of $B_j$;
	\item and, if $0\in S$, $W_0$ is contained in $V^0$ and $\mathfrak{a}(\Lambda_0) \cap B_0$ is a maximal self-dual $\mathfrak{o}_E$-$\mathfrak{o}_D$-order in $B_0$.
\end{itemize}
\end{rmk}

Now let ${V}=\bigoplus\limits_{j\in S} {W}_{j}$ be a  self-dual decomposition, then
 $$M\coloneqq \{g\in G \mid g W_{j}=W_{j}, j \in S\}$$ is a Levi subgroup of $G$.
In fact, every Levi subgroup of  $G$ arises in this way.
Notice that $h|_{W_0 \times W_0}$ is non-degenerate, and for each  $j\neq 0$, $W_j$ is totally isotropic and $h|_{W_{-j} \times W_{j}}$ is a perfect paring. 
We thus have an isomorphism 
\begin{equation}\label{E:levi}
	M \cong (\Aut_D(W_0)\cap G) \times \prod\limits_{j=1}^m \Aut_D(W_j).
\end{equation}
Let $P$ be an arbitrary but fixed parabolic subgroup of $G$ with Levi component $M$.
We denote by $U$ its unipotent radical, and by $U_{-}$ the unipotent radical of the  parabolic subgroup opposite to $P$.
We have groups
\begin{align*}
&\HH^1_P \coloneqq \HH^1(\JJ^1 \cap {U}),
\quad
\JJ^1_P \coloneqq \HH^1 (\JJ^1 \cap {P}),
\quad
\JJ_P \coloneqq \HH^1 (\JJ\cap {P}),
\quad
\JJ_P^{\circ} \coloneqq \HH^1 (\JJ^{\circ} \cap {P}); \quad \text{and}\\
&\HH^1_M \coloneqq \HH^1_P \cap M= \HH^1 \cap M,
\quad
\JJ^1_M \coloneqq \JJ^1_P \cap M,
\quad
\JJ_M \coloneqq \JJ_P \cap M,
\quad
\JJ_M^{\circ} \coloneqq  \JJ_P^{\circ} \cap M.
\end{align*}

\begin{lem}[{cf. \cite[Corollary 5.11]{stevens-super}}]\label{L:H^1-levi}
Let $\Delta$ be a self-dual semisimple stratum and let $(\mathbf{e}_j)_{j\in S}$ be an $E\otimes D$-partition which is self-dual subordinate to  $\Delta$.
Then
\[
\HH^1(\beta,\Lambda) \cap M \cong \HH^1(\beta_0,\Lambda_0) \times \prod\limits_{j=1}^m \tilde{\HH}^1 (\mathbf{e}_j\beta \mathbf{e}_j, \Lambda_j).
\]
A similar decomposition holds for $\JJ^1(\beta,\Lambda)\cap M$.
Moreover, if the partition is properly subordinate, then similar decompositions hold for $\JJ(\beta,\Lambda)\cap M$ and $\JJ^{\circ}(\beta,\Lambda)\cap M$.
\end{lem}

\section{Construction of semisimple types}\label{S:types}
We fix a Levi subgroup $M$ of  $G$ and a cuspidal representation  $\tau$ of $M$.
Let
\begin{equation}\label{E:levi-decomposition}
	V= \bigoplus_{j\in S}W_{j}
\end{equation}
be the self-dual decomposition of $V$ corresponding to  $M$.
We write $\{\mathbf{e}_j\}_{j\in S}$ for the corresponding set of  idempotents.
By the isomorphism~\eqref{E:levi}, we may write
$\tau = \tau_0 \boxtimes \bigboxtimes\limits_{j=1}^m \tau_j$, where $\tau_0$ is an irreducible cuspidal representation of  $(\Aut_D(W_0)\cap G)$, and for all $1 \leq j \leq m$,  $\tau_j$ is an \daniel{irreducible cuspidal} representation of $\Aut_D(W_j)$.

By~\cite[Theorem 3.1]{daniel-3}, there exists a skew semisimple stratum $\Delta_0=[\Lambda_0,n_0,0,\beta_0]$ and a self-dual semisimple character $\theta_0 \in \operatorname{C}(\Delta_0)$ such that $\tau_0$ contains $\theta_0$.
Moreover, $n_0/e(\Lambda_0| F)$ coincides with the depth of  $\tau_0$ by Proposition 3.6 in loc.\ cit..
Let~$\tilde{\theta}_0\in\tilde{\CCC}(\Delta_0)$ be the Glauberman lift of~$\theta_0.$

Let $1\leq j \leq m$.
If $\tau_j$ is of positive depth, then by~\cite[Th\'eor\`em 4.1]{repVI} 
there exist a simple stratum $\Delta_j=[\Lambda_j,n_j,0,\beta_j]$ and a simple character $\tilde{\theta}_j\in \tilde{\operatorname{C}}(\Delta_j)$ such that $\tau_j$ contains $\tilde{\theta}_j$.
If $\tau_j$ is of depth zero, then $\tau_j$ contains the trivial simple character $\tilde{\theta}_j$ of some simple null stratum  $[\Lambda_j,n_j,0,\beta_j]$ with $n_j=0, \beta_j=0$, see~\cite[Remarque 3.14, Th\`eor\'em 4.5]{min-sec}.
Note that, in either situation we have that $n_j/e(\Lambda_j| F)$ coincides with the depth of  $\tau_j$.

We index the $\Aut_D(W_j)$ such that the following holds
\begin{itemize}
	\item \daniel{There exists a non-negative~$j_0$ such that for all positive~$j\leq j_0$  the characters~$\tilde\theta_j, \tilde\theta_{-j}$ are endo-equivalent. See \cite[Definition 6.19(iii)]{daniel-1} for the definition of endo-equivalent characters.}

	\item \daniel{For all~$j,j'$ greater than~$j_0$ the characters~$\tilde\theta_j, \tilde\theta_{-j'}$ are not endo-equivalent.}
\item \daniel{For any~$j$ greater than~$j_0$ and positive~$r$, if $\tilde\theta_{j}$ and $\tilde\theta_{j+r}$ are endo-equivalent then $\tilde\theta_j, \tilde\theta_{j+1}, \cdots, \tilde\theta_{j+r}$ are pairwise endo-equivalent. See \cite[Definition 6.19(iii)]{daniel-1} for the definition of endo-equivalent characters.}
\end{itemize}

\subsection{Construction of the underlying strata and characters}
We denote by $\mathscr{M}$ the stabilizer of~\eqref{E:levi-decomposition} in $A$, so that $\mathscr{M}_{-} \coloneqq \mathscr{M} \cap A_{-}$ is the Lie algebra of $M$.
The main result of this subsection is:

\begin{thm}\label{T:sdssc}
There exists a pair $(\Delta,\theta)$ consisting of a self-dual semisimple stratum $\Delta=[\Lambda, n, 0, \beta]$ with $\beta \in \mathscr{M}_{-}$, and a self-dual semisimple character $\theta$ of $\mathrm{H}^{1}(\beta, \Lambda)$ such that:
\begin{itemize}
\item[(i)] The decomposition~\eqref{E:levi-decomposition} is exactly subordinate to $\Delta$;
\item[(ii)] $\mathrm{H}^{1}(\beta, \Lambda) \cap {M}\cong \mathrm{H}^{1}(\beta_0,\Lambda_0) \times \prod\limits_{j=1}^{m} \tilde{\mathrm{H}}^{1}(\beta_j,\Lambda_j)$; 
\item[(iii)] $\left.\theta\right|_{\mathrm{H}^{1}(\beta,\Lambda) \cap \mathrm{M}} \cong   \theta_{0} \boxtimes \bigboxtimes \limits_{j=1}^{m} (\tilde{\theta}_{j})^2$.
\item[(iv)] \daniel{Let~$V=\oplus_{i=-l}^{l}V^i$ be the decomposition associated to~$\Delta$ then, for each non-zero index~$i$, the restriction~$\Lambda^i=\Lambda\cap V^i$ represents a barycenter of a face in~$\mathscr{B}_{red}(\GL_D(V^i))$.}
\end{itemize}
\end{thm}

This subsection is devoted to the construction of the pair $(\Delta,\theta)$ in Theorem~\ref{T:sdssc}, in a block-wise manner.

For any $1\leq j \leq m$, we put
$$\Delta_{-j} \coloneqq (\Delta_j)^{\#}=[(\Lambda_j)^\#,n_j,0,-\sigma_h(\beta_j)], \quad \text{and} \quad \theta_{-j} \coloneqq (\theta_j)^{\sigma}.$$

\begin{cons}\label{C:central}
Put  $\Lambda \coloneqq \oplus_{j\in S} \Lambda_j$. By Lemma~\ref{L:affine}, we may and do assume that $\Lambda$ is self-dual, and each $\mathbf{e}_j$ is a central idempotent in $\tilde{\mathfrak{a}}_0 (\Lambda)/\tilde{\mathfrak{a}}_1 (\Lambda)$. 
In the sequel of the proof of Theorem~\ref{T:sdssc} all occuring lattice sequences are restrictions of~$\Lambda$ to sub-vector spaces of~$V$.
\end{cons}

By~\cite[Lemma 7.3]{daniel-1} (cf. the proof of~\cite[Lemma 12.7]{endo}), for any $j\in \{1,\cdots, m\}$
there exists a pair $(\Delta_{M,j},\tilde{\theta}_{M,j})$ consisting of a semisimple stratum $\Delta_{M,j}$ split by $W_{-j} \oplus W_j$,
and $\tilde{\theta}_{M,j}\in \tilde{\operatorname{C}}(\Delta_{M,j})$, satisfying: 
\begin{itemize}
    \item the associated splitting of $\Delta_{M,j}$ is a coarsening of $W_{-j} \oplus W_j$ and $\Delta_{M,j} \mid_{W_{-j}}= \Delta_{-j}$;
    \item $\tilde{\operatorname{C}}(\Delta_{M,j} \mid_{W_j})=\tilde{\operatorname{C}}(\Delta_j)$;
    \item  $\tilde{\theta}_{M,j} \mid_{\tilde{\HH}(\Delta_{-j})} = \tilde{\theta}_{-j}$ and $\tilde{\theta}_{M,j} |_{\tilde{\HH}(\Delta_j)}= \tilde{\theta}_j$.
\end{itemize}

\begin{lem}
There is a pair $(\Delta_{M,\pm},\tilde{\theta}_{M,\pm})$ where $\Delta_{M,\pm}$ is a semisimple stratum  split by $\oplus_{j\in S_\pm} W_j$, and $\tilde{\theta}_{M,\pm}\in \tilde{\operatorname{C}}(\Delta_{M,\pm})$ is a semisimple character, satisfying:
\begin{itemize}
	\item the associated splitting of $\Delta_{M,\pm}$ is a coarsening of $\oplus_{j\in S_\pm} W_j$;
	\item $\tilde{\operatorname{C}}(\Delta_{M,\pm} |_{W_j})= \tilde{\CCC}(\Delta_j)$ for all $j\in S_\pm$;
\item $\tilde{\theta}_{M,\pm}|_{W_j}=\tilde{\theta}_j$ for all $j\in S_\pm$.
\end{itemize}
\end{lem}

\begin{proof}
	By replacing $V$ by $W_{-1} \oplus W_1$ and $V'$ by $W_{-2}\oplus W_2$ in~\cite[Lemma 7.3]{daniel-1} , we obtain a pair $(\Delta_M^{(1,2)},\tilde{\theta}_M^{(1,2)})$, where $$\Delta_M^{(1,2)}=[\Lambda_{\mathsf{M}}^{(1,2)}, n_M^{(1,2)},0,\beta_M^{(1,2)}]$$ is a semisimple stratum such that
	$$\Delta_M^{(1,2)}|_{W_{-1}\oplus W_1}=\Delta_{M,1}\text{  and }\tilde{\operatorname{C}} \left(\Delta_M^{(1,2)}|_{W_{-2}\oplus W_2} \right)=\tilde{\operatorname{C}}(\Delta_{M,2}),$$ and $\tilde{\theta}_M^{(1,2)}\in \tilde{\operatorname{C}}(\Delta_M^{(1,2)})$ is a semisimple character such that $$\tilde{\theta}_M^{(1,2)}|_{\tilde\HH^1(\Delta_{M,1})}=\tilde{\theta}_{M,1}\text{ and }\tilde{\theta}_M^{(1,2)}|_{\tilde\HH^1(\Delta_M^{(1,2)}|_{W_{-2}\oplus W_2})}=\tilde{\theta}_{M,2}.$$ Since $\tilde{\operatorname{C}} \left(\Delta_M^{(1,2)}|_{W_{-2}\oplus W_2} \right) \cap \tilde{\operatorname{C}}(\Delta_{M,2})\neq \emptyset$, Lemma~\ref{L:central-idempotents} provides idempotents ${\mathbf{1}}^{-2},{\mathbf{1}}^2$ of  $\End_D\big(W_{-2}\oplus W_2\big)$ such that ${\mathbf{1}}^{j} \equiv \mathbf{e}_{j}~ (\modu S(\Delta_{M.2})-1)$ for $j=\pm 2$. Put ${\mathbf{1}}^j=\mathbf{e}_j$ for $j=\pm 1$ and put  $g^{(1,2)}=\sum\limits_{j=\pm 1,\pm 2}\mathbf{e}_j {\mathbf{1}}^j$.  Then $g^{(1,2)}\in S \left(\Delta_M^{(1,2)} \right)$ and sends  ${\mathbf{1}}^j$ to $\mathbf{e}_j$. We replace $\Delta_M^{(1,2)}$ by $$\left[\Lambda_{\mathsf{M}}^{(1,2)}, n_M^{(1,2)},0,g^{(1,2)}\beta_M^{(1,2)} \big( g^{(1,2)} \big)^{-1} \right],$$
and $\tilde{\theta}_M^{(1,2)}$ by $\leftidx{^{g^{(1,2)}}}{\tilde{\theta}_M^{(1,2)}}$; and so we may assume that $\Delta_M^{(1,2)}$ is split by $\oplus_{-2\leq j\leq 2, j\neq 0} W_j$ and $\tilde{\theta}_M^{(1,2)} |_{W_j}=\tilde{\theta}_j$ for $j=\pm 1,\pm 2$. Finally, the lemma follows by iterating the above procedure.
\end{proof}

\begin{lem}\label{L:DeltaM}
There exists a pair $(\Delta_M, \tilde{\theta})$ consisting of a semisimple stratum $\Delta_M=[\Lambda, n, 0, \beta_M]$ split by $V=\bigoplus\limits_{j\in S} W_j$, and $\tilde{\theta}\in \tilde{\operatorname{C}}(\Delta_M)$, satisfying the following:
\begin{itemize}
	\item Let $V=\oplus_{i\in I} V^i$ be the associated splitting of $\Delta_M$ and put $V^0:=\oplus_{i\in I_0} V^i$ then $\oplus_{i\in I_\pm\cup\{0\}} V^i$ is a coarsening of $\oplus_{j\in S} W_j$ with $W_0 \subseteq V^0$;
	\item $\Delta_M \mid_{\oplus_{j\in S_\pm} W_j}= \Delta_{M,\pm}$;
    \item $\tilde{\operatorname{C}}(\Delta_M |_{W_0})=\tilde{\operatorname{C}}(\Delta_0)$;
    \item $\tilde{\theta} = \ttheta_{0} \otimes \ttheta_{M,\pm} \in \tilde{\operatorname{C}}(\Delta_M)$.
\end{itemize}
\end{lem}

\begin{proof}
Apply~\cite[Lemma 7.3]{daniel-1} with replacing $V$ by $W_0$ and $V'$ by $\oplus_{j\in S_\pm} W_j$.
\end{proof}

We consider the datum~$(\Delta_M,\ \ttheta)$ from Lemma~\ref{L:DeltaM}. 
Now, the stratum~$\Delta_M$ is a semisimple stratum, where the lattice sequence $\Lambda$ is self-dual.
Moreover, we have that $\ttheta |_{\tilde{\operatorname{H}}(\Delta_j)}=\theta_j$ for all $j\in S$. We need to modify~$\beta_M$, because~$\Delta_M$ may not be self-dual.

\begin{lem}\label{L:theta-stable}
	The set~$\tilde{\operatorname{C}}( \Delta_M)$, the group~$\tilde{\HH}(\Delta_M)$ and the character~$\ttheta$ are stable under~$\Sigma$.
\end{lem}
\begin{proof}
	By construction, we have that 
	\[
		\ttheta |_{\tilde{\operatorname{H}}(\Delta_{M,j})}= \ttheta_{M,j}= (\ttheta_{M,j})^\sigma=\ttheta^\sigma |_{\tilde{\operatorname{H}}(\Delta_{M,j})}
	\] 
for all $j\in S$.
Since $\theta_0$ is self-dual, we have that $\theta |_{\HH(\Delta_0)}= \theta^\sigma |_{\HH(\Delta_0)}$.
The lemma then follows from~\cite[Corollary 5.41]{daniel-1}.
\end{proof}


\begin{prop}\label{P:DeltaTtheta}
There exists a self-dual semisimple stratum $\Delta=[\Lambda, n, 0, \beta]$ to which the decomposition~\eqref{E:levi-decomposition} is properly subordinate such that $\ttheta$ is a~$\Sigma$-fixed element of~$\tilde{\CCC}(\Delta)$.
Moreover, we have $\tilde{\HH}^1(\mathbf{e}_j \beta \mathbf{e}_j, \Lambda)= \tilde{\HH}^1(\beta_j, \Lambda)$ for all $j\in S$.
\end{prop}

\begin{proof}
Let $\{\mathbf{1}^i\}_{i\in I}$ be the set of idempotents according to the associated splitting of $\Delta_M$.
Put $\mathbf{1}^{ij} \coloneqq \mathbf{1}^i\mathbf{e}_j$ for all $j\in S$ and  $i\in I$.
By the construction of $\Delta_M$,  for any non-zero index $j\in S$, there is  a unique $i_j\in I$ such that $\mathbf{e}_j= \mathbf{1}^{i_j j}$, and in particular $\mathbf{1}^{ij}=0$ whenever $i\neq i_j$.
We then observe that each $\mathbf{1}^{ij}$ lies in $\tilde{\mathfrak{b}}_0(\beta_i, \Lambda)$,
that $\mathbf{1}^{ij}+\tilde{\mathfrak{b}}_1(\beta_M, \Lambda)$ is central in $\tilde{\mathfrak{b}}_0(\beta_M,  \Lambda)/ \tilde{\mathfrak{b}}_1(\beta_M, \Lambda)$ (note that the $\mathbf{e}_j$ satisfy the assumption in Construction~\ref{C:central}), and that the $\mathbf{1}^{ij}$ sum up to $1$.

By~\cite[Proposition 6.18]{daniel-2}, there exists a self-dual semisimple stratum $$\Delta^\circ=[\Lambda, n_M,0,\beta^\circ]$$ such that $\tilde{\operatorname{C}}(\Delta_M)=\tilde{\operatorname{C}}(\Delta^\circ)$.
By Lemma~\ref{L:central-idempotents}, there exits an idempotent $\mathbf{1}^{\circ ij}\in \tilde{\mathfrak{b}}_0( \beta^\circ, \Lambda)$ such that $\mathbf{1}^{\circ ij} \equiv \mathbf{1}^{ij}~ (\modu S(\Delta_M)-1)$,
for all $i\in I,\ j \in S$.
Moreover, each $\mathbf{1}^{\circ ij}+\tilde{\mathfrak{b}}_{1}(\beta^\circ, \Lambda)$ is central in $\tilde{\mathfrak{b}}_{0}( \beta^\circ, \Lambda)/ \tilde{\mathfrak{b}}_{1}(\beta^\circ, \Lambda)$.
By imitating the proof of~\cite[Proposition 9.9 (iv)]{daniel-int}, we conclude that  the $\mathbf{1}^{\circ ij}$ sum up to $1$.

We claim that there exists $g\in S(\Delta_M) \cap G$ such that $g \mathbf{1}^{ij} g^{-1}= \mathbf{1}^{\circ ij}$ for all $j \in S$ and  $i\in I$.
Let $g_0 = \sum\limits_{-m \leq j \leq m, i \in I} \mathbf{1}^{\circ ij} \mathbf{1}^{ij}$.
Then $g_0\in  \operatorname{S}(\Delta_M)$ and
$g_0\mathbf{1}^{ij} g_0^{-1}=\mathbf{1}^{\circ ij}$ for all indexes~$i,j$.
Put $\mathfrak{S} \coloneqq \{s\in S(\Delta_M) \mid s \mathbf{1}^{\circ ij} s^{-1}= \mathbf{1}^{\circ ij} ~\text{for~all}~ i,j \}$, which is a pro-$p$ group.
Notice that $\Sigma$ (which is of order $2$) acts on $g_0 \mathfrak{S}$.
Apply~\cite[Theorem 2.12 (ii) (a)]{KS2020} with replacing $U_1$ by the trivial group,  $g$ by $g_0$,  $U_2$ by  $\mathfrak{S}$, and $G$ by $\tilde G$,
we find an element  $g\in g_0 \mathfrak{S}$ such that  $\sigma(g)=g$, as claimed.

Because $g\in \operatorname{S}(\Delta_M)$, we have that $g$ normalizes  $\Lambda$ and then
$g^{-1} \tilde{\mathfrak{b}}_k (\beta^\circ, \Lambda) g = \tilde{\mathfrak{b}}_k (g^{-1}\beta^\circ g, \Lambda)$
for $k=0,1$.
For any non-zero $j\in S$, the image of $\mathbf{e}_j=\mathbf{1}^{i_j j}=g^{-1} \mathbf{1}^{\circ i_j j}g$ in the quotient $\tilde{\mathfrak{b}}_0 (g^{-1}\beta^\circ g, \Lambda)/\tilde{\mathfrak{b}}_1 (g^{-1}\beta^\circ g, \Lambda)$ is central; and so is $\mathbf{e}_0$ since $\mathbf{e}_0 = 1- \sum\limits_{j\neq 0} \mathbf{e}_j$. 
Thus, $(\mathbf{e}_j)_{j\in S}$ is properly subordinate to the self-dual semisimple stratum $[\Lambda, n, 0, g^{-1}\beta^\circ g]$.
Put $\beta=g^{-1}\beta^\circ g$.
Then the decomposition~\eqref{E:levi-decomposition} is properly subordinate to  $\Delta:=[\Lambda,n,0,\beta]$ as required.
Because  $g\in S(\Delta_M)$, $g$ normalizes~$\tilde{\HH}^1(\beta^\circ,\Lambda)$ and~$\theta$.
It follows that  
\[
\tilde{\HH}^1(\Delta)=\tilde{\HH}^1(g^{-1}\beta^\circ g,\Lambda)= {^g}  \tilde{\HH}^1(\beta^\circ, \Lambda)
=\tilde{\HH}^1(\beta^\circ, \Lambda)= \tilde{\HH}^1(\Delta_M),
\]
and~$\ttheta$ is an element of~$\tilde{\CCC}(\Delta)$.
By Lemma~\ref{L:theta-stable}, $\tilde{\theta}$ is stable under $\Sigma$.
Moreover, for any $j\in S$, we have
 \[
\tilde{\HH}^1(\mathbf{e}_j \beta \mathbf{e}_j, \Lambda) = \tilde{\HH}^1(\Delta) \cap \Aut_D(W_j) =
\tilde{\HH}^1(\Delta_M) \cap \Aut_D(W_j) = \tilde{\HH}^1(\beta_j, \Lambda)
\] 
since the \daniel{partition~$(\mathbf{e}_j)_{j\in S}$} is subordinate to $\Delta$ and $\Delta_M$.
\end{proof}

\begin{proof}[{Proof of Theorem~\ref{T:sdssc}}]
\daniel{We take~$(\Delta,\ttheta)$ from Proposition~\ref{P:DeltaTtheta}. For non-zero~$i$ the lattice sequence~$\Lambda^i$ represents a point of a facet in~$\mathscr{B}(\GL_D(V^i))$. We now replace~$\Lambda$  by a lattice sequence such that, for all non-zero~$i$,~$\Lambda^i$ represents the barycenter of the mentioned facet. This changes~$\Delta$ and we replace~$\ttheta$ by its transfer. We denote the restriction of~$\ttheta$ to~$\HH(\Delta)$ by~$\theta$.}
Assertion (ii) follows from Lemma~\ref{L:H^1-levi}.
From the decomposition in (ii) we obtain $\theta \cong \theta_0 \otimes \bigotimes_{j=1}^m (\ttheta_j)^{2}$.
Assertion (iii) then follows from the bijection between $\tilde{\operatorname{C}}(\beta_j,\Lambda_j)$ and $\tilde{\operatorname{C}}(2\beta_j,\Lambda_j)$ given by $\ttheta_j \mapsto (\ttheta_j)^2$ for all  $1\leq j \leq m$.

It remains to show that $(\mathbf{e}_j)_{j\in S}$ is exactly subordinate to~$\Delta$. 
Since $\tau_0$ contains  $\theta_0$, it contains some $\kappa_0 \otimes \rho_0$, where  $\kappa_0$ is a  $\beta_0$-extension of  $\theta_0$ and  $\rho_0$ a representation of  $\JJ(\beta_0,\Lambda_0)$ acting trivially on $\JJ^1(\beta_0,\Lambda_0)$.
It follows from~\cite[Theorem 11.1]{daniel-3} that $\kappa_0 \otimes \rho_0$ is a cuspidal type.
Hence, $\operatorname{P}^{\circ}(\Lambda_{\beta}^0 |_{W_0})$ is a maximal parahoric.
Likewise, for each $1\leq j \leq m$, there exists $i_j$ such that $W_j \subseteq V^{i_j}$, and $\tilde{\PP}(\Lambda^{i_j}_{\beta} |_{W_j})$ is a maximal parahoric by~\cite[Proposition 5.15, Corollaire 5.20]{repIV}.
Remark~\ref{R:ExSubordinate} finishes the proof.
\end{proof}

We have $\CCC(\beta_0,\Lambda_0) = \CCC(\mathbf{e}_0 \beta \mathbf{e}_0, \Lambda_0)$ and $\tilde\CCC(\beta_j,\Lambda_j) = \tilde\CCC (\mathbf{e}_j \beta \mathbf{e}_j, \Lambda_j)$ for all $j>0$.
In particular, we have $\JJ(\beta_0,\Lambda_0) = \JJ(\mathbf{e}_0 \beta \mathbf{e}_0, \Lambda_0)$ and $\tilde\JJ(\beta_j,\Lambda_j) = \tilde\JJ (\mathbf{e}_j \beta \mathbf{e}_j, \Lambda_j)$ for all $j>0$.
We henceforth identify $\beta_j$ with $\mathbf{e}_j \beta \mathbf{e}_j$ for all $j\in S$.

\subsection{Construction of types}\label{S:type}
	
Let $\eta$ be the Heisenberg extension of $\theta$.
We fix a standard  $\beta$-extension $\kappa$ of $\theta$ (relative to some \emph{standard} $\Lambda_{\mathsf{M}}$ such that $\tilde{\mathfrak{b}}(\Lambda_{\mathsf{M}})$ is maximal) to $\JJ$. \daniel{For the definition of standard~$\beta$-extension, see~\cite[Definition 10.8]{daniel-3} with Property (ORD) for the~$GL$-part.}

We define the character $\theta_P$ of $\HH^1_P$ via 
\[\theta_P(xy) \coloneqq \theta(x), \quad x\in \HH^1, y\in \JJ^1 \cap {U}.
\]
We define $\eta_P$ to be the representation of  $\JJ^1_P$ on the $(\JJ^1 \cap U)$-fixed-vectors in $\eta$,
and $\kappa_P$ the representation of  $\JJ_P$ on the $(\JJ^{1} \cap U)$-fixed-vectors in $\kappa$.

By~\cite[Proposition 8.5]{daniel-3}, $\kappa_P$ is an extension of $\eta_P$, and $\eta_P$ is the Heisenberg extension of  $\theta_P$.
Moreover, we have  $\ind_{\JJ_P^1}^{\JJ^1} \eta_P \cong \eta$ and $\ind_{\JJ_P}^{\JJ} \kappa_P \cong \kappa$. In particular, as~$G_\beta$ intertwines~$\theta_P$ and as the intertwining spaces of~$\eta$ are one-dimensional we obtain that
\[\operatorname{I}_G(\theta_P)=\operatorname{I}_G(\eta_P)=\JJ^1_P G_\beta\JJ^1_P .\]
Since the decomposition~\eqref{E:levi-decomposition} is exactly subordinate to $\Delta$,
we conclude from~\cite[Corollary 8.7]{daniel-3} that $\kappa_M \coloneqq \kappa_P |_{\JJ\cap M}$ is a standard $\beta$-extension of  $\eta_M \coloneqq \eta_P |_{\JJ^1 \cap M}$ 	.
Namely,  $\kappa_{M}=\kappa_{0} \boxtimes \bigboxtimes_{j=1}^{m} \tilde{\kappa}_{j}$, where $\tilde{\kappa}_{j}$ is a standard $2 \beta_{j}$-extension containing $\tilde{\theta}_{j}$ and $\kappa_{0}$ is a standard $\beta_{0}$-extension containing $\theta_{0}$.

Since $\tau$ contains  $\theta_M \coloneqq \theta|_{\operatorname{H}^1 \cap M}$, it contains  $\eta_M \coloneqq \eta_{P}|_{\operatorname{H}^1 \cap M}$, and hence an extension $\lambda_M^{\circ}$ of  $\eta_M$ to $\JJ^\circ \cap M$,
which is of the form $\lambda_M^\circ=\kappa_M \otimes \rho_M^{\circ}$ for some irreducible representation $\rho_M^{\circ}$ of $\JJ^{\circ} \cap M$ inflated from
$	\frac{\JJ^{\circ}\cap M}{\JJ^1\cap M}$.
By the isomorphism
\[
	\frac{\JJ^{\circ}\cap M}{\JJ^1\cap M}\cong \frac{\operatorname{P}^{\circ}(\Lambda_\beta)}{\operatorname{P}^1(\Lambda_\beta)} \cong \frac{\JJ^{\circ}_P}{\JJ^1_P}.
\]
we may view $\rho_M^{\circ}$ as a representation of  $\JJ^{\circ}_P$.

We then form the representation
\[
	\lambda_P^{\circ} = \kappa_P |_{\JJ^{\circ}_P} \otimes \rho_M^{\circ}.
\]

Because $(\mathbf{e}_j)_{j=-m}^m$ is exactly subordinate, we have
\[
	\operatorname{P}^{\circ}(\Lambda_\beta)/ \operatorname{P}_1(\Lambda_\beta) \cong \operatorname{P}^{\circ}(\Lambda^0_\beta |_{W_0})/ \operatorname{P}_1(\Lambda^0_\beta |_{W_0}) \times \prod_{j=1}^m \tilde{\operatorname{P}}(\Lambda^{i_j}_\beta |_{W_j})/ \tilde{\operatorname{P}}_1(\Lambda^{i_j}_\beta |_{W_j}).
\] 
Thus, we may write $\rho_M^{\circ}=\rho^{\circ}_0 \boxtimes \bigboxtimes_{j=1}^m \tilde{\rho}^{\circ}_j$,
where $\rho^{\circ}_0$ is a representation of $\operatorname{P}^{\circ}(\Lambda^0_\beta |_{W_0})$ inflated from an irreducible representation of  $\operatorname{P}^{\circ}(\Lambda^0_\beta |_{W_0})/ \operatorname{P}_1(\Lambda^0_\beta |_{W_0})$,
and each $\tilde{\rho}^{\circ}_j$ is a representation of $\tilde{\operatorname{P}}(\Lambda^{i_j}_\beta |_{W_j})$ inflated from an irreducible representation of  $\tilde{\operatorname{P}}(\Lambda^{i_j}_\beta |_{W_j})/ \tilde{\operatorname{P}}_1(\Lambda^{i_j}_\beta |_{W_j})$.
Moreover, we may write $\lambda^{\circ}_M= (\kappa_0\otimes \rho^{\circ}_0) \boxtimes \bigboxtimes_{j=1}^m (\tilde{\kappa}_j\otimes \tilde{\rho}^{\circ}_j)$.
It follows from~\cite[Theorem 9.1]{daniel-3} and~\cite[Proposition 5.15, Corollaire 5.20]{repIV} that $\rho^{\circ}_0$ and all $\tilde{\rho}^{\circ}_j$ are cuspidal representations.

We have the following:

\begin{thm}\label{T:cover}
	The pair $(\JJ_P^\circ, \lambda_P^{\circ})$ is a cover of $(\JJ_P^\circ \cap M, \lambda_M^{\circ})$.
\end{thm}

\begin{rmk}\label{R:cover}
We collect some useful facts from the theory of covers, following~\cite{BK-types}.
We say that a pair $(\JJ, \lambda)$ is \emph{a decomposed pair} over $(\JJ_{\mathrm{M}}, \lambda_{M})$ if, for any parabolic subgroup $P=MU$ with Levi factor $M$,
\begin{itemize}
\item[(i)] J has an Iwahori decomposition with respect to $(M,P)$ and $\JJ \cap M=\JJ_{M}$; and
\item[(ii)] $\lambda$ restricts to $\lambda_{\mathrm{M}}$ on $\mathrm{J}_{\mathrm{M}}$, and to a multiple of the trivial representation on $J \cap U$.
\end{itemize}
Such a decomposed pair $(\JJ, \lambda)$ is said to be a \emph{$G$-cover} of $(\JJ_M, \lambda_M)$ if
\begin{itemize}
\item[(iii)] for every parabolic subgroup $P$ of $G$ with Levi component $M$, there exists an invertible element of $\mathscr{H}(G, \tau)$ supported on a double coset $J z_P J$, where $z_P \in Z(M)$ is strongly $(P, J)$-positive.
\end{itemize}

Suppose that $(\JJ, \lambda)$ is a decomposed pair over $(\JJ_{\mathrm{M}}, \lambda_{M})$ such that 
\[
	\mathscr{H}(G,\lambda)_M \coloneqq \{ \varphi\in \mathscr{H}(G,\lambda) \mid \Supp(\varphi) \subseteq \JJ M\JJ \}
\] 
is a subalgerba of $\mathscr{H}(G,\lambda)$, then Condition (iii) is satisfied. 
Moreover, we have a support-preserving isomorphism of Heck algebras:
\[
\mathscr{H}(G,\lambda) \cong \mathscr{H}(M,\lambda_M)
\] 
See~\cite[Theorem (7.2) and Remarks (8.2)]{BK-types}.

As explained in~\cite[Remark 5.4]{miyauchi}, for any parabolic subgroup $P'$ containing  $M$ as its Levi component, there exists a support-preserving isomorphism
\[
	\mathscr{H}(G, \lambda^{\circ}_P) \cong \mathscr{H}(G, \lambda^{\circ}_{P'}).
\] 
Therefore, the condition on Hecke algebras in the definition of covers (\cite[Definition 8.2 (iii)]{BK-types}) is satisfied whenever the condition is proved true for $P$.

\end{rmk}

Granting Theorem~\ref{T:cover}, we can now prove the main theorem of this paper.

\begin{proof}[Proof of Theorem~\ref{T:main}]

	By construction, $\tau$ contains  $\lambda_{\operatorname{M}}^{\circ}$ and hence a representation $\lambda_M=\kappa_M\otimes \rho_M$ of $\JJ \cap M= \JJ_P \cap M$, where $\rho_M$ is a representation of  $\operatorname{P}(\Lambda_\beta)$ inflated from $\operatorname{P}(\Lambda_\beta)/\operatorname{P}_1(\Lambda_\beta)$ which contains $\rho^{\circ}_M$.
It then follows from~\cite[Theorem 9.1]{daniel-3} and~\cite[Proposition 5.15, Corollaire 5.20]{repIV} that $(\JJ\cap M, \lambda_M)$ is a cuspidal type in $M$,
which is an $[M,\tau]_M$-type.

We put $\lambda_P= \kappa_P \otimes \rho_M$, which is an extension of $\lambda_M$.
Then  $( \JJ_P,\lambda_P)$ is a decomposed pair above $(\JJ_M, \lambda_P \cap M)$.
Since $(\JJ_P^\circ, \lambda_P^{\circ})$ is a cover of $(\JJ_P^\circ \cap M, \lambda_M^{\circ})$ by Theorem~\ref{T:cover},
we conclude from~\cite[Lemma 3.9]{level-0} that $(\JJ_P, \lambda_P)$ is a cover of $(\JJ_M, \lambda_P \cap M)$.
It follows from~\cite[Theorem 8.3]{BK-types} that $(\JJ_P, \lambda_P)$ is an $[M, \tau]_{G}$-type, as desired.

\end{proof}

\subsection{}
We now turn to the proof of Theorem~\ref{T:cover}.
Let $L=L(\Delta)$ be the Levi subgroup corresponding to the associated self-dual decomposition of  $\Delta$.
We then have an identification
\[
L=G^0 \times \prod\limits_{i=1}^l \tilde{G}^i,	 
\] 
where $G^0=(\Aut_D(V^0)\cap G)$, and $\tilde{G}^i=\Aut_D(V^i)$ for all $1\leq i \leq l$.
Let $Q$ a parabolic subgroup of $G$ with Levi component $L$ such that $P \subseteq Q$.
We then have $M \subseteq L$.
We form $\lambda_{L}^{\circ}=\lambda_{P}^{\circ}|_{\JJ_{P}^{\circ} \cap L}$.

\begin{lem}[{\cite[Lemma5.5]{miyauchi}}]\label{L:P-L}
	The pair $(\JJ^{\circ}_P, \lambda^{\circ}_P)$ is a cover of $(\JJ^{\circ}_P \cap L, \lambda^{\circ}_L)$.
	Moreover, we have an support-preserving Hecke algebra isomorphism
	\[
		\mathscr{H}(G, \lambda^{\circ}_P) \cong \mathscr{H}(L, \lambda^{\circ}_L).
	\] 
\end{lem}

\begin{proof}
It is clear that $(\JJ^{\circ}_P,\lambda^{\circ}_P)$ is a decomposed pair above $(\JJ^{\circ}_P \cap M,\lambda^{\circ}_M)$, then by~\cite[Proposition 8.5 (ii)]{BK-types}, $(\JJ^{\circ}_P,\lambda^{\circ}_P)$ is a decomposed pair above $(\JJ^{\circ}_P \cap L,\lambda^{\circ}_L)$.
By~\cite[Theorem 7.2, Comments 8.2]{BK-types} (cf.\,Remark~\ref{R:cover}), it suffices to show that $\mathscr{H}(G, \lambda^{\circ}_P)_L =  \mathscr{H}(G,\lambda^{\circ}_P)$.
The support of $\mathscr{H}(G, \lambda^{\circ}_P)$ is $\operatorname{I}_{G}(\lambda^{\circ}_P)$, it thus suffices to show that $$\operatorname{I}_{G}(\lambda^{\circ}_P) \subseteq \JJ^\circ_{P} L \JJ^\circ_{P}.$$
In fact, we have a chain of containments:
\daniel{
\[
	\operatorname{I}_{G}(\lambda^{\circ}_P) \subseteq \operatorname{I}_G(\eta_P)  = \JJ^1_{\PP} G_\beta \JJ^1_{\PP} \subseteq \JJ^1_{\PP} L \JJ^1_{\PP},
\]} 
\daniel{in which the first inclusion is trivial.}
\daniel{The restriction of $\eta_P$ to $\JJ^1_M$ is equivalent to $\eta_M $ (see the proof of~\cite[Theorem 8.5]{daniel-3}), so all elements of~$G_{\beta}$ intertwine $\eta_P$.
Now the equality follows, because~$\eta_P$ induces to~$\eta$ whose intertwining spaces all have at most dimension~$1$.
}

\end{proof}

We next show that $(\JJ^{\circ}_P \cap L, \lambda^{\circ}_L)$ is a cover of $(\JJ^{\circ}_P \cap M, \lambda^{\circ}_M)$.

\subsection{}
We first introduce some notations following~\cite[\S 5.3]{miyauchi}.
We write $(\mathbf{1}^i)_{i=-l}^l$ for the tuple of idempotents of the associated self-dual splitting $V=\bigoplus\limits_{i=-l}^l V^i$ of $\Delta$, and write $\beta^{(i)}=\mathbf{1}^i \beta \mathbf{1}^i$ for all $-l\leq i \leq l$.
For each $-l \leq i \leq l$, there is a subset  $J_i\subseteq \{-m,\cdots,m\}$ such that  $V^i=\bigoplus_{j\in J_i} W_j$.

Let $U_Q$ be the unipotent radical of $Q$ and let $\kappa_Q$ be the representation of $\JJ_Q$ on the space of  $\JJ \cap U_Q$-fixed vectors in $\kappa$.
\daniel{Because as~$L$ is properly subordinate to~$\Delta$, we have  $\JJ_Q \cap L= \JJ \cap L$.} 
We write $\JJ_L \coloneqq \JJ\cap L$ and $\JJ^\circ_L \coloneqq \JJ^\circ \cap L$, and then define
\[
	\kappa^\sharp \coloneqq \kappa_Q|_{\JJ_L} \quad \text{and} \quad (\kappa^\sharp)^\circ \coloneqq \kappa_Q|_{\JJ^\circ_L}.
\] 
Notice that $P\cap L=M (U\cap L)$ is a parabolic subgroup of  $L$ with Levi component  $M$.
We have a decomposition
\[
\JJ^\circ_P \cap L = (\HH^1 \cap L) (\JJ^\circ_L \cap P).
\] 
Let $(\kappa^{\sharp})^\circ_{P \cap L}$ be the representation of $\JJ^{\circ}_P \cap L$ on the space of $\JJ^{\circ}_L \cap U$-fixed vectors in $(\kappa^{\sharp})^\circ$.
Then according to~\cite[Proposition 3.11]{miyauchi}, we have
\[
	(\kappa^{\sharp})^\circ_{P \cap L} \cong \kappa_P|_{\JJ^\circ_P \cap L}.
\]

Moreover, $\kappa^{\sharp}$ is a standard $\beta$-extension of $\theta_L$ to $\JJ_L$, and we can write $\lambda^{\circ}_L= (\kappa^{\sharp})^\circ_{P \cap L} \otimes \rho^{\circ}_M$.

Using the identification $L=G^0 \times \prod\limits_{i=1}^l \tilde{G}^i$, we may write 
$\theta|_{\operatorname{H}^{1}\cap L}= \theta^{(0)} \boxtimes \bigboxtimes\limits_{i=1}^l \theta^{(i)}$  and $(\kappa^{\sharp})^\circ= \kappa^{(0)}\boxtimes \bigboxtimes\limits_{i=1}^l \kappa^{(i)}$
where $\kappa^{(0)}$ is a standard $\beta^{(0)}$-extension of  $\theta^{(0)}$, and for each  $i>0$,  $\tilde\kappa^{(i)}$ is a  $2\beta^{(i)}$-extension of  $\theta^{(i)}$.

We have $P \cap L=P^{0}\times \prod\limits_{i=1}^l \tilde{P}^i$, where $P^{0}$ is a parabolic subgroup of  $G^{0}$ and  $\tilde{P}^i$ is a parabolic subgroup of $\tilde{G}^{i}$ for all $i>0$.
We put $(\rho^{(0)})^{\circ}=\rho^{\circ}_0 \boxtimes \bigboxtimes_{j\in J_0} \tilde{\rho}_j$, and $\tilde{\rho}^{(i)}= \bigboxtimes_{j\in J_i} \tilde{\rho}_j$ for $i>0$.
We further put $(\lambda^{(0)})^\circ= (\kappa^{(0)})^\circ \otimes (\rho^{(0)})^\circ$, and $\tilde{\lambda}^{(i)}=\tilde{\kappa}^{(i)}\otimes \tilde{\rho}^{(i)}$ for $i>0$.

Moreover, we have a decomposition $\JJ^{\circ}_P \cap L \cong \JJ^{\circ}_{P^{0}} \times \prod\limits_{i=1}^l \tilde{J}_{\tilde{P}^i}$ by Lemma~\ref{L:H^1-levi}.
We thus have decompositions 
$(\kappa^{\sharp})^\circ_{P \cap L}\cong (\kappa^{(0)}_{P^0})^\circ \boxtimes \bigboxtimes_{i=1}^l \tilde{\kappa}^{(i)}_{\tilde{P}^i}$
and
$\lambda^{\circ}_{L}\cong (\lambda^{(0)}_{P^0})^\circ \boxtimes \bigboxtimes_{i=1}^l \tilde{\lambda}^{(i)}_{\tilde{P}^i}$.

We identify $M$ with  $M^{0}\times \prod\limits_{i=1}^l \tilde{M}^i$, where  $M^{0}= M\cap G^{0}$ and $\tilde{M}^i =M\cap \tilde{G}^{i}$ for $1 \leq i \leq l$.
So, in order to prove that $(\JJ^{\circ}_P \cap L, \lambda^{\circ}_L)$ is a cover of $(\JJ^{\circ}_P \cap M, \lambda^{\circ}_M)$,
it suffices to show:
\begin{itemize}
	\item[(i)] $\left(\JJ^{\circ}_{P^{0}}, \left(\lambda^{(0)}_{P^0} \right)^\circ \right)$ is a cover of $\left(\JJ^{\circ}_{P^{0}}\cap M^{0}, \left(\lambda^{(0)}_{P^0} \right)^\circ \big|_{\JJ^{\circ}_{P^{0}}\cap M^{0}}\right)$;
	\item[(ii)] $\left(\tilde{\JJ}_{\tilde{P}^{i}},\tilde{\lambda}^{(i)}_{\tilde{P}^i} \right)$ is a cover of $\left(\tilde{\JJ}_{\tilde{P}^{i}} \cap \tilde{M}^i,\tilde{\lambda}^{(i)}_{\tilde{P}^i} \big|_{\tilde{\JJ}_{\tilde{P}^{i}} \cap \tilde{M}^i}\right)$ for $i>0$. 
\end{itemize}

Case (ii) is proved by~\cite[Proposition 7.1]{repVI}.
To prove case (i), it amounts to prove Theorem~\ref{T:cover} in the situation of $L=G$ (in particular, the stratum  is skew); this will be done in the following subsection.

\subsection{}
For a skew semisimple stratum $\Delta=[\Lambda,n,0,\beta]$, we need to show that $(\JJ_P^\circ, \lambda_P^{\circ})$ is a cover of $(\JJ_P^\circ \cap M, \lambda_M^{\circ})$.
By Lemmas~\ref{L: M'1} and~\ref{L: M'2} below and the transitivity of
covers, this will follow once we have constructed a suitable
intermediate Levi subgroup \(M'\).
Note that in our quaternionic case, $G_\beta$ has compact centre.

We proceed by induction on $m= m(G,M)$, the number appearing in the decomposition $V = \bigoplus_{j=-m}^{m} W_j$ given by $M$.
Suppose $m=0$, then the decomposition has only the block $W_0$, so $M=G$.
In particular, we have  $(\JJ_P^\circ \cap M, \lambda_M^{\circ})= (\JJ_P^\circ, \lambda_P^{\circ})$ and there is nothing to prove.

Fix $m \geq 1$.
Suppose that $G^{(0)}$ is a quaternionic form of classical group, $M^{(0)}$ a Levi subgroup, $P^{(0)}$  a parabolic subgroup with Levi factor $M^{(0)}$, and $\Delta^{(0)}$ a skew semisimple stratum to which the decomposition given by $M^{(0)}$ is exactly subordinate.
We assume that whenever $(G^{(0)}, M^{(0)}, P^{(0)}, \Delta^{(0)})$ satisfies the hypothesis as above and $m(G^{(0)}, M^{(0)})<m$, then $(\JJ_{P^{(0)}}^\circ, \lambda_{P^{(0)}}^{\circ})$ is a $G^{(0)}$-cover of $\big(\JJ_{P^{(0)}}^\circ \cap M^{(0)}, \lambda_{P^{(0)}}^{\circ}|_{\JJ_{P^{(0)}}^\circ \cap M^{(0)}}\big)$.
 
Recall that we write  $\rho_M^{\circ}=\rho^{\circ}_0 \boxtimes \bigboxtimes_{j=1}^m \tilde{\rho}^{\circ}_j$.
For $j>0$, we define  $\tilde{\rho}_{-j} \coloneqq \tilde{\rho}_j \circ \sigma_j$ where $\sigma_j$ is an involution of $\Aut_D(W_j)$ defined in~\cite[\S 6.2]{stevens-super}.
Then either there is $j>0$ such that $\tilde{\rho}_j \ncong \tilde{\rho}_{-j}$, or $\tilde{\rho}_j \cong \tilde{\rho}_{-j}$ for all $j\in S$;
these two cases are Case 1 and Case 2 in~\cite[\S 12.3, (iii)]{daniel-3}, respectively.
In either case, we obtain a properly subordinate decomposition 
\[
	V = Y^{(-1)} \oplus Y^{(0)} \oplus Y^{(1)},
\] 
and denote its corresponding Levi by $M'$.
By its construction, we have $m \big(G^{(0)}, M^{(0)}\big)<m$.
By the construction of strongly positive elements in~\cite[§12.3]{daniel-3} (cf. \cite[Proposition~7.13]{stevens-super}), we obtain:
\begin{lem}\label{L: M'1}
$(\JJ_P^\circ, \lambda_P^{\circ})$ is a cover of $(\JJ_P^\circ \cap M', \lambda_P^{\circ}|_{\JJ_P^\circ \cap M'})$.
\end{lem}

We next descend the cover from $M'$ to $M$:

\begin{lem}\label{L: M'2}
$(\JJ_P^\circ \cap M', \lambda_P^{\circ}|_{\JJ_P^\circ \cap M'})$ is a cover of $(\JJ_P^\circ \cap M, \lambda_M^{\circ})$.
\end{lem}
\begin{proof}
Put $G^{(0)} = \Aut_D(Y^{(0)}) \cap G$ and $\tilde{G}^{(1)} = \Aut_D(Y^{(1)})$.
We then have $M' \cong G^{(0)} \times \tilde{G}^{(1)}$.
We also set $M^{(0)} = M \cap G^{(0)}$ and $\tilde{M}^{(1)} = M \cap \tilde{G}^{(1)}$,
and then $M \cong M^{(0)} \times \tilde{M}^{(1)}$.
We have
\[
   \JJ_P^\circ\cap M'  \cong (\JJ_P^\circ\cap G^{(0)})\times(\JJ_P^\circ\cap\tilde G^{(1)}).
\]
Therefore, the lemma is equivalent to showing the following assertions:
\begin{itemize}
\item[(i)] $(\JJ_P^\circ \cap G^{(0)}, \lambda_P^{\circ}|_{\JJ_P^\circ \cap G^{(0)}})$ is a cover of $(\JJ_P^\circ \cap M^{(0)}, \lambda_P^{\circ}|_{\JJ_P^\circ \cap M^{(0)}})$;
\item[(ii)] $\big(\JJ_P^\circ \cap \tilde{G}^{(1)}, \lambda_P^{\circ}|_{\JJ_P^\circ \cap \tilde{G}^{(1)}}\big)$ is a cover of $\big(\JJ_P^\circ \cap \tilde{M}^{(1)}, \lambda_P^{\circ}|_{\JJ_P^\circ \cap \tilde{M}^{(1)}}\big)$.
\end{itemize}
Notice that  $P^{(i)} \coloneqq P \cap G^{(i)}$ is a parabolic subgroup containing $M^{(i)}$ and $\JJ_P^\circ \cap G^{(i)} = \JJ_{P^{(i)}}^\circ$, for $i=0,1$.
Again, Assertion (ii) is proved by~\cite[Proposition 7.1]{repVI}.
On the other hand, since $m \big(G^{(0)}, M^{(0)}\big)<m$, our induction hypothesis on~$\big(\Delta|_{Y^{(0)}}, G^{(0)}, P^{(0)}, M^{(0)}\big)$ implies Assertion (i).

\end{proof}

\section{Essential Conjugacy}\label{S:conjugacy}

In Section~\ref{S:type}, we construct a type $(\JJ_P,\lambda_P)=(\JJ_P,\lambda_P,\Delta,\theta)$ for the Bernstein component $\mathfrak{s}=[M,\tau]_G$, constructed from a self-dual semisimple stratum $\Delta=[\Lambda,n,0,\beta]$ and a self-dual semisimple character  $\theta$.
Recall that $L=G^0 \times \prod\limits_{i=1}^l \tilde{G}^i$ is the Levi subgroup of $G$ stabilizing the associated self-dual decomposition  of $\Delta$. 


Suppose that we have another type $(\JJ'_{P'},\lambda'_{P'})=(\JJ'_{P'},\lambda'_{P'},\Delta',\theta')$ for
  $[M,\tau]_G$ associated to a self-dual semisimple stratum $\Delta'=[\Lambda',n',0,\beta']$ and a self-dual semisimple character  $\theta'$. Without loss of generality we can assume that~$\Lambda$ and~$\Lambda'$ have the same~$\mathfrak{o}_F$-period~$e$, because replacing~$\Lambda$ by a multiple of~$\Lambda$ does not change the type. 
 \daniel{For all $-m \leq j \leq m$, both $n_j/e_j$ and $n'_j/e_j$ coincide with the depth of $\tau_j$. Hence~$n_j=n'_j$.}

\begin{prop}\label{P:L-conj}
	Let $(\JJ_P,\lambda_P)$ and $(\JJ'_{P'},\lambda'_{P'})$ be two types for  $(M,\tau)$.
	Then 
\begin{itemize}
\item[(i)] $L=L'$;
	\item[(ii)] \daniel{$\theta$ and~$\theta'$ intertwine in~$M$ and the matching~$\zeta \colon I \to I'$ satisfies:~$W_j \subseteq V^{i}$ if and only if $W_j \subseteq V^{\zeta(i)}$ for all $i\in I$ and~$j$.}
\item[(iii)] there exists a semisimple stratum $\Delta''=[\Lambda',n,0,\beta'']$ satisfying
\begin{itemize}
\item[(a)]  $\beta''$ and $\beta$ have the same minimal polynomial;
\item[(b)] $\beta''$ has the same associated splitting as $\beta'$ ;
\item[(c)] $\theta\times \theta'\in \CCC(\Delta\oplus \Delta'')$ ;
\item[(d)] $\theta'=\tau_{\Delta,\Delta''}(\theta)$;
\item[(e)] there exists $z\in G$ such that $V^{i} = zV^{\zeta(i)}$, $\beta=\leftidx{^z}{\!\beta''}{},$ $M=  \leftidx{{^z}}{\!M}{}$ and $L=\leftidx{{^z}}{\!L}{}$.
\end{itemize}

\end{itemize}
\end{prop}

\begin{proof}

To prove (i) it is equivalent to showing that if $j_1,j_2,\in J$ such that $W_{j_1}+W_{j_2} \subseteq V^{i}$ for some \daniel{non-zero} $i\in I$, then $W_{j_1}+W_{j_2} \subseteq V^{i'}$ for some \daniel{non-zero} $i'\in I'$.
For any $i\in I$, we have  $V^{i}=\bigoplus\limits_{j\in J_i} W_j$ for some subset $J_i$ of  $J$. \daniel{Let~$i\in I$ be non-zero and~$j_1,j_2\in J_i$.}
\daniel{We have that  $\theta^{(i)}=\theta|_{\tilde{\HH}(\Delta^{i})}$ is a simple character obtained by merging the simple characters $\theta_j$ for $j\in J_i$, so~$\theta_{j_1}$ and~$\theta_{j_2}$ are endo-equivalent (write~$\approx$) and not endo-equivalent to~$\theta_{-j_1}$.
On the other hand, we have  $\theta_{j_1} \approx \theta_{j_1}'$ and  $\theta_{j_2} \approx \theta_{j_2}'$, whence $\theta_{j_1}' \approx \theta_{j_2}'$ and not endo-equivalent to~$\theta_{-j_1}'$. 
Then $\theta_{j_1}'$ and $\theta_{j_2}'$ are merged into a common simple character $\theta'^{(i')}$ for some non-zero  $i'\in I'$, and therefore  $W_{j_1}+W_{j_2} \subseteq V^{i'}$, as desired.}

Since $\theta_0$ and  $\theta'_0$ are both contained in  $\tau_0$, we find  $x_0\in \Aut_D(W_0)\cap M$ which intertwines  $\theta_0$ with $\theta'_0$.
Similarly, there are $x_j\in \Aut_D(W_j)$ such that ${x_j}$ intertwines $\theta_j$ with $\theta'_j$, 
for $j= 1,\cdots, m$.
It follows from the Iwahori decomposition of semisimple characters that $x \coloneqq \sum\limits_{j=-m}^m x_j$ with~$x_{-j}:=\sigma_h(x_j)^{-1}$
intertwines $\theta$ with  $\theta'$.
By~\cite[Corollary 5.47]{daniel-1}, the two strata $\Delta$ and $\Delta'$ have the same group level and the same degree.
Then, by~\cite[Theorem 6.5]{daniel-2}, we have a matching from $\theta$ to $\theta'$.
Namely, there is a unique bijection $\zeta \colon I\to I'$, such that there is an element $\tilde{g}\in \prod\limits_{i\in I} \Hom_D(V^i,V^{\zeta(i)})$  satisfying
\begin{itemize}
	\item[(1)] $\tilde{g} V^i=V^{\zeta(i)}$ for all $i\in I$ ;
	\item[(2)] $\tilde{g}$ intertwines $\theta^{(i)}$ with  $\theta'^{(\zeta(i))}$.
\end{itemize}

We next claim that if $W_j \subseteq V^{i}$ for some \daniel{non-zero} $i\in I$, then  $W_j \subseteq V^{\zeta(i)}$;
since  $\zeta$ is $\sigma_h$-equivariant, this further implies that  $W_j \subseteq V^{0}$ implies  $W_j \subseteq V'^{0}$.
Assume that  $W_j \subseteq V^{i'}$ for some $i'\in I'$.
We then have  $\theta'^{(i')} \approx \theta'_j \approx \theta_j \approx \theta^{(i)} \approx \theta'^{(\zeta(i))}$, and hence $i'=\zeta(i)$. 
The claim then follows.

By~\cite[Theorems 6.5 and 6.10]{daniel-2}, 
there exists $\beta''$ such that
 \begin{itemize}
	 \item[(a)]  $\beta''$ and $\beta$ have the same minimal polynomial;
	 \item[(b)] $\beta''$ has the same associated splitting as $\beta'$ ;
	 \item[(c)] put $\Delta''=[\Lambda', n,0,\beta'']$ then  $\theta\otimes \theta'\in \CCC(\Delta\oplus \Delta'')$ ;
	 \item[(d)] $\theta'=\tau_{\Delta,\Delta''}(\theta)$.
\end{itemize}
It follows from (a) and (c)  that $\beta^{(i)}$ and $\beta''^{(\zeta(i))}$ have the same minimal polynomial.
Let $z_i$ be any automorphism of  $V^{i}= V^{\zeta(i)}$ such that $\beta^{(i)} = \leftidx{{^z}}{\beta''^{(\zeta(i))}}{}$.
Then, by~\cite[Corollary 4.14]{daniel-2},  $z \coloneqq \sum_{i\in I} z_i$ conjugates $\beta''$ to $\beta$.
In particular, we have $L= \leftidx{^z}{L}{}$.
Recall that we have $V^{i}=\bigoplus_{j\in J_i} W_j$ and  $\Delta^{i}=\bigoplus_{j\in J_i} \Delta_j$.
Thus, all the $\beta_j$ have the same minimal polynomial as $\beta^{(i)}$; and similarly on the prime side.
Therefore,  $z_i$ can be chosen so that $\beta_j= \leftidx{^{z_i}}{\beta''_j}{}$ for all $j\in J_i$, which implies  that $M=  \leftidx{{^z}}{\!M}{}$.

\end{proof}

\begin{defn}\label{D:endo-equiv}
Two triples $(M,\tau,P\cap L)$ and $(M',\tau',P'\cap L')$  are called \emph{endo-inertially equivalent} if there exists $y\in G$ such that
\begin{itemize}
	\item[(i)] $M={\leftidx{^{y}}{M'}{}}$;
	\item[(ii)]  $\tau= \leftidx{^{y}}{\tau'}{} \otimes \chi$ where $\chi$ is an unramified character of $M$;
	\item[(iii)] $P \cap L={\leftidx{^{y}}{(P'\cap L')}{}}$.
\end{itemize}
\end{defn}

\begin{lem}\label{L:global-translate}
\shu{
Assume $D=F$.
}
Let $(\JJ_P,\lambda_P)$ and $(\JJ'_P,\lambda'_{P'})$ be two types for  $[M,\tau]_G$.
Suppose that $(M,\tau, P\cap L)$ and  $(M',\tau',P'\cap L')$ are endo-inertially equivalent.
Then, for each $i\in I$, there exists $h_i\in \Aut_D(V^{i})$ and an integer $r^{(i)}$ such that 
\[
h_i\Lambda'^{\zeta(i)} = \Lambda^{i} + r^{(i)}.
\] 
\end{lem}

\begin{proof}
	We first replace $(\JJ'_P,\lambda'_{P'})$ by $\left(\leftidx{^{y}}{(\JJ'_{P'})}{},\leftidx{^{y}}{(\lambda'_{P'})}{}\right)$, and still denote by $(\JJ'_P,\lambda'_{P'})$.
We may thus assume $M=M', \tau=\tau'$.
By Proposition~\ref{P:L-conj}, we may further assume that $L=L'$ and hence $P\cap L =P'\cap L$; 
moreover, there is a bijection $\zeta \colon I \to I'$.
For any $i\in I$ and $j\in J_i$, \daniel{there is an integer $\sigma_i(j) \in J_{\zeta(i)}$ such that $W_j = W'_{\sigma_i(j)}$.}
By our construction of $\Delta$ and $\Delta'$, the index sets $J_i$ and $J_{\zeta(i)}$ are consecutive. We consider a non-zero index~$i$.
Since $P\cap L= P' \cap L$, we then have a bijection
\[
	\sigma_i  \colon J_i \longrightarrow J_{\zeta(i)},
\] 
which is increasing in the sense that $\sigma_i(j) < \sigma_i(j')$ for all $j, j'\in J_i$ with $j < j'$.
Therefore, there exists an integer $c_i\in \ZZ$ such that $\sigma_i(j) = j -c_i$ for all $j\in J_i$.

For $j>0$, the simple characters $\tilde{\theta}_j$ and $\tilde{\theta}'_j$ are both contained in $\tilde{\tau}_j$.
Write $\tilde\Theta_j$ and $\tilde\Theta'_{\sigma_i(j)}$ for the ps-characters defined by  $\tilde{\theta}_j$ and $\tilde{\theta}'_{\sigma_i(j)}$, respectively, then
\begin{itemize}
	\item $\tilde\Theta_j$ and $\tilde\Theta'_{\sigma_i(j)}$ are endo-equivalent;
	\item $(F[\beta_j], \Lambda_j)$ and $\left(F[\beta_{\sigma_i(j)}], \Lambda'_{\sigma_i(j)} \right)$ have the same embedding type.
\end{itemize}
By~\cite[Proposition 4.3 and Lemma 3.2]{repVI}, there exists $h_j\in \Aut_D(W_j)$ and $r_j\in\ZZ$ such that 
\[
  h_j \Lambda'_{\sigma_i(j)} = \Lambda_j + r_j.
\]
We denote by \daniel{$Jp(\cdot)$ the set of jump indices of a lattice sequence.} Since
$h_j$ is an automorphism of $W_j$, we have for~$j,j+1\in J_i$
\[
  Jp(\Lambda_j+r_j)=Jp(\Lambda'_{\sigma_i(j)})=Jp(\Lambda'_{\sigma_i(j+1)})-e(\Lambda)/|J_i| 
\]
and
\[
Jp(\Lambda_{j+1}+r_{j+1})-e(\Lambda)/|J_i|=Jp(\Lambda_j+r_{j+1}).
\]
Using a uniformizer of~$F[\beta'_{\sigma_i(j+1)}]$ we can change~$h_j$ such that~$r_j=r_{j+1}$. This proves the non-zero~$i$ case. 

\daniel{For $j=0$, by~\cite[\S 13]{daniel-3}, we find $h_0\in \Aut_D(W_0)$ and $r_0\in \ZZ$ such that
\[
	h_0 \Lambda'_0 = \Lambda_0 +r_0.
\]
The extension of~$h_0$ to the whole of~$V^0$ is immediate. }
This proves the lemma.
\end{proof}

Recall that we have $L \cong G^0 \times \prod\limits_{j=1}^l \tilde{G}^i$ and
we also have $M \cong M^{0}\times \prod\limits_{i=1}^l \tilde{M}^i$, where  $M^{0}= M\cap G^{0}$ and $\tilde{M}^i =M\cap \tilde{G}^{i}$ for $1 \leq i \leq l$.
We then have identifications
\[
P\cap L = P^0 \times \prod\limits_{j=1}^l \tilde{P}^i, \quad\quad\quad  \JJ_P \cap L = \JJ_{P^0} \times \prod\limits_{j=1}^l \tilde{\JJ}_{\tilde{P}^i},
\] 
where $P^0$ is a parabolic subgroup of $G^0$ with Levi component $M^0$, and $\tilde{P}^i$ is a parabolic subgroup of  $\tilde{G}^i$ with Levi component $\tilde{M}^i$.

Let $\kappa^{\sharp} \coloneqq \kappa_Q|_{\JJ_L}$ where $\JJ_L \coloneqq \JJ \cap L$.
Then according to~\cite[Proposition 3.11]{miyauchi}, $\kappa^{\sharp}$ is a standard $\beta$-extension of $\theta|_{\HH^1 \cap L}$.
Moreover, we have $\kappa_{L} \coloneqq\kappa_P|_{\JJ_P \cap L} \cong \kappa^{\sharp}_{P \cap L}$.

We also have $$\lambda_{L} \coloneqq\lambda_P|_{\JJ_P \cap L} \cong \lambda^{(0)}_{P^0} \boxtimes \bigboxtimes_{i=1}^l \tilde{\lambda}^{(i)}_{\tilde{P}^i}.$$

For $1 \leq i \leq l$, we have that $(\tilde \JJ_{\tilde P^i},\tilde\lambda^{(i)}_{\tilde P^i})$ is a
homogeneous semisimple type in the sense of~\cite[\S 7.1]{repVI}, and that $\Delta^{i}$ is a simple stratum.

\begin{defn}
	Two types $(\JJ_P,\lambda_P)$ and $(\JJ'_{P'},\lambda'_{P'})$ for
  $[M,\tau]_G$  
 are called \emph{essentially conjugate} if
	there exist a bijection $\zeta \colon I \to I'$ and an element  $g\in G$ such that
\begin{itemize}
	\item[(i)] $V^i \cong g V^{\zeta(i)}$ for all  $i\in I$;
	\item[(ii)]  $\lambda_L \cong \leftidx{^g}{\!\lambda'_{L'}}{}$.
\end{itemize}
\end{defn}

The main result of this section is the following:
\begin{thm}\label{T:ess-conj}
Assume $D=F$.
Let $(\JJ_P,\lambda_P)$ and $(\JJ'_P,\lambda'_{P'})$ be two types for  $[M,\tau]_G$.
Suppose that $(M,\tau, P\cap L)$ and  $(M',\tau',P'\cap L')$ are endo-inertially equivalent,
then  $(\JJ_P,\lambda_P)$ and $(\JJ'_P,\lambda'_{P'})$ are essentially conjugate.
\end{thm}

\begin{proof}
	We first replace $(\JJ'_P,\lambda'_{P'})$ by $\left(\leftidx{^{y}}{(\JJ'_P)}{},\leftidx{^{y}}{(\lambda'_{P'})}{}\right)$, where $y\in G$ is given by Definition~\ref{D:endo-equiv}, and still denote by $(\JJ'_P,\lambda'_{P'})$.
We may thus assume $M=M', \tau=\tau'$.
By Proposition~\ref{P:L-conj}, we may further assume that $L=L'$ and hence $P\cap L =P'\cap L$; 
moreover, there is a bijection $\zeta \colon I \to I'$.
We replace $\Delta'$ by $\Delta''=[\Lambda', n, 0, \beta'']$.
Let $z\in G$ be an element given by Proposition~\ref{P:L-conj}, we then replace $\Delta''$ by $z\Delta''= [z \Lambda', n, 0, \leftidx{^z}{\beta''}{}=\beta]$.
\daniel{After the adjustments above, we may and do assume that we are working with two self-dual semisimple strata $\Delta$ and $\Delta'$ with $\beta=\beta'$.}

To prove the theorem,  it amounts to showing that there exist
\begin{itemize}
	\item $g^{(0)} \in G$ such that $\JJ_{P^0} \cong \leftidx{^{g^{(0)}}}{\!\JJ'_{P'^0}}{}$ and an isomorphism 
	$\phi^{(0)} \colon  \lambda^{(0)}_{P^0} \cong \leftidx{^{g^{(0)}}}{\!\lambda'^{(0)}_{P'^0}}{}$;
	\item  $g^{(i)} \in G$ such that $\tilde\JJ_{P^i} \cong \leftidx{^{g^{(i)}}}{\tilde\JJ'_{P'^{\zeta(i)}}}{}$ and an isomorphism $\phi^{(i)} \colon  \tilde\lambda^{(i)}_{\tilde P^i} \cong \leftidx{^{g^{(i)}}}{\!\tilde\lambda'^{(\zeta(i))}_{\tilde P'^{\zeta(i)}}}{}$  for all $i\in I_+$.
\end{itemize}
Granting this, we define
 \[
	 g \coloneqq \sum_{i\in I_{0,+}} g^{(i)}, \quad \text{and} \quad	 \phi \coloneqq \bigboxtimes_{i\in  I_{0,+}} \phi^{(i)},
\] 
we then have $\phi \colon \lambda_L \cong \leftidx{^g}{\lambda'_{L}}{}$, and the theorem follows.

Fix $i\in I_{0,+}$.
For the remainder of the proof, we work in the $i$-block.
To lighten the notation, we omit the decoration $\tilde{~}$ on the objects attached in the $i$-block when $i>0$ (and similarly for indices  $j>0$).
For example, when $i>0$, the group $\HH^1(\beta^{(i)}, \Lambda^{i})$ will always mean the group usually denoted $\tilde\HH^1(\beta^{(i)}, \Lambda^{i})$.

As in the proof of Lemma~\ref{L:global-translate}, there exists an increasing bijection $\sigma_i \colon J_i \to J_{\zeta(i)}$; moreover, we may assume that $\beta_j= \beta'_{\sigma_i(j)}$ by our construction of  $\beta''$ in Proposition~\ref{P:L-conj}.
By~\cite[Theorem 5.7.1]{BK-book} for $j>0$ and~\cite[Theorem 11.9]{endo} for $j=0$, there exists $h_j\in \Aut_D(W_j)$  and $r_j\in \ZZ$ such that
\begin{itemize}
	\item[(i)] $\JJ(\beta_j,\Lambda_j)= \leftidx{^{g_j}}{\!\JJ(\beta_j, \Lambda'_{\sigma_i(j)})}{}$;
	\item[(ii)] $\lambda_j \cong \leftidx{^{g_j}}{\lambda'_{\sigma_i(j)}}{}$.
\end{itemize}
It follows from (ii) that $g_j$ conjugates and hence intertwines $\lambda'_{\sigma_i(j)}$ and $\lambda_j$. Therefore  $g_j$ intertwines  $\theta'_{\sigma_i(j)}$ and $\theta_j$, which implies that $g_j\in \JJ^1(\beta_j,\Lambda'_{\sigma_i(j)}) G_{\beta_j} \JJ^1(\beta_j, \Lambda_j)$.
Without loss of generality, we may choose $g_j\in G_{\beta_j}$.
Intersecting both ends in (i) with $G_{\beta_j}$, we have that  $\PP(\Lambda_{j,\beta_j})= \leftidx{^{g_j}}{\!\PP(\Lambda_{j, \beta_j})}{} = \PP \left(g_j\Lambda'_{\sigma_i(j),\beta_j} \right)$.
We then have
\begin{itemize}
	\item[(iii)] $g_j\Lambda'_{\sigma_i(j),\beta_j} = \Lambda_{j,\beta_j}+ r_j~(r_j\in\ZZ), \HH^1(\beta_j,\Lambda_j)= \leftidx{^{g_j}}{\HH^1 \left(\beta_j,\Lambda'_{\sigma_i(j)} \right)}{}$ and $\theta_j  \cong \leftidx{^{g_j}}{\theta'_{\sigma_i(j)}}{}$.
\end{itemize}

We write $E_j = F[\beta_j]$, then $\Lambda_j$ and $\Lambda'_{\sigma_i(j)}$ both have \daniel{the same $\mathfrak{o}_{E_j}$-period $e= e(\Lambda)/e_j$ where $e_j$ denotes the ramification index of $E_j/F$.}
Therefore, imitating the proof of Lemma~\ref{L:global-translate}, we may and do replace $g_j$ by  $\varpi_{E_j}^{k_j}g_j$ for a suitable $k_j\in \ZZ$ and assume that all the $r_j$ are the same integer, for all $j\in J_i$.
Notice that, in the case $D=F$, we have that $E^\times_j$ normalizes $\lambda'_{\sigma_i(j)}$.
We thus have
\begin{itemize}
	\item[(i')] $\JJ(\beta_j,\Lambda_j)= \leftidx{^{g_j}}{\!\JJ(\beta_j, \Lambda'_{\sigma_i(j)})}{}$;
	\item[(ii')] $\lambda_j  \cong \leftidx{^{g_j}}{\lambda'_{\sigma_i(j)}}{}$;
	\item[(iii')] 	$\Lambda_j = g_j\Lambda'_{\sigma_i(j)} + c_i, \HH^1 (\beta_j,\Lambda_j) = \leftidx{^{g_j}}{\!\HH^1 \big(\beta_j,\Lambda'_{\sigma_i(j)} \big)}{}$ and $\theta_j  \cong \leftidx{^{g_j}}{\theta'_{\sigma_i(j)}}{}$;
	
\end{itemize}
Put $g^{(i)}= \sum_{j\in J_i} g_j$.
Replacing $\lambda'^{(\zeta(i))}$ by  $\leftidx{^{g^{(i)}}}{\lambda'^{(\zeta(i))}}{}$, we may and do assume that
$\JJ(\beta_j,\Lambda_j) = \JJ(\beta_j, \Lambda'_{\sigma_i(j)})$ and  $\lambda_j \cong \lambda'_{\sigma_i(j)}$.
Moreover, we assume 
$\Lambda^{i} = \Lambda'^{\zeta(i)}$, $\theta^{(i)} = \theta'^{(\zeta(i))},\HH^1\left(\beta^{(i)},\Lambda^{i}\right) = \HH^1 \left(\beta^{(\zeta(i))}, \Lambda'^{\zeta(i)} \right)$ and $ \JJ\left(\beta^{(i)},\Lambda^{i}\right) = \JJ \left(\beta^{(\zeta(i))}, \Lambda'^{\zeta(i))} \right)$.

Since $(\mathbf{e}_j)_{j\in J_i}$ is properly subordinate to $\Delta^{i}$,
by~\cite[Proposition 5.4 and Corollary 5.11]{stevens-super}, we have 
\begin{align*}
	\JJ \left(\beta^{(i)}, \Lambda^{i}\right) \cap M^{i} \cong \prod\limits_{j\in J_i}\JJ(\beta_j,\Lambda_j) = \prod\limits_{j\in J_{\zeta(i)}}\JJ \left(\beta_j,\Lambda'_{\sigma_i(j)} \right)   \cong	\JJ \left(\beta^{(\zeta(i))}, \Lambda'^{\zeta(i)} \right) \cap M^{\zeta(i)}.
\end{align*}
We have $F$-vector space decompositions
\[
	\mathscr{V}^{i} = \bigotimes_{j\in J_i} \mathscr{V}_j, \quad\quad  \mathscr{V}'^{\zeta(i)} =  \bigotimes_{\sigma_i(j)\in J_{\zeta(i)}} \mathscr{V}'_{\sigma_i(j)},
\] 
where $\mathscr{V}^{i}$ (resp. $\mathscr{V}'^{\zeta(i)}$) is the representation space of $\lambda^{(i)}$ (resp. $\lambda'^{(\zeta(i))}$) for all $i\in I$; 
$\mathscr{V}_j$ (resp. $\mathscr{V}'_{\sigma_i(j)}$) is the representation space of $\lambda_j$ (resp. $\lambda'_{\sigma_i(j)}$) for all $j\in J_i$.  
For any $j\in J_i$, we write $\phi_j \colon \lambda_j \to \lambda'_{\sigma_i(j)}$ for the $\JJ(\beta_j,\Lambda_j)$-isomorphism.
Put $\phi^{(i)} \coloneqq  \bigboxtimes\limits_{j\in J_i} \phi_j$.
Then $\phi^{(i)}$ is a $\left(\JJ \left(\beta^{(i)}, \Lambda^{i}\right) \cap M^{i}\right)$-equivariant automorphism of $\mathscr{V}^{i}$ to  $\mathscr{V}'^{\zeta(i)}$.

It remains to show that $\phi^{(i)} \colon \mathscr{V}^{i} \to \mathscr{V}'^{\zeta(i)}$ is $\JJ\left(\beta^{(i)},\Lambda^{i}\right)$-equivariant.
We first note that $\JJ_{P^i} =\JJ'_{P^{\zeta(i)}}$ since $P\cap L = P'\cap L$.
Denote by $U^i$ (resp.\,$U^i_{-}$) the unipotent radical (resp.\,lower unipotent radical) of $M^i$.
Moreover, because $(\mathbf{e}_j)_{j\in J_i}$ is exactly subordinate to $\Delta^{i}$, we have 
\[
	\JJ_{P^i} = \left(\HH^1(\beta^{(i)}, \Lambda^{i}) \cap U^i_{-} \right) \left(\JJ \left(\beta^{(i)}, \Lambda^{i}\right) \cap M^i \right) \left(\JJ^1(\beta^{(i)}, \Lambda^{i})\cap U^i \right).
\]
We have $\JJ_{P^i}= \JJ'_{P'^{\zeta(i)}}$ since $\JJ^1\left(\beta^{(i)}, \Lambda^{i}\right)= \JJ^1\left(\beta^{(\zeta(i))}, \Lambda^{\zeta(i)}\right)$.
We have that $\kappa^{(i)}_{P^i}$ is trivial on the upper and lower unipotent parts, and $\rho^{(i)}_{P^i}$ is trivial on  $\JJ^1(\beta^{(i)}, \Lambda^{i})$.
Therefore, $\lambda^{(i)}_{P^i}= \kappa^{(i)}_{P^i} \otimes \rho^{(i)}_{P^i}$ is trivial on $\HH^1(\beta^{(i)}, \Lambda^{i}) \cap U^i_{-}$ and $\JJ^1(\beta^{(i)}, \Lambda^{i})\cap U^i$.
For any $x= hmj\in \JJ_{P^i}$ and $v\in V^{i}$, we compute
\[
\phi^{(i)} \left(\lambda^{(i)}_{P^i}(x)v \right)= \phi^{(i)} \left(\lambda^{(i)}_{P^i}(m)v \right) = \lambda^{(i)}_{P^i}(m) (\phi^{(i)}(v)) = \lambda^{(i)}_{P^i}(x) (\phi^{(i)}(v)),
\] 
which finishes the proof.

\end{proof}

\section{Compatibility of $\beta$-extensions}\label{S:compatibility}

\shu{We fix in this section an element  $\beta = \sum_{i\in I}\beta^{(i)}$ in the Lie algebra of $G$}
coming from a (selfdual) semisimple stratum, i.e.~$\beta$ has negative critical exponent and generates a product of fields and is assumed to be selfdual in the case of classical groups. 

In~\cite[Definition 6.8]{daniel-3} (cf.~\cite[Definition 4.5]{stevens-super} in the non-quaternionic case),
the author defines set $\beta\sd\ext_{\Lambda_{\mM}}(\Lambda)$ of \textit{$\beta$-extensions of $\eta_{\Lambda}$ to $\JJ_{\Lambda}$ relative to $\Lambda_{\mM}$}, where $\Lambda_{\mM}$ is a self-dual $\mathfrak{o}_E\sd \mathfrak{o}_D$-lattice sequence such that
\begin{itemize}
\item[(i)] $\tilde{\mathfrak{b}}(\Lambda_{\mM})$ is a maximal $\mathfrak{o}_E \sd \mathfrak{o}_D$ self-dual order in $B=\End_{E\otimes_F D}(V)$;
\item[(ii)]  $\tilde{\mathfrak{b}}(\Lambda_{\mM})\supseteq \tilde{\mathfrak{b}}(\Lambda)$.
\end{itemize}
Equivalently, $\Lambda_{\mM}$ is required such that, in the building $\mathscr{B}(G_\beta)$ equipped with the weak structure, the point $\Lambda_{\mM,\beta}$ is a vertex of the chamber $\mathcal{C}$ containing $\Lambda_\beta$.

\subsection{Generalities on the signature of permutations}

Let $X$ be a finite set and let  $\sigma\in \Sym(X)$ be a permutation of $X$.
We denote by  $\sgn(\sigma) = \sgn(\sigma | X)\in \{\pm 1\}$ the signature of $\sigma$.

\begin{lem}\label{L:sign-product}
Let $X=\prod_{i=1}^r X_i$ be a finite product of finite sets. For each $i$, let $\sigma_i\in\Sym(X_i)$, and define
\[
\sigma=\sigma_1\times\cdots\times\sigma_r\in\Sym(X),
\qquad
\sigma(x_1,\dots,x_r)=(\sigma_1(x_1),\dots,\sigma_r(x_r)).
\]
Then
\[
\sgn(\sigma | X)=\prod_{i=1}^r \sgn(\sigma_i | X_i)^{\prod\limits_{j\neq i}|X_j|}.
\]
In particular, if all $|X_i|$ are odd, then $\sgn(\sigma| X)=\prod_{i=1}^r \sgn(\sigma_i| X_i)$.
\end{lem}
\begin{proof}
We first assume that $r=2$, then $X=X_1\times X_2$ and $\sigma=\sigma_1\times\sigma_2 = (\sigma_1 \times \Id) \circ (\Id \times \sigma_2)$.
Hence $\sgn(\sigma| X) = \sgn(\sigma_1 \times \Id | X) \sgn(\Id\times \sigma_2| X)$.
Since $X = \bigsqcup\limits_{y\in X_2} X_1 \times \{y\}$, and $\sgn(\sigma_1 \times \Id | X_1 \times \{y\}) = \sgn(\sigma_1)$, it follows that
$\sgn(\sigma_1 \times \Id | X) = \left( \sgn(\sigma_1 | X_1) \right) ^{|X_2|}$.
Similarly, we have $\sgn(\Id \times \sigma_2 | X) = \left( \sgn(\sigma_2 | X_2) \right) ^{|X_1|}$.
Therefore, we have 
\[
	\sgn(\sigma | X) = \left( \sgn(\sigma_1 | X_1) \right) ^{|X_2|} \left( \sgn(\sigma_2 | X_2) \right) ^{|X_1|}.
\]
The lemma then follows from an induction on $r$.
\end{proof}

\begin{lem}\label{L:sign-translation}
Let $X$ be a finite set equipped with a right action of a group $\Gamma$.
Suppose that $\sgn(\operatorname{R}_\gamma | X)= 1$ for all  $\gamma\in \Gamma$, where $\operatorname{R}_\gamma \in \Sym(X)$ denotes the permutation of $X$ given by the action of $\gamma$.
Let $Y$ be a finite set and let $\psi \colon Y \to \Gamma$ be any map.
Let $\phi_X \in \Sym(X), \phi_Y\in \Sym(Y)$ and let
\[
	\Phi \colon X \times Y  \longto X \times Y, \qquad  (x,y) \longto (\phi_X(x).\psi(y), \phi_Y(y)).
\] 
Then $\sgn(\Phi) = \sgn(\phi_X)^{|Y|} \sgn(\phi_Y)^{|X|}$.

\end{lem}

\begin{proof}
Define a bijective map
\[
	\tau \colon X \times Y \longto  X \times Y, \qquad (x,y) \longto (x.\psi(\phi_Y^{-1}(y)), y).
\] 
We then have $\Phi = \tau \circ (\phi_X \times \phi_Y)$, and thus  $\sgn(\Phi) = \sgn(\tau) \sgn(\phi_X \times \phi_Y)$.
We claim that $\sgn(\tau)=1$; the lemma then follows immediately from Lemma~\ref{L:sign-product}.
Indeed, for any fixed $y\in Y$, we have that $\tau|_{X \times \{y\}} = \operatorname{R}_{\psi(\phi_Y^{-1}(y))} \times \Id$.
By assumption, we have that $\sgn(\operatorname{R}_{\psi(\phi_Y^{-1}(y))} | X)= 1$, and hence $\sgn(\tau|_{X \times \{y\}}) =1$ by Lemma~\ref{L:sign-product}.
Because $X\times Y = \bigsqcup_{y\in Y} X \times \{y\}$, we have that
$\sgn (\tau) = \prod_{y\in Y} \sgn (\tau|_{X \times \{y\}}) = 1$, which completes the proof of the claim and hence the lemma.

\end{proof}

\begin{lem}\label{L:sign-coset}
Let $\Sigma$ be a group, and let $K \subseteq H \subseteq \Sigma$ be subgroups such that the left cosets $K \backslash H$ and $H \backslash \Sigma$ are finite sets.
Let $\phi\in \Aut(\Sigma)$ such that  $\phi(K) = K$ and $\phi(H) = H$.
Assume that, for all $h\in 
H$, the right translation $$\operatorname{R}_h \colon K \backslash H \to K \backslash H, \qquad Hx \longmapsto Hxh$$ satisfies $\sgn(\operatorname{R}_h | K \backslash H) =1$.
Denote by  $\phi_{K \backslash H}\in \Sym(K \backslash H), \phi_{H \backslash \Sigma}\in \Sym(H \backslash \Sigma)$ and  $\phi_{K \backslash \Sigma}\in \Sym(K \backslash\Sigma)$ the permutations induced by  $\phi$, then
\[
	\sgn(\phi_{K \backslash \Sigma} | K \backslash \Sigma )= \sgn(\phi_{K \backslash H} | K \backslash H)^{|H \backslash \Sigma|} \sgn(\phi_{K \backslash \Sigma} | H \backslash \Sigma)^{|K  \backslash H|}.
\] 

\end{lem}

\begin{proof}
Choose a set-theoretic section $s \colon H \backslash \Sigma \to \Sigma$ of the projection $\pi \colon \Sigma \to H \backslash \Sigma$, then $\pi\circ s$ is the identity on $H \backslash \Sigma$.
We then define
\begin{align*}
	\Theta_s \colon K \backslash H \times H \backslash \Sigma  \longto K \backslash \Sigma, \qquad	(Kh ,X)  \longmapsto  Khs(X).
\end{align*}
It is obvious that $\Theta_s$ is well-defined.
We next show that $\Theta_s$ is surjective and choose any $K\sigma \in K \backslash \Sigma$, 
Put $X \coloneqq H\sigma\in H \backslash \Sigma$.
Then $X = \pi(s(X)) = \pi(\sigma)$, which implies that $\sigma = hs(X)$ for some $h\in H$, and hence
$K\sigma = Khs(X) = \Theta_s(Kh, X)$.
For the injectivity, we suppose that  $\Theta_s(Kh_1,X_1) = \Theta_s(Kh_2,X_2)$.
We then have  $Kh_1s(X_1)=Kh_2s(X_2)$, multiplying $H$ from the left implies that  $Hs(X_1)=Hs(X_2)$, or equivalently, $\pi(s(X_1)) = \pi(s(X_2))$ and hence  $X_1=X_2$.
We thus have $Kh_1 = Kh_2$, and then  $(Kh_1,X_1) = (Kh_2,X_2)$.

We now define $\Psi \coloneqq \Theta_s^{-1} \circ \phi_{K \backslash \Sigma} \circ \Theta_s \in \Sym(K \backslash H \times H \backslash \Sigma)$.
We then have $\sgn(\Psi) = \sgn(\phi_{K \backslash \Sigma})$.
Thus, it suffices to show
\begin{equation}\label{E:sign-translation}
\sgn(\Psi) = \sgn(\phi_{K,H} | K \backslash H)^{|H \backslash \Sigma|} \sgn(\phi_{K,\Sigma} | H \backslash \Sigma)^{|H \backslash K|}.
\end{equation}
For any $X\in H \backslash \Sigma$, we have $\pi(\phi(s(X))) = \phi_{H \backslash \Sigma}(\pi(s(X)))=\phi_{H \backslash \Sigma}(X)$, and hence $\phi(s(X)) = \delta(X) s(\phi_{H \backslash \Sigma}(X))$ for a unique $\delta(X)\in H$.
We thus obtain a map $\delta \colon H \backslash \Sigma \to H$.
For any $(Kh,X) \in K \backslash H \times H \backslash \Sigma$, we compute
\begin{align*}
	\Psi(Kh,X) & = \Theta_s^{-1} \circ \phi_{K \backslash \Sigma} (Khs(X)) \\
		   &= \Theta_s^{-1}(K\phi(h)\phi(s(X)))\\
		   &= \Theta_s^{-1}(\phi_{K \backslash H}(K h)\delta(X) s(\phi_{H \backslash \Sigma}(X)))\\
		   &= (K \phi(h)\delta(X), \phi_{H \backslash \Sigma}(X))   \\
		   &= (\phi_{K \backslash H}(Kh). \delta(X), \phi_{H  \backslash \Sigma}(X)).
\end{align*}
Then~\eqref{E:sign-translation} follows from Lemma~\ref{L:sign-translation}, which completes the proof.

\end{proof}

For our later applications, the most relevant case is that $H$ is a  pro-$p$ group and $K \backslash H$ is finite. In particular, $\Sym( K \backslash H)$ is finite.
Then the assignment $h \mapsto \operatorname{R}_{h^{-1}}$ gives a group homomorphism $\rho \colon H \to \Sym( K \backslash H)$,
whose image $\Img(\rho)$ is then isomorphic with a finite quotient of $H$, whence a finite $p$-group.
For any $h\in H$, the order of $\operatorname{R}_h \in \Img(\rho)$ is then odd, we thus have
$\sgn(\operatorname{R}_h | K \backslash H) =1$.

\begin{cor}\label{C:sign-extension}
Let $1\to A \to C \to B \to 1$
be a short exact sequence of finite groups such that $|C|$ is odd.
Let $\varphi\in \Aut(C)$ satisfy $\varphi(A)=A$.
Let $\varphi_A:=\varphi|_A$ and let $\varphi_B\in \Aut(B)$ be the induced automorphism of
$B$.
Then
\[
\sgn(\varphi| C)=\sgn(\varphi_A| A)\cdot \sgn(\varphi_B| B).
\]
\end{cor}

\begin{proof}
For any $a\in A$, the order  $m$ of  $a$ is odd since $|A|$ is odd.
It follows that $\sgn(\operatorname{R}_a) = (-1)^{(m-1)\frac{|A|}{m}} =1$.
Let $K = \{1\}, H= A$ and $\Sigma = C$.
Then $K \backslash H =A, H \backslash \Sigma \cong B$ and $K \backslash \Sigma = C$, and the corollary follows from Lemma~\ref{L:sign-coset}.
\end{proof}

For a finite field $k= \FF_q$, we denote by $\left( \frac{\cdot}{k} \right)$ the unique non-trivial quadratic character of $k^\times$.

\begin{lem}\label{L:sign-linear-det}
Let $k = \FF_q$.
Let $W$ be a finite-dimensional $k$-vector space and $g\in \GL(W)$.
Viewing $g$ as a permutation of the underlying set $W$, we have
\[
\sgn(g| W)=\left( \frac{\det(g)}{k} \right).
\]
\end{lem}

\begin{proof}
Choose a $k$-basis for $W$, so $W \cong k^d$ and $g$ corresponds to a matrix in $\GL_d(k)$. 
The group homomorphism
\[
\epsilon \colon \GL_d(k)\to\{\pm 1\},\qquad g\mapsto \sgn(g| k^d)
\]
then factors through the abelianization $\GL_d(k)/\SL_d(k)$ of $\GL_d(k)$, which is isomorphic to $k^\times$ via the determinant map.
Hence, $\epsilon = \eta \circ \det$ for some quadratic character  $\eta \colon k^\times \to \{\pm 1\}$.
It then suffices to show that $\eta$ is non-trivial, hence the unique quadratic character $\left( \frac{\cdot}{k} \right)$.
Notice that $\eta(a)= \epsilon(\Diag(a,1,\cdots,1))$ for all $a \in k^\times$.

Since $|k|=q$ is odd, we have $\epsilon(\Diag(a,1,\cdots,1)) = \sgn(\mu_a | k)$ by Lemma~\ref{L:sign-product}, where $\mu_a$ denotes the multiplication by  $a$.
Fix $a \in k^\times \setminus (k^\times)^2$.
We claim that $\sgn(\mu_a | k) =-1$.
Indeed, let  $m$ be the order of $a$ in $k^\times$, then $\sgn(\mu_a | k)= (-1)^{(\frac{q-1}{m})(m-1)}$.
Because $m$ divides $q-1$ and  $m-1$ is odd, it suffices to show that $(q-1)/m$ is odd.
Since $a^{\frac{q-1}{2}}= -1$, we have that $m$ does not divide  $(q-1)/2$, which implies that  $(q-1)/m$ is odd, and we are done.

\end{proof}

\subsection{Extensions of the signature character}\label{S:sign-extension}

Let $k= \FF_q$ and let $\ell$ be the quadratic extension of $k$.
We denote by $\operatorname{N}^1(\ell/k) = \{z\in \ell | \operatorname{N}_{\ell/k} (z) =1\}$ the norm-$1$ elements in $\ell$.

In this subsection, $\mathbf{H}$ is one of $\GL_r(k), \Sp_{2r}(k), \SO(r,r_{\an})(k), \UU_{2r}(\ell/k)$ or $\UU_{2r+1}(\ell/k)$.
In the notation $\SO(r,r_{\an})(k)$, $r$ denotes the Witt index and $r_{\an}\in \{0,1,2\}$ denotes the dimension of the anisotropic part of the underlying bilinear form.

For any $x\in \mathbf{H}$, we denote by $\sigma_x$ the conjugation action of  $\langle x \rangle$ on $\mathbf{H}$.
Let $\mathbf{B}$ a Borel subgroup of $\mathbf{H}$ with unipotent radical $\mathbf{U}$.
For any $b\in \mathbf{B}$, we denote by $\mathbf{U}/\langle b \rangle$ the set of orbits of the conjugation action $\sigma_b$ of  $\langle b \rangle$ on $\mathbf{U}$.
We have $\sgn(\sigma_b | \mathbf{U}) = (-1)^ {|\mathbf{U}/\langle b \rangle|-1}$.

The main result of this subsection is:

\begin{prop}\label{P:sign-extension}
There exists a unique quadratic character $\chi \colon \mathbf{H} \to \{\pm 1\}$ such that for any Borel subgroup $\mathbf{B}$ with unipotent radical $\mathbf{U}$, we have $\chi(b) = \sgn(\sigma_b | \mathbf{U})$ for all $b\in \mathbf{B}$.
\end{prop}

\begin{proof}

We first prove the uniqueness.
Suppose that we have two quadratic characters $\chi_1$ and $\chi_2$ of $\mathbf{H}$ satisfying the desire property, then
$\delta \coloneqq \chi_1 \chi_2^{-1}$ is a quadratic character of $\mathbf{H}$ such that $\delta|_{\mathbf{B}}$ is trivial for all Borel subgroup $\mathbf{B}$ of $\mathbf{H}$.
Fix a Borel $\mathbf{B}$ and put $\mathbf{K} \coloneqq \langle x \mathbf{B}x^{-1} | x\in \mathbf{H} \rangle$ the subgroup of $\mathbf{H}$ generated by all conjugates of $\mathbf{B}$.
Then $\mathscr{N}_{\mathbf{H}}(\mathbf{K}) = \mathbf{H}$ and $\delta$ is trivial on $\mathbf{K}$.
On the other hand, since $\mathbf{K}$ contains $\mathbf{B}$,  it is then a parabolic subgroup of $\mathbf{H}$, which implies that $\mathscr{N}_{\mathbf{H}}(\mathbf{K}) = \mathbf{K}$.
It follows that $\mathbf{H} = \mathbf{K}$ and hence $\delta$ is trivial on  $\mathbf{H}$.

We next prove the existence.
By Propositions~\ref{P:sign-gl},
\ref{P:sign-sp},
\ref{P:sign-orthogonal-0},
\ref{P:sign-orthogonal-1},
\ref{P:sign-orthogonal-2},
\ref{P:sign-unitary-odd},
and \ref{P:sign-unitary-even} below, there  exists a quadratic character
\[
	\chi:\mathbf{H}\longrightarrow \{\pm 1\}
\]
such that
\[
	\chi(t)=\sgn(\sigma_t| \mathbf{U})\qquad\text{for all }t\in \mathbf{T},
\]
for a Borel pair $(\mathbf{B},  \mathbf{T})$, where $\mathbf{B}$ is a  Borel subgroup $\mathbf{B}\subseteq \mathbf{H}$ with unipotent radical $\mathbf{U}$, and $\mathbf{T}$ is a maximal $k$-torus such that $\mathbf{B}=\mathbf{T}\mathbf{U}$.

Let $b\in \mathbf{B}$. Write $b=tu$ with $t\in \mathbf{T}$ and $u\in \mathbf{U}$. 
The conjugation action satisfies $\sigma_{tu}=\sigma_t\circ\sigma_u$ on the set \(\mathbf{U}\), hence
\[
\sgn(\sigma_b| \mathbf{U})=\sgn(\sigma_t| \mathbf{U})\,\sgn(\sigma_u| \mathbf{U}).
\]
Since $\mathbf{U}$ is a finite $p$-group with $p$ odd, $|\mathbf{U}|$ is odd and every $\langle u\rangle$-orbit in $\mathbf{U}$ has odd size; therefore $\sgn(\sigma_u| \mathbf{U})=1$. 
Consequently,
\[
\sgn(\sigma_b| \mathbf{U})=\sgn(\sigma_t| \mathbf{U}).
\]
On the other hand, since $\chi$ is quadratic and u has odd order, we have $\chi(u)=1$. 
Hence
\[
\chi(b)=\chi(tu)=\chi(t)=\sgn(\sigma_t| \mathbf{U})=\sgn(\sigma_b| \mathbf{U}),
\]
so the desired identity holds for all \(b\in \mathbf{B}\).

Now let $\mathbf{B}'\subseteq \mathbf{H}$ be any other Borel subgroup with unipotent radical $\mathbf{U}'$. 
As all Borels of $\mathbf{H}$ are conjugate, there exists $h\in \mathbf{H}$ such that $\mathbf{B}'=h\mathbf{B} h^{-1}$ and $\mathbf{U}'=h\mathbf{U} h^{-1}$. 
For any $b'\in \mathbf{B}'$, write $b'=hbh^{-1}$ with $b\in \mathbf{B}$. 
Then conjugation by $h$ induces a bijection $\mathbf{U}\to \mathbf{U}', x\mapsto hxh^{-1}$, which intertwines the actions of $b$ on $\mathbf{U}$ and $b'$ on $\mathbf{U}'$.
We thus have
\[
\sgn(\sigma_{b'}| \mathbf{U}')=\sgn(\sigma_b| \mathbf{U}).
\]
Since $\chi$ is a character, it is invariant under conjugation, so $\chi(b')=\chi(b)$. 
Therefore,
\[
\chi(b')=\chi(b)=\sgn(\sigma_b| \mathbf{U})=\sgn(\sigma_{b'}| \mathbf{U}'),
\]
which proves the proposition.
\end{proof}

Suppose that $\mathbf{L}$ is a Levi quotient of $\mathbf{H}$, then $\mathbf{L}$ is a product of certain groups considered in this section (i.e.\,some general linear groups and a classical group over $k$).
By Proposition~\ref{P:sign-extension}, there exists a quadratic character $\chi_{\mathbf{L}} \colon \mathbf{L} \to \{\pm 1\}$ such that 
for any Borel subgroup $\mathbf{B}_{\mathbf{L}}$ of $\mathbf{L}$ with unipotent radical $\mathbf{U}_{\mathbf{L}}$, we have 
\[
	\chi_{\mathbf{L}}(b) =\sgn(\sigma_b | \mathbf{U}_{\mathbf{L}}), \quad \forall b\in \mathbf{B}_{\mathbf{L}}.
\]
\begin{cor}\label{C:sign-parabolic}
Let $\mathbf{Q}$ be a parabolic subgroup of $\mathbf{H}$ with unipotent radical $\mathbf{N}$.
Let $\pi \colon \mathbf{Q} \to \mathbf{L} \coloneqq \mathbf{Q}/\mathbf{N}$ be the projection.
Then
\[
	\chi(x) = \sgn(\sigma_x | \mathbf{N}) \cdot \chi_{\mathbf{L}}(\pi(x)),  \quad  \forall x \in \mathbf{P},
\] 
where $\chi= \chi_{\mathbf{H}}$ is the quadratic character of $\mathbf{H}$ given by Proposition~\ref{P:sign-extension}.
\end{cor}

\begin{proof}
Choose an arbitrary Borel subgroup $\mathbf{B}_{\mathbf{L}}$ of $\mathbf{L}$ with unipotent radical $\mathbf{U}_{\mathbf{L}}$.
Put $\mathbf{B} \coloneqq \pi^{-1}(\mathbf{B}_{\mathbf{L}}) \subseteq \mathbf{Q}$.
Then $\mathbf{B}$ is a Borel subgroup of $\mathbf{H}$ and we have  an exact sequence
\[
	1 \longto \mathbf{N} \longto R_u(\mathbf{B}) \longto \mathbf{U}_{\mathbf{L}} \longto 1.
\]
It follows from Corollary~\ref{C:sign-extension} that $\sgn(\sigma_b |  R_u(\mathbf{B})) = \sgn (\sigma_b | \mathbf{N}) \sgn(\sigma_{\pi(b)} | \mathbf{U}_{\mathbf{L}})$ for all $b\in \mathbf{B}$, and thus
\begin{equation}\label{E:sign-parabolic}
	\chi(b) = \sgn(\sigma_b | \mathbf{N}) \chi_{\mathbf{L}}(\pi(b)), \quad \forall b\in \mathbf{B}.
\end{equation}

We define $\delta \colon \mathbf{Q} \to \{\pm 1\}$ via $\delta(x) \coloneqq \chi(x) \sgn(\sigma_x | \mathbf{N})^{-1} \chi_{\mathbf{L}}^{-1}(\pi(x))$ for all $x\in \mathbf{Q}$.
Since $\delta$ is quadratic and $\mathbf{N}$ is a $p$-group, $\delta|_{\mathbf{N}}$ is trivial and $\delta$ factors through $\bar\delta \colon \mathbf{L} \to \{\pm 1\}$.
For any $l\in \mathbf{B}_{\mathbf{L}}$, there exits $b\in \mathbf{B}$ with $\pi(b)=l$.
By~\eqref{E:sign-parabolic} we have $\delta(b)=1$, which implies 
 \[
\bar\delta(l) = \bar\delta(\pi(b)) = \delta(b) = 1.
\] 
It follows that $\bar\delta$ is trivial on $\mathbf{B}_{\mathbf{L}}$ and hence on $\mathbf{L}$, noting that $\mathbf{L}$ is generated by conjugates on $\mathbf{B}_{\mathbf{L}}$.
Thus, $\delta$ is trivial on $\mathbf{Q}$ and the corollary follows.
\end{proof}

We next address the existence of a quadratic character $\chi$ extending the signature character arising from the conjugation action of $\mathbf{T}$ on $\mathbf{U}$, 
for an arbitrary but fixed Borel pair $(\mathbf{B},\mathbf{T})$ in $\mathbf{H}$. 
We let $\mathbf{S} \subseteq \mathbf{T}$ be the maximal $k$-split subtorus. 
By Lemma~\ref{L:sign-roots} below, the computation of the signature character on $\mathbf{T}$ to computing the determinant of the adjoint action of $\mathbf{T}$ on each $\mathbf{S}$-weight space. 
We treat each classical group $\mathbf{H}$ case by case.
In most cases the resulting signature character on $\mathbf{T}$ is trivial, and we may take $\chi$ to be the trivial character of $\mathbf{H}$.
In the remaining three cases, $\GL_r(k),\SO(r,1)(k)$ and  $\UU_{2r}(k)$, the signature character on $\mathbf{T}$ is non-trivial; we identify it explicitly and extend it to a quadratic character of $\mathbf{H}$.

\subsubsection{}

Let $\mathbf{B}$ be a Borel subgroup of $\mathbf{H}$ with unipotent radical $\mathbf{U}$.
Let $\mathbf{T}$ be a maximal $k$-torus such that $\mathbf{B}=\mathbf{TU}$ and let $\mathbf{S}$ be the maximal $k$-split subtorus of $\mathbf{T}$.
Let $\Phi = \Phi(\mathbf{H}, \mathbf{S})$ be the root system.
For any $\alpha\in \Phi$,  we put $\mathfrak{h}_\alpha \coloneqq \{x\in \Lie(\mathbf{H}) \mid \Ad(s)(x) = \alpha(s)x,\forall s\in \mathbf{S}\}$, the $\mathbf{S}$-weight space for $\alpha$, whose vectors are called $\mathbf{S}$-weights for  $\alpha$.

We denote by $\Phi^+$ the associated positive roots.
Following~\cite[Proposition 3.3.6]{pseudo}, we denote by $\mathbf{U}_\alpha$ the root subgroup of $\alpha$ such that $\Lie(\mathbf{U}_\alpha)$ is spanned by the $\mathbf{S}$-weights for all $m\alpha\in \Phi^+ ~ (m\in\ZZ_{>0})$.
If $2\alpha\notin \Phi^+$, then  $\mathbf{U}_\alpha$ is a vector group and hence $\mathbf{U}_\alpha \cong \Lie(\mathbf{U}_\alpha) =\mathfrak{h}_\alpha$.
If $2\alpha\in \Phi^+$, we have $\Lie(\mathbf{U}_\alpha) = \mathfrak{h}_\alpha \oplus \mathfrak{h}_{2\alpha}$ and $\Lie(\mathbf{U}_{2\alpha})= \mathfrak{h}_{2\alpha}$; moreover, $\mathbf{U}_\alpha / \mathbf{U}_{2\alpha}$ is a vector group and hence $\mathbf{U}_\alpha / \mathbf{U}_{2\alpha} \cong \Lie(\mathbf{U}_\alpha / \mathbf{U}_{2\alpha}) = \mathfrak{h}_\alpha$
(cf.\,Lemma 3.3.8 in loc.\,cit.).

Note that $\mathbf{T}$ acts on $\mathfrak{h}_\alpha$ and we denote by $\Ad(t | \mathfrak{h}_\alpha)$ the adjoint action of  $t\in \mathbf{T}$ on $\mathfrak{h}_\alpha$.

\begin{lem}\label{L:sign-roots}
For any Borel $\mathbf{B}$ containing $\mathbf{T}$, let $\mathbf{U}$ be the unipotent radical.
We have
\[
	\sgn(\sigma_t | \mathbf{U}) =  \prod_{\alpha\in \Phi^+} \left(\frac{\det(\Ad(t | \mathfrak{h}_\alpha))}{k} \right)  \qquad \text{for all} \quad t\in \mathbf{T},
\] 
where $\Phi^+ \subseteq \Phi$ is a set of positive roots.
\end{lem}

\begin{proof}	
Fix $t\in \mathbf{T}$.
For any $\alpha\in \Phi^+$, we denote by  $\operatorname{ht}(\alpha)$ the height of $\alpha$ and let  $m \coloneqq \max \{\operatorname{ht}(\alpha) \mid \alpha\in \Phi^+\}$.
For $1\leq i \leq m$, we let  $\mathbf{U}_i= \langle \mathbf{U}_\alpha  \mid \operatorname{ht}(\alpha) \geq i \rangle$.
We then have a filtration
\[
	1 \eqqcolon \mathbf{U}_{m+1} \subseteq \mathbf{U}_m \subseteq  \cdots \subseteq \mathbf{U}_1 =\mathbf{U}.
\]
Put $\gr(\mathbf{U}) \coloneqq \prod\limits_{i=1}^m \mathbf{U}_i/ \mathbf{U}_{i+1}$.
We have an isomorphism of $\mathbf{T}$-sets
$$\gr(\mathbf{U}) \cong \prod_{\alpha\in \Phi^+} \mathfrak{h}_\alpha.$$
Since each $|\mathfrak{h}_\alpha|$ is odd, it follows from Lemma~\ref{L:sign-product} and Lemma~\ref{L:sign-linear-det} that 
\[
\sgn (\sigma_t | \gr(\mathbf{U})) = \prod_{\alpha\in \Phi^+} \sgn( \Ad(t| \mathfrak{h}_\alpha))= \prod_{\alpha\in \Phi^+} \left( \frac{\det(\Ad(t | \mathfrak{h}_\alpha))}{k} \right).
\] 
Therefore, it suffices to show $\sgn( \sigma_t | \mathbf{U}) = \sgn(\sigma_t | \gr(\mathbf{U}))$.
By applying Corollary~\ref{C:sign-extension} to the exact sequence $1 \to \mathbf{U}_2 \to \mathbf{U} \to \mathbf{U}/\mathbf{U}_2 \to 1$ we obtain $\sgn( \sigma_t | \mathbf{U}) = \sgn(\sigma_t | \mathbf{U}_2) \sgn(\sigma_t | \mathbf{U}/\mathbf{U}_2)$.
Iterating this process then yields
\[
	\sgn( \sigma_t | \mathbf{U}) = \prod_{i=1}^m \sgn(\sigma_t | \mathbf{U}_i/\mathbf{U}_{i+1}) =\sgn(\sigma_t | \gr(\mathbf{U})),
\] 
which completes the proof.
\end{proof}

\subsubsection{$\mathbf{H}=\GL_r(k)$}

Let $\mathbf{H}=\GL_r(k)$.
Let
\[
\mathbf{T}=\{\Diag(t_1,\dots,t_r) \mid t_1,\dots,t_r\in k^\times\},
\]
and let $\mathbf{B}$ be the Borel subgroup consisting of upper triangular matrices.
Its unipotent radical is the subgroup $\mathbf{U}$ of upper unitriangular matrices.

For $1\leq i\leq r$, we denote by $\epsilon_i$ the character of $\mathbf{T}$ defined by
\[
\epsilon_i(t)=t_i,\qquad
t=\Diag(t_1,\dots,t_r)\in \mathbf{T}.
\]
Then the set of positive roots is
\[
\Phi^+=\{\epsilon_i-\epsilon_j \mid 1\leq i<j\leq r\}.
\]

\begin{prop}\label{P:sign-gl}
Let $\mathbf{H}=\GL_r(k)$.
Define a character $\chi \colon \mathbf{H}\to \{\pm 1\}$ by
\begin{align*}
\chi(g)= \left\{
	\begin{array}{ll}
1, &   r \text{ is odd},\\
\left(\dfrac{\det(g)}{k}\right), &  r \text{ is even}.
\end{array} \right.
\end{align*}
Then $\chi$ is a quadratic character of $\mathbf{H}$ and
\[
\chi(t)=\sgn(\sigma_t| \mathbf{U})
\qquad\text{for all } t\in \mathbf{T}.
\]
\end{prop}

\begin{proof}
	For a positive root $\alpha=\epsilon_i-\epsilon_j$ with $1\leq i<j\leq r$, we have $\Ad(t| \mathfrak{h}_\alpha)=\mu_{t_it_j^{-1}}$, the scalar multiplication by $t_it_j^{-1}$, and therefore

\[
\left(\frac{\det(\Ad(t| \mathfrak{h}_\alpha))}{k}\right)
=
\left(\frac{t_it_j^{-1}}{k}\right).
\]
By Lemma~\ref{L:sign-roots}, we obtain
\[
\sgn(\sigma_t| \mathbf{U}) =
\prod_{1\leq i<j\leq r}\left(\frac{t_it_j^{-1}}{k}\right)
=\prod_{i=1}^r \left(\frac{t_i}{k}\right)^{\,r+1-2i}
=\prod_{i=1}^r \left(\frac{t_i}{k}\right)^{\,r+1}
=\left(\frac{\det(t)}{k}\right)^{\,r+1}.
\]
The proposition then follows.

\end{proof}

\subsubsection{$\mathbf{H}=\Sp_{2r}(k)$}
Let $J_r$ be the  $r$ by $r$ antidiagonal matrix with $1$ on the antidiagonal and let $J= \begin{psmallmatrix} 0 & J_r\\ - J_r & 0 \end{psmallmatrix}$. 
Let $\Sp_{2r} = \{g\in \GL_{2r}(k)  \mid g^{\operatorname{T}}J g = J\}$.
Let $$\mathbf{T}= \{\Diag(t_1,\cdots, t_r, t_r^{-1},\cdots, t_1^{-1})  \mid t_1,\cdots,t_r\in k^\times\},$$
and let $\mathbf{B}$ be the Borel subgroup consisting of upper triangle symplectic matrices.
For $1\leq i \leq r$, we denote by  $\epsilon_i$ the character of  $\mathbf{T}$ defined via $\epsilon_i(t) \coloneqq  t_i$ for  $t= \Diag(t_1,\cdots, t_r, t_r^{-1},\cdots, t_1^{-1})$.
Then $\Phi^+ = \{\epsilon_i \pm \epsilon_j | 1 \leq i < j \leq r\} \cup \{2 \epsilon_i | 1\leq i \leq j\}$.

\begin{prop}\label{P:sign-sp}
Let $\mathbf{H}=\Sp_{2r}(k)$.
Then $\sgn(\sigma_t | \mathbf{U})=1$ for all $t\in \mathbf{T}$.
\end{prop}
\begin{proof}
We first compute
\begin{align*}
	\Ad(t | \mathfrak{h}_\alpha) = \left\{
		\begin{array}{lll}
			\mu_{t_it_j},   &\alpha = \epsilon_i +\epsilon_j \\
			\mu_{t_it_j^{-1}},   & \alpha = \epsilon_i -\epsilon_j\\
			\mu_{t_i^2},  & \alpha = 2\epsilon_i
	\end{array} \right.
\end{align*}
We thus have
\begin{align*}
	 \left( \frac{\det(\Ad(t | \mathfrak{h}_\alpha))}{k} \right) = \left\{
		\begin{array}{lll}
			\left(\frac{t_it_j}{k} \right),   &\alpha = \epsilon_i +\epsilon_j \\
			\left(\frac{t_it_j^{-1}}{k} \right),   & \alpha = \epsilon_i -\epsilon_j\\
			\left(\frac{t_i^2}{k} \right)=1,  & \alpha = 2\epsilon_i
	\end{array} \right.
\end{align*}
It follows from Lemma~\ref{L:sign-roots} that $\sgn(\sigma_t | \mathbf{U})=\prod\limits_{1\leq i< j \leq r}\left(\frac{t_it_j}{k} \right)\left(\frac{t_it_j^{-1}}{k} \right) = 1$.
\end{proof}

\subsubsection{$\mathbf{H}=\SO(r,0)(k)$}
Let $J_{2r}= \begin{psmallmatrix} 0 & J_r\\ J_r & 0 \end{psmallmatrix}$.
Let $\SO(r,0)(k) = \{g\in \SL_{2r}(k)  \mid g^{\operatorname{T}}J_{2r} g = J_{2r}\}$.
Let $$\mathbf{T}= \{\Diag(t_1,\cdots, t_r, t_r^{-1},\cdots, t_1^{-1})  \mid t_1,\cdots,t_r\in k^\times\},$$
and let $\mathbf{B}$ be the Borel subgroup consisting of upper triangle orthogonal matrices.
For $1\leq i \leq r$, we denote by  $\epsilon_i$ the character of  $\mathbf{T}$ defined via $\epsilon_i(t) \coloneqq  t_i$ for  $t= \Diag(t_1,\cdots, t_r, t_r^{-1},\cdots, t_1^{-1})$.
Then $\Phi^+ = \{\epsilon_i \pm \epsilon_j  \mid 1 \leq i < j \leq r\}$.

\begin{prop}\label{P:sign-orthogonal-0}
Let $\mathbf{H}=\SO(r,0)(k)$.
Then $\sgn(\sigma_t | \mathbf{U})=1$ for all $t\in \mathbf{T}$.
\end{prop}
\begin{proof}
We first compute
\begin{align*}
	\Ad(t | \mathfrak{h}_\alpha) = \left\{
		\begin{array}{ll}
			\mu_{t_it_j},   &\alpha = \epsilon_i +\epsilon_j \\
			\mu_{t_it_j^{-1}},   & \alpha = \epsilon_i -\epsilon_j
	\end{array} \right.
\end{align*}
We thus have
\begin{align*}
 \left( \frac{\det(\Ad(t | \mathfrak{h}_\alpha))}{k} \right) = \left\{
		\begin{array}{lll}
			\left(\frac{t_it_j}{k} \right),   &\alpha = \epsilon_i +\epsilon_j \\
			\left(\frac{t_it_j^{-1}}{k} \right),   & \alpha = \epsilon_i -\epsilon_j
	\end{array} \right.
\end{align*}
It follows from Lemma~\ref{L:sign-roots} that
$$\sgn(\sigma_t | \mathbf{U})=\prod\limits_{1\leq i< j \leq r}\left(\frac{t_it_j}{k} \right)\left(\frac{t_it_j^{-1}}{k} \right)= 1.$$
\end{proof}

\subsubsection{$\mathbf{H}=\SO(r,1)(k)$}
Let $J_{2r+1}= \begin{psmallmatrix} 0 & 0&  J_r\\ 0&1 &0 \\ J_r & 0 &0\end{psmallmatrix}$.
Let $\SO(r,1)(k) = \{g\in \SL_{2r+1}(k)  \mid g^{\operatorname{T}}J_{2r+1} g = J_{2r+1}\}$.
Let $$\mathbf{T}= \{\Diag(t_1,\cdots, t_r, 1,  t_r^{-1},\cdots, t_1^{-1})  \mid t_1,\cdots,t_r\in k^\times\},$$
and let $\mathbf{B}$ be the Borel subgroup consisting of upper triangle orthogonal matrices.
For $1\leq i \leq r$, we denote by  $\epsilon_i$ the character of  $\mathbf{T}$ defined via $\epsilon_i(t) \coloneqq  t_i$ for  $t= \Diag(t_1,\cdots, t_r, 1, t_r^{-1},\cdots, t_1^{-1})$.
Then $\Phi^+ = \{\epsilon_i \pm \epsilon_j  \mid 1 \leq i < j \leq r\} \cup \{\epsilon_i  \mid 1 \leq i \leq r\}$.

\begin{lem}\label{L:sign-orthogonal-1}
Let $\mathbf{H}=\SO(r,1)(k)$.
Then $$\sgn(\sigma_t | \mathbf{U})=\left(\frac{\prod\limits_{1 \leq i \leq r} t_i}{k} \right)$$ 
for all $t= \Diag(t_1,\cdots, t_r, 1, t_r^{-1},\cdots, t_1^{-1})\in \mathbf{T}$.
\end{lem}
\begin{proof}
We first compute
\begin{align*}
	\Ad(t | \mathfrak{h}_\alpha) = \left\{
		\begin{array}{lll}
			\mu_{t_it_j},   &\alpha = \epsilon_i +\epsilon_j \\
			\mu_{t_it_j^{-1}},   & \alpha = \epsilon_i -\epsilon_j\\
			\mu_{t_i},  & \alpha = \epsilon_i
	\end{array} \right.
\end{align*}
We thus have
\begin{align*}
	 \left( \frac{\det(\Ad(t | \mathfrak{h}_\alpha))}{k} \right) = \left\{
		\begin{array}{lll}
			\left(\frac{t_it_j}{k} \right),   &\alpha = \epsilon_i +\epsilon_j \\
			\left(\frac{t_it_j^{-1}}{k} \right),   & \alpha = \epsilon_i -\epsilon_j\\
			\left(\frac{t_i}{k} \right),  & \alpha = \epsilon_i
	\end{array} \right.
\end{align*}
It follows from Lemma~\ref{L:sign-roots} that $\sgn(\sigma_t | \mathbf{U})= \prod\limits_{1\leq i \leq r}\left(\frac{t_i}{k}\right)$.
\end{proof}

Let $W$ be an $2r+1$-dimensional  $k$-vector space equipped with a non-degenerated symmetric bilinear form $\operatorname{b}$.
Let  $(e_1,\cdots, e_r, u, f_r,\cdots,f_1)$ be a basis for $W$ whose Gram matrix is  $J_{2r+1}$.
We first recall the notion of \emph{spinor norm} $\chi_{\sn} \colon \mathbf{H} \to \{\pm 1\}$ (see, for example~\cite[Chapter 9, Definition 3.4]{scharlau}).
For any anisotropic vector $v\in W$, we define the reflection $\tau_v$ by $\tau_v(w) \coloneqq w-2 \frac{\operatorname{b}(w,v)}{ \operatorname{b}(v,v)}v$.
For any $h\in \mathbf{H}$, there exist anisotropic vectors  $v_1,\cdots,v_m$ such that $h=\tau_{v_1} \cdots \tau_{v_m}$.
The spinor norm is defined via $\chi_{\sn}(h) \coloneqq \left( \frac{\operatorname{b}(v_1,v_1) \cdots  \operatorname{b}(v_m,v_m)}{k} \right)$.

\begin{lem}\label{L:spinor}
	We have $\chi_{\sn}(t)=\left(\frac{\prod\limits_{1 \leq i \leq r} t_i}{k} \right)$ 
	for all $t= \Diag(t_1,\cdots, t_r, 1, t_r^{-1},\cdots, t_1^{-1})\in \mathbf{T}$.
\end{lem}

\begin{proof}
For $1\leq i \leq r$, let  $H_i \coloneqq  \langle e_i,f_i \rangle$.
Then $W = H_1\oplus \cdots \oplus H_r \oplus \langle u \rangle$.
Moreover, $t$ acts via identity on $\langle u \rangle$.
On each $H_i$, we have  $te_i=t_ie_i$ and  $tf_i=t^{-1}_i f_i$. 
We set $v_i =e_i+f_i$ and  $w_i= e_i + t_i^{-1}f_i$, then  $t|_{H_i} = r_{w_i} r_{v_i}$, which implies that $t=\prod\limits_{1\leq i\leq r} r_{w_i} r_{v_i}$.
We compute
\[
\prod\limits_{1\leq i\leq r}  \operatorname{b}(w_i,w_i) \operatorname{b}(v_i,v_i)=\prod\limits_{1\leq i\leq r} 2t^{-1}_i \cdot 2
\equiv \prod\limits_{1\leq i\leq r} t_i^{-1} \equiv \prod\limits_{1\leq i\leq r} t_i \pmod{(k^\times)^2)},
\]
which implies  $\chi_{\sn}(t) = \left(\frac{\prod\limits_{1 \leq i \leq r} t_i}{k} \right)$.
\end{proof}

Combining Lemmas~\ref{L:sign-orthogonal-1} and~\ref{L:spinor}, we obtain:

\begin{prop}\label{P:sign-orthogonal-1}
Let $\mathbf{H}=\SO(r,1)(k)$.
Then the spinor norm $\chi_{\sn} \colon \mathbf{H}\to \{\pm 1\}$ satisfies $\chi_{\sn}(t)= \sgn(\sigma_t| \mathbf{U})$ for all $t\in \mathbf{T}$.
\end{prop}

\subsubsection{$\mathbf{H}=\SO(r,2)(k)$}

Fix $\mathbf{u}\in k^\times\setminus (k^\times)^2$.
Let  $J= \begin{psmallmatrix} 0 & 0&  J_r\\ 0&U &0 \\ J_r & 0 &0\end{psmallmatrix}$, 
where $U= \begin{psmallmatrix} 1 &0 \\ 0 & -\mathbf{u} \end{psmallmatrix}$.
Let $\SO(r,2)(k) = \{g\in \SL_{2r+2}(k)  \mid g^{\operatorname{T}}J g = J\}$.

Let  $\omega\in \bar k^\times$ such that $\omega^2= \mathbf{u}$.
Then $\ell = k[\omega]$, and  $\operatorname{N}_{\ell/k}(a+b\omega) = a^2 - \mathbf{u} b^2$ for all $a,b\in k$.
For any $z= a+b \omega\in \ell$, we define  $U(z) = \begin{psmallmatrix} a & b \mathbf{u} \\ b & a \end{psmallmatrix}$, then  $\det U(z) = \operatorname{N}_{\ell/k} (z)$.
Let $$\mathbf{T}= \{\Diag(t_1,\cdots, t_r, U(z),   t_r^{-1},\cdots,  t_1^{-1}) | t_1,\cdots,t_r\in k^\times,  z\in \operatorname{N}^1(\ell/k)\},$$
\shu{and let $\mathbf{B}$ be the subgroup of $\mathbf{H}$ consisting of block upper triangular matrices w.r.t.\,the partition $(\underbrace{1,\cdots,1}_{r},2,\underbrace{1,\cdots,1}_{r})$.}
Let $$\mathbf{S}= \{\Diag(s_1,\cdots, s_r, 1, 1,  s_r^{-1},\cdots, s_1^{-1})  \mid s_1,\cdots,s_r\in k^\times\}.$$
Then $\Phi^+ = \{\epsilon_i \pm \epsilon_j \mid 1 \leq i < j \leq r\} \cup \{\epsilon_i  \mid 1\leq i \leq r\}$.

\begin{prop}\label{P:sign-orthogonal-2}
Let $\mathbf{H}=\SO(r,2)(k)$.
Then $\sgn(\sigma_t | \mathbf{U})=1$ for all $t\in \mathbf{T}$.
\end{prop}

\begin{proof}
Fix $t= \Diag(t_1,\cdots, t_r, U(z),  \bar t_r^{-1},\cdots, \bar t_1^{-1})\in \mathbf{T}$.
For any $1\leq i \leq r$, the weight space $\mathfrak{h}_{\epsilon_i}$ for $\epsilon_i$ is a $2$-dimensional $k$-space.
For any $w= \begin{psmallmatrix} x \\y \end{psmallmatrix}\in \mathfrak{h}_{\epsilon_i}$, we compute that $$\Ad(t | \mathfrak{h}_{\epsilon_i}). w = \mu_{t_i} U(z)w,$$
here $\mu_{t_i}$ denotes the scalar multiplication by  $t_i$.
Since $\det U(z)=1$, we have
$$\sgn(\sigma_t | \mathbf{U}_{\epsilon_i}) = \left( \frac{\det (t_i U(z))}{k} \right) = \left( \frac{\det (t_i)}{k} \right)= \left( \frac{t^2_i}{k} \right) =1.$$
We then have
\begin{align*}
	\Ad(t | \mathfrak{h}_\alpha) = \left\{
		\begin{array}{lll}
			\mu_{t_it_j},   &\alpha = \epsilon_i +\epsilon_j \\
			\mu_{t_it_j^{-1}},   & \alpha = \epsilon_i -\epsilon_j\\
			\mu_{t_i} U(z),  & \alpha = \epsilon_i
	\end{array} \right.
\end{align*}
We thus have
\begin{align*}
	\left( \frac{\det(\Ad(t | \mathfrak{h}_\alpha))}{k} \right) = \left\{
		\begin{array}{lll}
			\left(\frac{t_it_j}{k} \right),   &\alpha = \epsilon_i +\epsilon_j \\
			\left(\frac{t_it_j^{-1}}{k} \right),   & \alpha = \epsilon_i -\epsilon_j\\
			1,  & \alpha = \epsilon_i
	\end{array} \right.
\end{align*}
It follows from Lemma~\ref{L:sign-roots} that
$\sgn(\sigma_t | \mathbf{U}) = 1$, as desired.
\end{proof}

\subsubsection{$\mathbf{H}= \UU_{2r+1}(\ell/k)$}
Let $J_{2r+1}= \begin{psmallmatrix} 0 & 0&  J_r\\ 0&1 &0 \\ J_r & 0 &0\end{psmallmatrix}$.
Let $\UU_{2r+1} = \{g\in \SL_{2r+1}(\ell)  \mid \bar{g}^{\operatorname{T}}J_{2r+1} g = J_{2r+1}\}$.
Let $$\mathbf{T}= \{\Diag(t_1,\cdots, t_r, a,   t_r^{-1},\cdots,  t_1^{-1})  \mid t_1,\cdots,t_r\in \ell^\times, a\in \operatorname{N}^1(\ell/k)\},$$
and let $\mathbf{B}$ be the subgroup of $\mathbf{H}$ consisting of block upper triangular matrices.
Let $$\mathbf{S}= \{\Diag(s_1,\cdots, s_r, 1,  s_r^{-1},\cdots, s_1^{-1})  \mid s_1,\cdots,s_r\in k^\times\}.$$
For $1\leq i \leq r$, we denote by  $\epsilon_i$ the character of  $\mathbf{S}$ defined via $\epsilon_i(s) \coloneqq  s_i$ for $s= \Diag(s_1,\cdots, s_r, 1, s_r^{-1},\cdots, s_1^{-1})$.
Then $\Phi^+ = \{\epsilon_i \pm \epsilon_j  \mid 1 \leq i < j \leq r\} \cup \{\epsilon_i  \mid 1\leq i \leq r\} \cup \{2\epsilon_i  \mid 1\leq i \leq r\}$.

\begin{prop}\label{P:sign-unitary-odd}
Let $\mathbf{H}=\UU_{2r+1}(\ell/k)$.
Then $\sgn(\sigma_t | \mathbf{U})=1$ for all $t\in \mathbf{T}$.
\end{prop}

\begin{proof}
We first compute
\begin{align*}
	\Ad(t | \mathfrak{h}_\alpha) = \left\{
		\begin{array}{llll}
			\mu_{t_i \bar t_j} \in \GL_2(k),   &\alpha = \epsilon_i +\epsilon_j \\
			\mu_{t_it_j^{-1}} \in \GL_2(k),   & \alpha = \epsilon_i -\epsilon_j\\
			\mu_{\frac{t_i}{a}} \in \GL_2(k)   & \alpha= \epsilon_i\\
			\mu_{t_i \bar t_i} =\mu_{\operatorname{N}_{\ell/k}(t_i)}\in \GL_1(k),  & \alpha = 2\epsilon_i
	\end{array} \right.
\end{align*}
Notice that if $\Ad(t | \mathfrak{h}_\alpha)\in \GL_2(k)$, then $\det(\Ad(t | \mathfrak{h}_\alpha)) = \operatorname{N}_{\ell/k}(\Ad(t | \mathfrak{h}_\alpha))$; notice also that $\operatorname{N}_{\ell/k}(a)=\operatorname{N}_{\ell/k}(a^{-1})=1$.
We thus have
\begin{align*}
\left( \frac{\det(\Ad(t | \mathfrak{h}_\alpha))}{k} \right)= \left\{
	\begin{array}{lll}
		\left(\frac{\operatorname{N}_{\ell/k}(t_it_j)}{k} \right),   &\alpha = \epsilon_i +\epsilon_j \\
		\left(\frac{\operatorname{N}_{\ell/k}(t_it_j^{-1})}{k} \right),   & \alpha = \epsilon_i -\epsilon_j\\                        \left(\frac{\operatorname{N}_{\ell/k}(t_i)\operatorname{N}_{\ell/k}(a^{-1})}{k} \right)=\left(\frac{\operatorname{N}_{\ell/k}(t_i)}{k} \right),   & \alpha = \epsilon_i\\
		\left(\frac{\operatorname{N}_{\ell/k}(t_i)}{k} \right),  & \alpha = 2\epsilon_i
	\end{array} \right.
\end{align*}
It follows from Lemma~\ref{L:sign-roots} that $\sgn(\sigma_t | \mathbf{U})=1$.
\end{proof}

\subsubsection{$\mathbf{H}= \UU_{2r}(\ell/k)$}

Let $J_{2r}= \begin{psmallmatrix} 0 &  J_r\\ J_r & 0\end{psmallmatrix}$.
Let $\UU_{2r} = \{g\in \SL_{2r+1}(\ell)  \mid \bar{g}^{\operatorname{T}}J_{2r} g = J_{2r}\}$.
Let $$\mathbf{T}= \{\Diag(t_1,\cdots, t_r,  \bar t_r^{-1},\cdots, \bar t_1^{-1})  \mid t_1,\cdots,t_r\in \ell^\times\},$$
and let $\mathbf{B}$ be the subgroup of $\mathbf{H}$ consisting of block upper triangular matrices.
Let $$\mathbf{S}= \{\Diag(s_1,\cdots, s_r,  s_r^{-1},\cdots, s_1^{-1})  \mid s_1,\cdots,s_r\in k^\times\}.$$
For $1\leq i \leq r$, we denote by  $\epsilon_i$ the character of  $\mathbf{S}$ defined via $\epsilon_i(s) \coloneqq  s_i$ for  $s= \Diag(s_1,\cdots, s_r, s_r^{-1},\cdots, s_1^{-1})$.
Then $\Phi^+ = \{\epsilon_i \pm \epsilon_j  \mid 1 \leq i < j \leq r\}  \cup \{2\epsilon_i  \mid 1\leq i \leq r\}$.

\begin{lem}\label{L:sign-unitary-even}
Let $\mathbf{H}=\UU_{2r}(\ell/k)$.
Then $$\sgn(\sigma_t | \mathbf{U})=\left(\frac{\prod\limits_{1 \leq i \leq r} \operatorname{N}_{\ell/k}(t_i)}{k} \right)$$ 
for all $t= \Diag(t_1,\cdots, t_r, \bar t_r^{-1},\cdots, \bar t_1^{-1})\in \mathbf{T}$.
\end{lem}
\begin{proof}
We first compute
\begin{align*}
\Ad(t | \mathfrak{h}_\alpha) = \left\{
	\begin{array}{lll}
		\mu_{t_i \bar t_j}\in \GL_2(k),   &\alpha = \epsilon_i +\epsilon_j \\
		\mu_{t_it_j^{-1}}\in \GL_2(k),   & \alpha = \epsilon_i -\epsilon_j\\
		\mu_{t_i \bar t_i} =\mu_{\operatorname{N}_{\ell/k}(t_i)}\in \GL_1(k),  & \alpha = 2\epsilon_i
        \end{array} \right.
\end{align*}
We thus have
\begin{align*}
\left( \frac{\det(\Ad(t | \mathfrak{h}_\alpha))}{k} \right)=  \left\{
	\begin{array}{lll}
		\left(\frac{\operatorname{N}_{\ell/k}(t_it_j)}{k} \right),   &\alpha = \epsilon_i +\epsilon_j \\
		\left(\frac{\operatorname{N}_{\ell/k}(t_it_j^{-1})}{k} \right),   & \alpha = \epsilon_i -\epsilon_j\\     
		\left(\frac{\operatorname{N}_{\ell/k}(t_i)}{k} \right),  & \alpha = 2\epsilon_i
	\end{array} \right.
\end{align*}
It follows from Lemma~\ref{L:sign-roots} that $\sgn(\sigma_t | \mathbf{U})= \left(\frac{\prod\limits_{1 \leq i \leq r} \operatorname{N}_{\ell/k}(t_i)}{k}\right)$.
\end{proof}

For any $x\in \operatorname{N}^1(\ell/k)$, there exists  $a\in \ell^\times$ such that $x=a \bar{a}^{-1}$ by Hilbert's 90.
We define 
\[
	\delta \colon \operatorname{N}^1(\ell/k) \longto \{\pm 1\}, \qquad  x \longmapsto \left( \frac{\operatorname{N}_{\ell/k}(a)}{k} \right).
\]
Suppose that $x= a \bar{a}^{-1}= b \bar{b}^{-1}$, then $a/b = \overline{(a/b)}\in k^\times$. 
We then write  $a=by$ for some  $y\in k^\times$. 
Note that $\operatorname{N}_{\ell/k}(y)=y^2\in (k^\times)^2$.
We then have  
$$\left( \frac{\operatorname{N}_{\ell/k}(a)}{k} \right) = \left( \frac{\operatorname{N}_{\ell/k}(by)}{k} \right) = \left( \frac{\operatorname{N}_{\ell/k}(b)\operatorname{N}_{\ell/k}(y)}{k} \right)= \left( \frac{\operatorname{N}_{\ell/k}(b)}{k} \right),$$
which implies that $\delta$ is well-defined.
It is obvious that $\delta$ is a group homomorphism.

Notice that the determinant map $\det \colon \mathbf{H} \to \ell^\times$ takes value in $\operatorname{N}^1(\ell/k)$, we may then define
 $$\chi_{\det} \coloneqq \delta\circ\det \colon \mathbf{H} \to \{\pm 1\}.$$

 \begin{prop}\label{P:sign-unitary-even}
 Let $\mathbf{H}=\UU_{2r}(\ell/k)$.
 Then $\chi_{\det}(t)= \sgn(\sigma_t| \mathbf{U})$ for all $t\in \mathbf{T}$.
\end{prop}

\begin{proof}
Let  $t= \Diag(t_1,\cdots, t_r, \bar t_r^{-1},\cdots, \bar t_1^{-1})$, then
\[
	\chi_{\det}(t) = \delta \left( \left(\prod\limits_{i=1}^r t_i \right) \left(\prod\limits_{i=1}^r \bar t^{-1}_i \right) \right)
	= \left(\frac{\prod\limits_{1 \leq i \leq r} \operatorname{N}_{\ell/k}(t_i)}{k}\right),
\] 
as desired.
\end{proof}

\subsection{Compatible families of $\beta$-extensions}
We shall employ the notations of~\cite{daniel-3}, especially from \S 6 and \S 9 in \emph{loc.\,cit.}, for $\beta$-extensions.
In this subsection, we fix an arbitrary chamber $\mathcal{C}$ in $\mathscr{B}(G_\beta)$.
We fix a self-dual ${\mathfrak{o}}_E$-${\mathfrak{o}}_D$ lattice sequence $\Upsilon$ corresponding to a point in $\mathcal{C}$, together with a semisimple character $\theta_{\Upsilon}\in \CCC(\Upsilon, 0, \beta)$.
We denote by $\eta_{\Upsilon}$ the Heisenberg representation containing $\theta_{\Upsilon}$.
For any  self-dual ${\mathfrak{o}}_E$-${\mathfrak{o}}_D$ lattice sequence $\Lambda$ corresponding to a point in $\overline{\mathcal{C}}$, there is a unique $\theta_{\Lambda} \in \CCC(\Lambda, 0, \beta)$ which is the transfer of $\theta_{\Upsilon}$.
We denote by $\eta_{\Lambda}$ the Heisenberg representation containing $\theta_{\Lambda}$.

Let $\Lambda$ and $\Lambda'$ be two self-dual $\mathfrak{o}_E$-$\mathfrak{o}_D$-lattice sequences \shu{with $\tilde{\mathfrak{b}}(\Lambda) \subseteq \tilde{\mathfrak{b}}(\Lambda')$.}
\begin{itemize}
\item  Assume $\tilde{\mathfrak{a}}(\Lambda) \subseteq \tilde{\mathfrak{a}}(\Lambda')$.
By~\cite[Proposition 4.5]{daniel-3}, there exists \shu{up to equivalence} a unique irreducible representation $\eta_{\Lambda,\Lambda'}$ of  $\JJ^1_{\Lambda,\Lambda'}$ which extends $\eta_{\Lambda'}$ and such that
	$$\ind^{\PP_1(\Lambda)}_{\JJ^1_{\Lambda,\Lambda'}} \eta_{\Lambda,\Lambda'} \cong \ind^{\PP_1(\Lambda)}_{\JJ^1_{\Lambda}}\eta_{\Lambda}.$$
\item  
For general $\Lambda$ and $\Lambda'$ with $\tilde{\mathfrak{b}}(\Lambda) \subseteq \tilde{\mathfrak{b}}(\Lambda')$,	we denote by $\ext(\Lambda, \Lambda')$ the set of isomorphism classes of extensions of $\eta_{\Lambda, \Lambda'}$ from $\JJ^1_{\Lambda,\Lambda'}$ to $\JJ_{\Lambda,\Lambda'}$, \shu{provided by~\cite[Proposition 4.7]{daniel-3} and~\cite[Definition 4.8]{daniel-3}.}
In particular, we set $\ext(\Lambda) \coloneqq \ext(\Lambda,\Lambda)$.
\end{itemize}

For a vertex $\Lambda_{\mathsf{M}}$ of $\mathcal{C}$, there is a subset $\beta$-$\ext(\Lambda_{\mathsf{M}})$ of $\ext(\Lambda_{\mathsf{M}})$ whose elements are called $\beta$-extensions of the Heisenberg representation $\eta_M \coloneqq \eta_{\Lambda_{\mathsf{M}}}$.

Given three  $\mathfrak{o}_D$-$\mathfrak{o}_D$-lattice sequences $\Lambda, \Lambda'$ and $\Lambda''$ satisfying
\[
\tilde{\mathfrak{b}}(\Lambda) \subseteq \tilde{\mathfrak{b}}(\Lambda') \cap \tilde{\mathfrak{b}}(\Lambda''),
\] 
we have a transfer bijection
\[
	\Psi_{\Lambda, \Lambda', \Lambda''} \colon \ext(\Lambda, \Lambda') \longto \ext(\Lambda, \Lambda''),
\] 
For a vertex $\Lambda_{\mathsf{M}}$ with $\tilde{\mathfrak{b}}(\Lambda) \subseteq \tilde{\mathfrak{b}}(\Lambda_{\mathsf{M}})$, we then define the set 
$$\beta\text{-}\ext_{\Lambda_{\mathsf{M}}}(\Lambda)\coloneqq  \Psi_{\Lambda,\Lambda_{\mathsf{M}},\Lambda} \left(\Res^{\JJ_{\Lambda_{\mathsf{M}}}}_{\JJ_{\Lambda,\Lambda_{\mathsf{M}}}} \left(\beta\text{-}\ext(\Lambda_{\mathsf{M}}) \right) \right),$$
whose elements are called $\beta$-extensions of $\eta_{\Lambda}$ to $\JJ_{\Lambda}$ with respect to $\Lambda_{\mathsf{M}}$.

We also define
\[
	\beta\text{-}\ext^\circ_{\Lambda_{\mathsf{M}}}(\Lambda)\coloneqq \Res^{\JJ_{\Lambda}}_{\JJ^\circ_{\Lambda}} \left( \beta\text{-}\ext_{\Lambda_{\mathsf{M}}}(\Lambda) \right),
\] 
whose elements are called \emph{$\beta$-extensions} of $\eta_{\Lambda}$ to $\JJ^\circ_{\Lambda}$ \emph{relative to $\Lambda_{\mathsf{M}}$}.
We put 
$$\beta\text{-}\ext^\circ(\Lambda_{\mathsf{M}}) \coloneqq \beta\text{-}\ext^\circ_{\Lambda_{\mathsf{M}}}(\Lambda_{\mathsf{M}}).$$

We define $\Ext^\circ(\Lambda, \Lambda')$ to be the set of isomorphism classes of extensions of $\eta_{\Lambda, \Lambda'}$ from $\JJ^1_{\Lambda,\Lambda'}$ to $\JJ^\circ_{\Lambda,\Lambda'}$.

\begin{lem}\label{L:Psi-circ}
There exists a bijection
	\[
\Psi^\circ_{\Lambda, \Lambda', \Lambda''} \colon \Ext^\circ(\Lambda,\Lambda') \longto \Ext^\circ(\Lambda, \Lambda'').
\]
such that
\begin{equation}\label{E:Psi-res}
\Psi^\circ_{\Lambda, \Lambda', \Lambda''} \circ \Res^{\JJ_{\Lambda,\Lambda'}}_{\JJ^\circ_{\Lambda,\Lambda'}}= \Res^{\JJ_{\Lambda,\Lambda''}}_{\JJ^\circ_{\Lambda,\Lambda''}} \circ \Psi_{\Lambda, \Lambda', \Lambda''}. 
\end{equation}
\end{lem}

\begin{proof}
As in~\cite[\S 6.1]{daniel-3}, we may choose a path of self-dual $\mathfrak{o}_E$-$\mathfrak{o}_D$-lattice sequences
\[
	\Lambda' = \Lambda_{(0)}, \Lambda_{(1)}, \cdots, \Lambda_{(u)}= \Lambda''
\] 
such that for all indexes $s \in\{1, \ldots, u\}$ the condition
$$
\tilde{\mathfrak{a}}\left(\Lambda_{(s-1)}\right) \cap \tilde{\mathfrak{a}}\left(\Lambda_{(s)}\right) \in\left\{\tilde{\mathfrak{a}}\left(\Lambda_{(s-1)}\right), \tilde{\mathfrak{a}}\left(\Lambda_{(s)}\right)\right\}
$$
is satisfied.
We may thus reduce to the case where $\tilde{\mathfrak{a}} (\Lambda) \subseteq  \tilde{\mathfrak{a}}(\Lambda') \cap \tilde{\mathfrak{a}}(\Lambda'')$.
As in the proof of Lemma 6.4 in \emph{loc.\,cit.}, the map $\Psi^\circ_{\Lambda, \Lambda', \Lambda''}$ is characterized by the formula
\begin{equation}\label{E:Psi-circ}
	\ind^{\PP^\circ_{\Lambda,\Lambda}}_{\JJ^\circ_{\Lambda,\Lambda'}} \kappa{'^\circ} \cong	\ind^{\PP^\circ_{\Lambda,\Lambda}}_{\JJ^\circ_{\Lambda,\Lambda''}} \Psi^\circ_{\Lambda,\Lambda', \Lambda''}(\kappa{'^\circ}),  \qquad \forall \kappa{'^\circ}\in \Ext^\circ(\Lambda,\Lambda').
\end{equation}

For any $\kappa' \in \ext(\Lambda,\Lambda')$, we have
\begin{equation}\label{E:Psi}
	\ind^{\PP_{\Lambda,\Lambda}}_{\JJ_{\Lambda,\Lambda'}} \kappa \cong	\ind^{\PP_{\Lambda,\Lambda}}_{\JJ_{\Lambda,\Lambda''}} \Psi_{\Lambda, \Lambda', \Lambda''}(\kappa').
\end{equation}
We then restrict~\eqref{E:Psi} to obtain
	\[
		\ind^{\PP^\circ_{\Lambda,\Lambda}}_{\JJ^\circ_{\Lambda,\Lambda'}} \Res^{\JJ_{\Lambda,\Lambda'}}_{\JJ^\circ_{\Lambda,\Lambda'}}\kappa' \cong	\ind^{\PP^\circ_{\Lambda,\Lambda}}_{\JJ^\circ_{\Lambda,\Lambda''}} \Res^{\JJ_{\Lambda,\Lambda''}}_{\JJ^\circ_{\Lambda,\Lambda''}}\Psi_{\Lambda, \Lambda', \Lambda''}(\kappa').
\]
It then follows from~\eqref{E:Psi-circ} that
\[
	\Psi^\circ_{\Lambda,\Lambda',\Lambda''} \left(\Res^{\JJ_{\Lambda,\Lambda'}}_{\JJ^\circ_{\Lambda,\Lambda'}}\kappa' \right) = \Res^{\JJ_{\Lambda,\Lambda''}}_{\JJ^\circ_{\Lambda,\Lambda''}}\Psi_{\Lambda, \Lambda', \Lambda''}(\kappa'),
\] 
proving~\eqref{E:Psi-res}.
\end{proof}

\begin{defn}[{\cite[Definition 9.1]{daniel-3}}]
Let $\kappa$ be a finite dimensional representation of a locally profinite group $J$.
We say that  $\kappa$ satisfies  \textbf{(ORD)} if $\det(\kappa)$ is a character of order dividing $2^\epsilon p^s$ for $\epsilon \in \{0,1\}$ and some integer $s \geq 0$; we say that $\kappa$ satisfies  \textbf{(ORD)}-$p$ if $\epsilon =0$.
\end{defn}

\begin{defn}\label{D:compatible-family}
We say that $\{ \hat{\kappa}^\circ_x  \mid x\in\overline{\mathcal{C}}\}$ is a \emph{compatible family} of $\beta$-extensions if 
\begin{itemize}
\item[(i)]  $\hat{\kappa}^\circ_x \in \beta\text{-}\ext^\circ_{\Lambda_{\mathsf{M}}}(\Lambda_x)$ for some vertex $\Lambda_{\mathsf{M}}$ of $\overline{\mathcal{C}}$ with $\tilde{\mathfrak{b}}(\Lambda_x) \subseteq \tilde{\mathfrak{b}}(\Lambda_{\mathsf{M}})$;
\item[(ii)] for any vertices $\Lambda_1, \Lambda_2$ of \shu{$\overline{\mathcal{C}}$} such that $\tilde{\mathfrak{b}}(\Lambda_x) \subseteq \tilde{\mathfrak{b}}(\Lambda_{1}) \cap \tilde{\mathfrak{b}}(\Lambda_{2})$, we have
\[
	\hat{\kappa}^\circ_x \cong 
	\Psi^\circ_{\Lambda_x,\Lambda_1, \Lambda_x} \left(\Res^{\JJ^\circ_{\Lambda_1}}_{\JJ^\circ_{\Lambda_x, \Lambda_1}}\hat{\kappa}^\circ_1 \right) \cong \Psi^\circ_{\Lambda_x,\Lambda_2, \Lambda_x} \left(\Res^{\JJ^\circ_{\Lambda_2}}_{\JJ^\circ_{\Lambda_x, \Lambda_2}}\hat{\kappa}^\circ_2 \right).
\] 
\end{itemize}

\end{defn}

\begin{cons}\label{C:chi}

Fix an arbitrary $x\in \overline{ \mathcal{C}}$.
We denote by $\Lambda_x$ the $\mathfrak{o}_D$-$\mathfrak{o}_D$-lattice sequence corresponding to $j_\beta(x)\in \mathscr{B}(G)$.
For any $g\in \JJ^\circ_{\Lambda_x}$, we denote by $\sigma_g$ the conjugation action of  $g$ on the left coset $\JJ^1_{\Lambda_x} \backslash \PP_1(\Lambda_x)$.
We define
\[
	\chi^\circ_{x,1} \colon  \JJ^\circ_{\Lambda_x} \longto \{ \pm 1\}, \qquad   g \longmapsto \sgn \left(\sigma_g | \JJ^1_{\Lambda_x} \backslash \PP_1(\Lambda_x)\right) 
\]

We identify the reductive quotient $\mathbf{H}^\circ_{x,\beta} \coloneqq \PP^\circ(\Lambda_{x,\beta})/ \PP_1(\Lambda_{x,\beta})$ with a product of some finite reductive groups described in Section~\ref{S:sign-extension}.
By Proposition~\ref{P:sign-extension}, there exists a quadratic character 
$\bar\chi^\circ_{x,2} \colon  \mathbf{H}^\circ_{x,\beta}  \to \{ \pm 1\}$
such that $\bar\chi^\circ_{x,2}(b) = \sgn \left(\sigma_b | \PP_1(\mathcal{C})/\PP_1(\Lambda_{x,\beta}) \right)$ for all $b\in \PP^\circ(\mathcal{C})/\PP_1(\Lambda_{x,\beta})$.
We view $\bar\chi^\circ_{x,2}$ as a quadratic character of $\JJ^\circ_{\Lambda_x} / \JJ^1_{\Lambda_x}$ and then inflate it to a quadratic character
\[
	\chi^\circ_{x,2} \colon \JJ^\circ_{\Lambda_x} \longto \{ \pm 1\}.
\]

We have that $\mathbf{H}^\circ_x \coloneqq \PP^\circ(\Lambda_x)/ \PP_1(\Lambda_x)$ is a product of some finite reductive groups described in Section~\ref{S:sign-extension}.
By Proposition~\ref{P:sign-extension}, there exists a quadratic character 
$\bar\chi^\circ_{x,3} \colon \mathbf{H}^\circ_x \to \{ \pm 1\}$
such that $\bar\chi^\circ_{x,3}(b) = \sgn \left(\sigma_b | \mathbf{U}_x \right)$ for any Borel subgroup $\mathbf{B}_x$ of $\mathbf{H}^\circ_x$ with unipotent radical  $\mathbf{U}_x$ and for all
$b\in \mathbf{B}_x$.
Let 
\[
	\chi^\circ_{x,3} \colon \PP^\circ(\Lambda_x) \longto \{ \pm 1\}
\]
be the inflation of  $\bar\chi^\circ_{x,3}$ to $\PP^\circ(\Lambda_x)$.

Finally, we define
\[
	\chi^\circ_x \coloneqq \chi^\circ_{x,1}\cdot\chi^\circ_{x,2}\cdot\chi^\circ_{x,3} \colon \JJ^\circ_{\Lambda_x} \longto   \{ \pm 1\}.
\]
Then $\chi^\circ_x$ is a quadratic character of $\JJ^\circ_{\Lambda_x}$ which is trivial on $\JJ^1_{\Lambda_x}$.

If $x$ is a vertex, we let $\kappa_x\in \beta$-$\ext(\Lambda_x)$ be the unique $\beta$-extension satisfying \textbf{(ORD)}-$p$ given by~\cite[Lemma 9.5]{daniel-3}.
In general, we choose a self-dual $\mathfrak{o}_E$-$\mathfrak{o}_D$-lattice sequence $\Lambda_{\mathsf{M}}$ which corresponds to $j_\beta(x_{\mathsf{M}})$ for some vertex $x_{\mathsf{M}}$ of ${\mathcal{C}}$ such that $\tilde{\mathfrak{b}}(\Lambda_x) \subseteq \tilde{\mathfrak{b}}(\Lambda_{\mathsf{M}})$, and then choose $\kappa_{{\mathsf{M}},x}\in \beta$-$\ext_{\Lambda_{\mathsf{M}}}(\Lambda_x)$ satisfying \textbf{(ORD)}.
As in the proof of \emph{ibid.}, there exists a quadratic character $\phi_x$ of $\JJ_{\Lambda_x}$ trivial on $\JJ^1_{\Lambda_x}$ such that
$\kappa_x \coloneqq \kappa_{{\mathsf{M}},x} \otimes \phi_x$ satisfies \textbf{(ORD)}-$p$.

We put
\[
	\hat{\kappa}^\circ_x \coloneqq  \left( \Res^{\JJ_{\Lambda_x}}_{\JJ^\circ_{\Lambda_x}}\kappa_x \right) \otimes \chi^\circ_x.
\] 

\end{cons}

\begin{rmk}
Suppose that $x=x_{\sM}$ is a vertex in $\overline{\mathcal{C}}$.
\shu{
Then $\mathbf{H}^\circ_{x,\beta} \cong \prod_{i\in I} \mathbf{H}^\circ_{x,\beta_i}$
where each $\mathbf{H}^\circ_{x,\beta_i}$ is a product of two finite reductive groups of Lie type considered in Subsection~\ref{S:sign-extension}.
We write $\mathbf{H}_{x,\beta} \coloneqq \PP(\Lambda_{x,\beta})/ \PP_1(\Lambda_{x,\beta})$ and $\mathbf{H}_x \coloneqq \PP(\Lambda_x)/ \PP_1(\Lambda_x)$.
Then $\mathbf{H}_{x,\beta}$ is a subgroup of $\prod_{i\in I} \mathbf{H}_{x,\beta_i}$
where each $\mathbf{H}^\circ_{x,\beta_i}$ has  index at most $2$ in $\mathbf{H}_{x,\beta_i}$.
}
By construction, $\bar\chi^\circ_{x,2}$ is invariant under the conjugation action of $\mathbf{H}_i$, therefore 
extends to a quadratic character $\bar\chi_{x,2}$ of $\mathbf{H}_{x,\beta}$.
We then inflate $\bar\chi_{x,2}$ to a quadratic character $\chi_{x,2}$ of  $\JJ_{\Lambda_x}$ extending  $\chi^\circ_{x,2}$.
Similarly, $\chi^\circ_{x,3}$ extends to a quadratic character $\chi_{x,3}$ of $\JJ_{\Lambda_x}$.
Finally, it is obvious that $\chi^\circ_{x,1}$ extends to a quadratic character $\chi_{x,1}$ of $\JJ_{\Lambda_x}$.
Put $\chi_{x} \coloneqq \chi_{x,1} \cdot \chi_{x,2} \cdot \chi_{x,3}$.
Then $\chi_{x}$ is a quadratic character of  $\JJ_{\Lambda_{x}}$ such that $\Res^{\JJ_{\Lambda_{x}}}_{\JJ^\circ_{\Lambda_{x}}} \chi_{x} = \chi^\circ_{x}$.
In particular, $\chi_{x}$ is trivial on $\JJ^1_{\Lambda_{x}}$.
We then have that $\hat\kappa_{x} \coloneqq \kappa_{x} \otimes \chi_{x} \in \beta\text{-}\ext(\Lambda_{x})$ is a  $\beta$-extension and 
\begin{equation}\label{E:kappaM}
	\Res^{\JJ_{\Lambda_{x}}}_{\JJ^\circ_{\Lambda_{x}}} \hat{\kappa}_{x} = \hat{\kappa}^\circ_{x}.
\end{equation}

\end{rmk}

\begin{rmk}\label{R:sign-cardinality}
Let $S$ be a finite set with  $|S|$ odd, equipped with an action of a finite group $H$.
For any  $g\in H$, we denote by $\sigma_g$ the action of  $\langle g \rangle$ on $S$ and
by $S/ \langle g \rangle$ the orbit of $S$ under $\sigma_g$.
We then have
\begin{align}\label{E:sign-cardinality}
\sgn(\sigma_g | S) =(-1)^{|S|-|S/ \langle g \rangle|}=  (-1)^{|S/ \langle g \rangle|-1}.
\end{align}

\end{rmk}

\begin{prop}\label{P:compatible-beta-extension}
For all self-dual $\mathfrak{o}_E$-$\mathfrak{o}_D$-lattice sequences $\Lambda, \Lambda_x$ and $\Lambda_y$ satisfying 
$\tilde{\mathfrak{b}}(\Lambda) \subseteq \tilde{\mathfrak{b}}(\Lambda_x) \cap \tilde{\mathfrak{b}}(\Lambda_y)$, we have 
\[
	\Psi^\circ_{\Lambda,\Lambda_x,\Lambda_y} \left(\Res^{\JJ^\circ_{\Lambda_x}}_{\JJ^\circ_{\Lambda, \Lambda_x}}\hat\kappa^\circ_x \right) \cong   \Res^{\JJ^\circ_{\Lambda_y}}_{\JJ^\circ_{\Lambda, \Lambda_y}} \hat{\kappa}^\circ_y.
\] 
\end{prop}

\begin{proof}
As in~\cite[\S 6.1]{daniel-3}, we may choose a path of self-dual $\mathfrak{o}_E$-$\mathfrak{o}_D$-lattice sequences
\[
	\Lambda_x = \Lambda_{(0)}, \Lambda_{(1)}, \cdots, \Lambda_{(u)}= \Lambda_y
\]
such that for all $s \in\{1, \ldots, u\}$ one has
\[
\tilde{\mathfrak{a}}\left(\Lambda_{(s-1)}\right) \cap \tilde{\mathfrak{a}}\left(\Lambda_{(s)}\right)
\in\left\{\tilde{\mathfrak{a}}\left(\Lambda_{(s-1)}\right), \tilde{\mathfrak{a}}\left(\Lambda_{(s)}\right)\right\} \quad \text{and} \quad 
\tilde{\mathfrak{b}}\left(\Lambda\right)\subseteq \tilde{\mathfrak{b}} \left(\Lambda_{(s-1)}\right) \cap \tilde{\mathfrak{b}}\left(\Lambda_{(s)}\right).
\]
We may thus reduce to the case where
\[
\tilde{\mathfrak{a}}(\Lambda)\subseteq \tilde{\mathfrak{a}}(\Lambda_x)\cap \tilde{\mathfrak{a}}(\Lambda_y)
\qquad\text{and}\qquad
\tilde{\mathfrak{a}}(\Lambda_y)\subseteq \tilde{\mathfrak{a}}(\Lambda_x),
\]
as in~\cite[\S 6.3]{daniel-3}. In particular, $\PP_1(\Lambda_x)\subseteq \PP_1(\Lambda_y)$.

We write
\[
\kappa^\circ_{\Lambda,\Lambda_x}\coloneqq \Res^{\JJ_{\Lambda_x}}_{\JJ^\circ_{\Lambda,\Lambda_x}}\kappa_x,
\qquad
\kappa^\circ_{\Lambda,\Lambda_y}\coloneqq \Res^{\JJ_{\Lambda_y}}_{\JJ^\circ_{\Lambda,\Lambda_y}}\kappa_y,
\]
and we have
\[
\Psi^\circ_{\Lambda,\Lambda_x,\Lambda_y}(\kappa^\circ_{\Lambda,\Lambda_x})
\cong \kappa^\circ_{\Lambda,\Lambda_y}\otimes \nu^\circ_{x,y}
\]
for some quadratic character $\nu^\circ_{x,y}$ of $\JJ^\circ_{\Lambda,\Lambda_y}$ trivial on $\JJ^1_{\Lambda,\Lambda_y}$.
Thus
\begin{equation}\label{E:nu-rewrite}
	\ind^{\PP^\circ_{\Lambda,\Lambda}}_{\JJ^\circ_{\Lambda,\Lambda_x}} \kappa^\circ_{\Lambda,\Lambda_x}
	\cong
	\ind^{\PP^\circ_{\Lambda,\Lambda}}_{\JJ^\circ_{\Lambda,\Lambda_y}}
	(\kappa^\circ_{\Lambda,\Lambda_y}\otimes \nu^\circ_{x,y}).
\end{equation}

Put
\[
\pi_x \coloneqq \ind^{\PP^\circ_{\Lambda,\Lambda}}_{\JJ^\circ_{\Lambda,\Lambda_x}} \kappa^\circ_{\Lambda,\Lambda_x},
\qquad
\pi_y \coloneqq \ind^{\PP^\circ_{\Lambda,\Lambda}}_{\JJ^\circ_{\Lambda,\Lambda_y}} \kappa^\circ_{\Lambda,\Lambda_y}.
\]
\daniel{For any $g\in \PP^\circ(\Lambda_\beta)$}, it follows from~\eqref{E:nu-rewrite} that
\[
\det(\pi_x)(g)=\det(\pi_y)(g)\cdot \nu^\circ_{x,y}(g)^{\dim(\kappa^\circ_y)}
=\det(\pi_y)(g)\cdot \nu^\circ_{x,y}(g),
\]
since $\dim(\kappa^\circ_y)$ is odd. 
Hence
\[
\nu^\circ_{x,y}(g)=\frac{\det(\pi_x)(g)}{\det(\pi_y)(g)}.
\]
Since $\kappa^\circ_x$ and $\kappa^\circ_y$ satisfy \textbf{(ORD)}-$p$, by~\cite[Lemma 9.6]{daniel-3} we have
\[
\det(\pi_x)(g)=(-1)^{\left| (\JJ^\circ_{\Lambda,\Lambda_x}\backslash \PP^\circ_{\Lambda,\Lambda})/\langle g\rangle \right|-1},
\qquad
\det(\pi_y)(g)=(-1)^{\left| (\JJ^\circ_{\Lambda,\Lambda_y}\backslash \PP^\circ_{\Lambda,\Lambda})/\langle g\rangle \right|-1}.
\]
Therefore
\[
\nu^\circ_{x,y}(g)
=
(-1)^{e_{x,y}(g)},
\]
where
\[
e_{x,y}(g)\coloneqq
\left| (\JJ^\circ_{\Lambda,\Lambda_x}\backslash \PP^\circ_{\Lambda,\Lambda})/\langle g\rangle \right|
+
\left| (\JJ^\circ_{\Lambda,\Lambda_y}\backslash \PP^\circ_{\Lambda,\Lambda})/\langle g\rangle \right|.
\]

Set
\[
e_{z,1}(g)\coloneqq
\left|(\JJ^1_{\Lambda_z}\backslash \PP_1(\Lambda_z))/\langle g\rangle\right|,   \qquad z\in \{x,y\},
\]
\[
e_{x,y,2}(g)\coloneqq
\left|(\PP_1(\Lambda_{x,\beta})\backslash \PP_1(\Lambda_{y,\beta}))/\langle g\rangle\right|,
\quad \text{and} \quad
e_{x,y,3}(g)\coloneqq
\left|(\PP_1(\Lambda_x)\backslash \PP_1(\Lambda_y))/\langle g\rangle\right|.
\]
We now rewrite the exponent $e_{x,y}(g)$.
We have
\[
\left| (\JJ^\circ_{\Lambda,\Lambda_x}\backslash \PP^\circ_{\Lambda,\Lambda})/\langle g\rangle \right|
=
\left| (\JJ^1_{\Lambda,\Lambda_x}\backslash \PP^1(\Lambda))/\langle g\rangle \right|,
\]
and then by Lemma~\ref{L:sign-coset},
\[
\left| (\JJ^\circ_{\Lambda,\Lambda_x}\backslash \PP^\circ_{\Lambda,\Lambda})/\langle g\rangle \right|
\equiv
\left| (\PP^1_{\Lambda,\Lambda_y}\backslash \PP^1_{\Lambda})/\langle g\rangle \right|
+
\left| (\JJ^1_{\Lambda,\Lambda_x}\backslash \PP^1_{\Lambda,\Lambda_y})/\langle g\rangle \right|
-1
\pmod 2.
\]
Similarly,
\[
\left| (\JJ^\circ_{\Lambda,\Lambda_y}\backslash \PP^\circ_{\Lambda,\Lambda})/\langle g\rangle \right|
\equiv
\left| (\PP^1_{\Lambda,\Lambda_y}\backslash \PP^1_{\Lambda})/\langle g\rangle \right|
+
\left| (\JJ^1_{\Lambda,\Lambda_y}\backslash \PP^1_{\Lambda,\Lambda_y})/\langle g\rangle \right|
-1
\pmod 2.
\]
Hence
\begin{equation}\label{E:exy-0}
e_{x,y}(g)
\equiv
\left| (\JJ^1_{\Lambda,\Lambda_x}\backslash \PP^1_{\Lambda,\Lambda_y})/\langle g\rangle \right|
+
\left| (\JJ^1_{\Lambda,\Lambda_y}\backslash \PP^1_{\Lambda,\Lambda_y})/\langle g\rangle \right|
\pmod 2.
\end{equation}
The second summand of the RHS of~\eqref{E:exy-0} satisfies
\[
\left| (\JJ^1_{\Lambda,\Lambda_y}\backslash \PP^1_{\Lambda,\Lambda_y})/\langle g\rangle \right|
\equiv e_{y,1}(g)\pmod 2,
\]
moreover, a repeated application of Lemma~\ref{L:sign-coset} to the first summand gives
\begin{equation*}
\scalebox{0.83}{$\displaystyle
\begin{aligned}
&\left| (\JJ^1_{\Lambda,\Lambda_x} \backslash \PP^1_{\Lambda,\Lambda_y})/ \langle g \rangle \right| \\
 \equiv& \left| (\JJ^1_{\Lambda,\Lambda_x} \backslash \PP^1_{\Lambda,\Lambda_x})/ \langle g \rangle \right| + \left| (\PP^1_{\Lambda,\Lambda_x} \backslash \PP^1_{\Lambda,\Lambda_y})/ \langle g \rangle \right| -1 \\
 \equiv& \left| (\JJ^1_{\Lambda_x} \backslash \PP_1(\Lambda_x)/ \langle g \rangle \right| + \left| (\PP^1_{\Lambda,\Lambda_x} \backslash \PP^1_{\Lambda,\Lambda_y})/ \langle g \rangle \right| -1 \\
\equiv& \left| (\JJ^1_{\Lambda_x} \backslash \PP_1(\Lambda_x))/ \langle g \rangle \right| + \left(\left| (\PP_1(\Lambda_x) \backslash \PP^1_{\Lambda,\Lambda_x})/ \langle g \rangle \right| +\left| (\PP_1(\Lambda_x) \backslash \PP^1_{\Lambda,\Lambda_y})/ \langle g \rangle \right| \right) \\
\equiv& \left| (\JJ^1_{\Lambda_x} \backslash \PP_1(\Lambda_x))/ \langle g \rangle \right| + \left| (\PP_1(\Lambda_{x,\beta}) \backslash \PP_1(\Lambda_\beta))/ \langle g \rangle \right| +\left| (\PP_1(\Lambda_x) \backslash \PP^1_{\Lambda,\Lambda_y})/ \langle g \rangle \right|  \\
\equiv& \left| (\JJ^1_{\Lambda_x} \backslash \PP_1(\Lambda_x))/ \langle g \rangle \right| + \left| (\PP_1(\Lambda_{x,\beta}) \backslash \PP_1(\Lambda_\beta))/ \langle g \rangle \right| 
      + \left(\left| (\PP_1(\Lambda_x) \backslash \PP_1(\Lambda_y))/ \langle g \rangle \right| +  \left| (\PP_1(\Lambda_y) \backslash \PP^1_{\Lambda, \Lambda_y})/ \langle g \rangle \right|-1 \right) \\
\equiv& \left| (\JJ^1_{\Lambda_x} \backslash \PP_1(\Lambda_x))/ \langle g \rangle \right| + \left| (\PP_1(\Lambda_{x,\beta}) \backslash \PP_1(\Lambda_\beta))/ \langle g \rangle \right| 
  + \left| (\PP_1(\Lambda_x) \backslash \PP_1(\Lambda_y))/ \langle g \rangle \right| +  \left| (\PP_1(\Lambda_{y,\beta}) \backslash \PP_1(\Lambda_\beta))/ \langle g \rangle \right|-1 \\
\equiv& \left| (\JJ^1_{\Lambda_x} \backslash \PP_1(\Lambda_x))/ \langle g \rangle \right|+ \left| (\PP_1(\Lambda_{x,\beta}) \backslash \PP_1(\Lambda_{y,\beta}))/ \langle g \rangle \right|  + \left| (\PP_1(\Lambda_x) \backslash \PP_1(\Lambda_y))/ \langle g \rangle \right|\\
\equiv & e_{x,1}(g)+e_{x,y,2}(g)+e_{x,y,3}(g)\pmod{2}.
\end{aligned}
$}
\end{equation*}
Therefore, \eqref{E:exy-0} reads
\begin{equation}\label{E:exy-rewrite}
e_{x,y}(g)\equiv e_{x,1}(g)+e_{y,1}(g)+e_{x,y,2}(g)+e_{x,y,3}(g)\pmod 2.
\end{equation}
Next, \shu{for $g\in \PP^\circ(\mathcal{C})$,} we put
\[
e_{z,2}(g)\coloneqq
\left|(\PP_1(\Lambda_{z,\beta})\backslash \PP_1(\mathcal{C})) /\langle g\rangle\right|
\qquad z\in \{x,y\}.
\]
Since
$\PP_1(\Lambda_{x,\beta})\subseteq \PP_1(\Lambda_{y,\beta}) \subseteq \PP_1(\mathcal{C})$, by Lemma~\ref{L:sign-coset} we have
\begin{equation}\label{E:Dxy-Cxy}
e_{x,y,2}(g)\equiv e_{x,2}(g)+e_{y,2}(g)-1 \pmod 2.
\end{equation}
Combining~\eqref{E:exy-rewrite} and~\eqref{E:Dxy-Cxy}, we obtain
\begin{equation}\label{E:exy-final}
e_{x,y}(g)\equiv e_{x,1}(g)+e_{y,1}(g)+e_{x,2}(g)+e_{y,2}(g)+e_{x,y,3}(g)-1 \pmod 2.
\end{equation}

\shu{
For $z\in\{x,y\}$, let $\epsilon_z$ be the unique character of $\PP^\circ_{\Lambda,\Lambda}$ such that $\epsilon_z|_{\PP^\circ(\Lambda_\beta)} = \chi^\circ_z|_{\PP^\circ(\Lambda_\beta)}$ and
$\epsilon_z|_{\PP_1(\Lambda)}$ is trivial.}
We claim that
\begin{equation}\label{C:nu}
\nu^\circ_{x,y}=\epsilon_x(\epsilon_y)^{-1}
\qquad\text{on } \JJ^\circ_{\Lambda,\Lambda_y}.
\end{equation}
Indeed, we define
\[
\delta\coloneqq \nu^\circ_{x,y}\cdot \bigl(\epsilon_x(\epsilon_y)^{-1}\bigr)^{-1}
\]
on $\JJ^\circ_{\Lambda,\Lambda_y}$.
Since $\nu^\circ_{x,y}$, $\epsilon_x$ and $\epsilon_y$ are all trivial on $\JJ^1_{\Lambda_y}$, the character $\delta$ factors through a quadratic character
\[
\bar\delta \colon \JJ^\circ_{\Lambda,\Lambda_y}/\JJ^1_{\Lambda_y} \longto \{\pm1\}.
\]
To prove the claim, it amounts to proving that $\bar\delta$ is trivial.
For $z\in \{x,y\}$, we put
\[
	\mathbf{H}^\circ_{z,\beta}\coloneqq \PP^\circ(\Lambda_{z,\beta})/\PP_1(\Lambda_{z,\beta})  \cong \JJ^\circ_{\Lambda_z}/\JJ^1_{\Lambda_z}, \quad \text{and} \quad \mathbf{B}^\circ_{z,\beta} \coloneqq \PP^\circ(\mathcal{C})/\PP_1(\Lambda_{z,\beta}).
\]
Then $\mathbf{B}^\circ_{z,\beta}$ is a Borel subgroup of  $\mathbf{H}^\circ_{z,\beta}$.
Notice that $\PP^\circ_{\Lambda,y}$ is a parabolic subgroup of $\mathbf{H}^\circ_{y,\beta}$.
For $z\in\{x,y\}$, by definition of $\chi^\circ_{z,1}$ we have
\begin{equation}\label{E:chi1}
	\shu{\chi^\circ_{z,1}(g)=(-1)^{e_{z,1}(g)-1}.}
\end{equation}
The image $g_{z,\beta}$ of $g$ in
$
\mathbf{H}^\circ_{z,\beta}$
lies in the Borel subgroup
$\mathbf{B}^\circ_{z,\beta}$, 
hence Construction~\ref{C:chi} gives
\begin{equation}\label{E:chi2}
	\shu{\chi^\circ_{z,2}(g)=(-1)^{e_{z,2}(g)-1}.}
\end{equation}

We next treat the factor \shu{$\chi^\circ_{z,3}$.}
For $z\in \{x,y\}$, let $\mathbf{H}^\circ_z \coloneqq \PP^\circ(\Lambda_z)/\PP_1(\Lambda_z)$.
Let
$\iota_z \colon   \mathbf{H}^\circ_{z,\beta}  \hookrightarrow  \mathbf{H}^\circ_z$ be the canonical embedding,
and let $g_z \coloneqq \iota_z(g_{z,\beta})$ be the image of $g$ in $\mathbf{H}^\circ_z$.
Put
\[
\mathbf{Q} \coloneqq \PP^\circ(\Lambda_y)/\PP_1(\Lambda_x)\subseteq \mathbf{H}^\circ_x.
\]
Then $\mathbf{Q}$ is a parabolic subgroup of $\mathbf{H}^\circ_x$, and the natural quotient map
\[
\pi \colon  \mathbf{Q} \longto \mathbf{H}^\circ_y
\]
has kernel
\[
\mathbf{N}\coloneqq \PP_1(\Lambda_y)/\PP_1(\Lambda_x)=R_u(\mathbf{Q}).
\]
Since
$g\in \PP^\circ(\mathcal{C})\subseteq \PP^\circ(\Lambda_{y,\beta})\subseteq \PP^\circ(\Lambda_y)$,
the element $g_x$ lies in $\mathbf{Q}$, and we have $\pi(g_x)=g_y$.
By Corollary~\ref{C:sign-parabolic} we have
\[	\chi_{\mathbf{H}^\circ_x} (g_x) = \sgn (\sigma_{g_x} | \mathbf{N}) \chi_{\mathbf{H}^\circ_y}(\pi(g_x))
					=\sgn (\sigma_{g_x} | \mathbf{N}) \chi_{\mathbf{H}^\circ_y}(g_y).
\]
Since \shu{$\chi^\circ_{z,3}(g) =\chi_{\mathbf{H}^\circ_z} (g_z)$} for $z\in \{x,y\}$,  and $\sgn (\sigma_{g_x} | \mathbf{N}) = (-1)^{\,|\mathbf{N}/\langle g\rangle|-1}$, we have
\begin{equation}\label{E:chi3}
	\shu{\chi^\circ_{x,3}(\chi^\circ_{y,3})^{-1}(g)}
=
(-1)^{\,|\mathbf{N}/\langle g\rangle|-1}
=
(-1)^{\,\left|(\PP_1(\Lambda_x)\backslash \PP_1(\Lambda_y))/\langle g\rangle\right|-1}
=(-1)^{e_{x,y,3}(g)-1}.
\end{equation}

Combining~\eqref{E:chi1}, \eqref{E:chi2} and \eqref{E:chi3}, we obtain
\begin{align*}
	\shu{\epsilon_x (\epsilon_y)^{-1}(g) }
 =(-1)^{e_{x,1}(g)+e_{y,1}(g)+e_{x,2}(g)+e_{y,2}(g)+e_{x,y,3}(g)-1};					     
\end{align*}
by~\eqref{E:exy-final}, the RHS equals $(-1)^{e_{x,y}(g)}$, and therefore
\[
	\shu{\epsilon_x(g)(\epsilon_y)^{-1}(g) }
=
\nu^\circ_{x,y}(g).
\]
This proves the claim~\eqref{C:nu}.

\shu{We now compute}
\begin{align*}
\ind^{\PP^\circ_{\Lambda,\Lambda}}_{\JJ^\circ_{\Lambda, \Lambda_x}}  \left(\Res^{\JJ^\circ_{\Lambda_x}}_{\JJ^\circ_{\Lambda, \Lambda_x}} \hat\kappa^\circ_x \right) &= \ind^{\PP^\circ_{\Lambda,\Lambda}}_{\JJ^\circ_{\Lambda, \Lambda_x}} \left(\Res^{\JJ^\circ_{\Lambda_x}}_{\JJ^\circ_{\Lambda, \Lambda_x}} \left( \kappa_x|_{\JJ^\circ_{\Lambda_x}} \otimes \epsilon_x|_{\JJ^\circ_{\Lambda_x}} \right) \right)
	\cong \ind^{\PP^\circ_{\Lambda,\Lambda}}_{\JJ^\circ_{\Lambda, \Lambda_x}} \kappa^\circ_{\Lambda,\Lambda_x} \otimes \epsilon_x\\
   & \cong \ind^{\PP^\circ_{\Lambda,\Lambda}}_{\JJ^\circ_{\Lambda, \Lambda_y}} \Psi^\circ_{\Lambda,\Lambda_x,\Lambda_y} \left(\kappa^\circ_{\Lambda, \Lambda_x} \right)\otimes \epsilon_x
   \cong \ind^{\PP^\circ_{\Lambda,\Lambda}}_{\JJ^\circ_{\Lambda,\Lambda_y}} \left(\kappa^\circ_{\Lambda, \Lambda_y}  \otimes \nu^\circ_{x,y} \right) \otimes \epsilon_x \\
   & \cong \ind^{\PP^\circ_{\Lambda,\Lambda}}_{\JJ^\circ_{\Lambda,\Lambda_y}} \left(\kappa^\circ_{\Lambda, \Lambda_y}  \otimes \nu^\circ_{x,y} \otimes \epsilon_x|_{\JJ^\circ_{\Lambda,\Lambda_y}} \right) \cong  \ind^{\PP^\circ_{\Lambda,\Lambda}}_{\JJ^\circ_{\Lambda,\Lambda_y}} (\kappa^\circ_{\Lambda, \Lambda_y} \otimes \epsilon_y|_{\JJ^\circ_{\Lambda,\Lambda_y}}) \\
   & = \ind^{\PP^\circ_{\Lambda,\Lambda}}_{\JJ^\circ_{\Lambda,\Lambda_y}} \Res^{\JJ^\circ_{\Lambda_y}}_{\JJ^\circ_{\Lambda,\Lambda_y}} \hat{\kappa}^\circ_y,
\end{align*}
which implies that
\[
\Psi^\circ_{\Lambda,\Lambda_x,\Lambda_y}
\left(\Res^{\JJ^\circ_{\Lambda_x}}_{\JJ^\circ_{\Lambda,\Lambda_x}}\hat\kappa^\circ_x\right)
\cong
\Res^{\JJ^\circ_{\Lambda_y}}_{\JJ^\circ_{\Lambda,\Lambda_y}}\hat\kappa^\circ_y.
\]
\end{proof}

\begin{thm}\label{T:compatible-beta-extension}
 $\{ \hat{\kappa}^\circ_x \mid x\in\overline{\mathcal{C}}\}$ is a compatible family of $\beta$-extensions.
\end{thm}

\begin{proof}
	By construction, we have that $\hat\kappa^\circ_{\mathsf{M}}\in\beta$-$\ext^\circ(\Lambda_{\mathsf{M}})$ for all vertices $\Lambda_{\mathsf{M}}$ of ${\mathcal{C}}$.
Let $x\in \overline{\mathcal{C}}$ and let $\Lambda_{\mathsf{M}}$ be a vertex such that $\tilde{\mathfrak{b}}(\Lambda_x) \subseteq \tilde{\mathfrak{b}}(\Lambda_{\mathsf{M}})$.
By Proposition~\ref{P:compatible-beta-extension}, we have
\[
	\Psi^\circ_{\Lambda_x,\Lambda_{\mathsf{M}},\Lambda_x} \Res^{\JJ^\circ_{\Lambda_{\mathsf{M}}}}_{\JJ^\circ_{\Lambda_x,\Lambda_{\mathsf{M}}}} \hat\kappa^\circ_{\mathsf{M}} \cong \hat{\kappa}^\circ_x.
\]
By~\eqref{E:kappaM} we have that $\hat\kappa^\circ_{\mathsf{M}} = \Res^{\JJ_{\Lambda_{\mathsf{M}}}}_{\JJ^\circ_{\Lambda_{\mathsf{M}}}} \hat{\kappa}_{\mathsf{M}}$ for some $\hat{\kappa}_{\mathsf{M}}\in\beta\text{-}\ext(\Lambda_{\mathsf{M}})$.
We thus have $\hat{\kappa}^\circ_x \cong \Psi^\circ_{\Lambda_x,\Lambda_{\mathsf{M}},\Lambda_x} \Res^{\JJ_{\Lambda_x,\Lambda_{\mathsf{M}}}}_{\JJ^\circ_{\Lambda_x,\Lambda_{\mathsf{M}}}} \Res^{\JJ_{\Lambda_{\mathsf{M}}}}_{\JJ_{\Lambda_x,\Lambda_{\mathsf{M}}}}\hat\kappa_{\mathsf{M}}$.
By Lemma~\ref{L:Psi-circ}, we have 
$$\Psi^\circ_{\Lambda_x,\Lambda_{\mathsf{M}},\Lambda_x} \Res^{\JJ_{\Lambda_x,\Lambda_{\mathsf{M}}}}_{\JJ^\circ_{\Lambda_x,\Lambda_{\mathsf{M}}}}= \Res^{\JJ_{\Lambda_x}}_{\JJ^\circ_{\Lambda_x}} \Psi_{\Lambda_x,\Lambda_{\mathsf{M}},\Lambda_x},$$
and therefore
\[
\hat{\kappa}^\circ_x \cong \Res^{\JJ_{\Lambda_x}}_{\JJ^\circ_{\Lambda_x}} \Psi_{\Lambda_x,\Lambda_{\mathsf{M}},\Lambda_x} \left(\Res^{\JJ_{\Lambda_{\mathsf{M}}}}_{\JJ_{\Lambda_x,\Lambda_{\mathsf{M}}}}\hat\kappa_{\mathsf{M}} \right),
\] 
which implies that $\hat{\kappa}^\circ_x\in \beta$-$\ext^\circ_{\Lambda_{\mathsf{M}}}(\Lambda_x)$ is a $\beta$-extension.

Now let $\Lambda_{1}$ and $\Lambda_{2}$ be two vertices of ${\mathcal{C}}$ satisfying $\tilde{\mathfrak{b}}(\Lambda_x) \subseteq \tilde{\mathfrak{b}}(\Lambda_{1}) \cap \tilde{\mathfrak{b}}(\Lambda_{2})$.
It follows from Proposition~\ref{P:compatible-beta-extension} that
\[
\Psi^\circ_{\Lambda_x,\Lambda_1, \Lambda_x} \left(\Res^{\JJ^\circ_{\Lambda_1}}_{\JJ^\circ_{\Lambda_x, \Lambda_1}}\hat{\kappa}^\circ_1 \right)
\cong \Psi^\circ_{\Lambda_x,\Lambda_2, \Lambda_x} \circ \Psi^\circ_{\Lambda_x,\Lambda_1, \Lambda_2} \left(\Res^{\JJ^\circ_{\Lambda_1}}_{\JJ^\circ_{\Lambda_x, \Lambda_1}}\hat{\kappa}^\circ_1 \right)
\cong \Psi^\circ_{\Lambda_x,\Lambda_2, \Lambda_x} \left(\Res^{\JJ^\circ_{\Lambda_2}}_{\JJ^\circ_{\Lambda_x, \Lambda_2}}\hat{\kappa}^\circ_2 \right),
\]
which implies that $\{ \hat{\kappa}^\circ_x \mid x\in\overline{\mathcal{C}}\}$ is a compatible family of $\beta$-extensions.
\end{proof}

\subsection{General linear groups}

\shu{
In this subsection, we consider $\tilde{G} = \GL_D(V)$.
We fix a semisimple element $\beta$ with negative critical exponent in the Lie algebra of $\tilde{G}$ and an arbitrary chamber $\mathcal{C}$ in $\mathscr{B}(\tilde{G}_\beta)$.
We fix an $\mathfrak{o}_E$-$\mathfrak{o}_D$-lattice sequence $\Upsilon$ corresponding to a point in $\mathcal{C}$, together with a semisimple character $\tilde{\theta}_{\Upsilon}\in \tilde{\CCC}(\Upsilon, 0, \beta)$.
We denote by $\tilde{\eta}_{\Upsilon}$ the Heisenberg representation containing $\tilde{\theta}_{\Upsilon}$.
}

\shu{
For any $\mathfrak{o}_E$-$\mathfrak{o}_D$-lattice sequence $\Lambda$ corresponding to a point in $\overline{\mathcal{C}}$, there is a unique $\tilde{\theta}_{\Lambda} \in \tilde{\CCC}(\Lambda, 0, \beta)$ which is the transfer of $\tilde{\theta}_{\Upsilon}$.
We denote by $\tilde{\eta}_{\Lambda}$ the Heisenberg representation containing $\tilde{\theta}_{\Lambda}$.
Following~\cite[\S 2.4]{repII}, we call a representation $\tilde{\kappa}_\Lambda$ of $\tJJ_{\Lambda}$ a \emph{$\beta$-extension} if $\tilde{\kappa}_\Lambda$ is an extension of $\tilde{\eta}_\Lambda$ with $\operatorname{I}_{\tilde{G}} (\tilde{\kappa}_{\Lambda}) = \tJJ_{\Lambda} \tilde{G}_\beta \tJJ_{\Lambda}$. 
We denote by $\beta\text{-}\widetilde{\ext}(\Lambda)$ the set of $\beta$-extensions.
}

\shu{
	For any $\mathfrak{o}_E$-$\mathfrak{o}_D$-lattice sequences $\Lambda$ and $\Lambda'$ with $\tilde{\mathfrak{b}}(\Lambda) \subseteq \tilde{\mathfrak{b}}(\Lambda')$, we denote by $\widetilde{\ext}(\Lambda,\Lambda')$ the set of isomorphism classes of extensions of $\tilde{\eta}_{\Lambda,\Lambda'}$ from $\tJJ^1_{\Lambda,\Lambda'}$ to $\tJJ_{\Lambda,\Lambda'}$, where $\tilde{\eta}_{\Lambda,\Lambda'}$ is defined \emph{mutatis mutandis} to~\cite[Definition 4.8]{daniel-3}.
}

\begin{lem}\label{L:beta-conjugation}
Suppose that $\tilde{\mathfrak{a}}(\Lambda) \subseteq \tilde{\mathfrak{a}}\left(\Lambda^{\prime}\right) \cap \tilde{\mathfrak{a}}\left(\Lambda^{\prime \prime}\right) \in\left\{\tilde{\mathfrak{a}}\left(\Lambda^{\prime}\right), \tilde{\mathfrak{a}}\left(\Lambda^{\prime \prime}\right)\right\}$.
Take $\tilde{\kappa}_{\cong}^{\prime} \in \tdext\left(\Lambda, \Lambda^{\prime}\right)$.
For any $\mathbf{r}\in \tilde{\PP}(\Lambda_\beta)$, we have $\leftidx{}{^\mathbf{r}}{\left(\Psi_{\Lambda,\Lambda',\Lambda''} (\tilde{\kappa}')\right)}\cong \Psi_{\Lambda,\mathbf{r}\cdot\Lambda',\mathbf{r}\cdot\Lambda''} (\leftidx{}{^\mathbf{r}}{ \tilde{\kappa}'})$.
\end{lem}
\begin{proof}

We have $\leftidx{}{^\mathbf{r}}{\tilde{\PP}_{\Lambda,\Lambda}}=\tilde{\PP}_{\Lambda,\Lambda}, \leftidx{}{^\mathbf{r}}{\tilde{\JJ}_{\Lambda, \Lambda^{\prime}}}= \tilde{\JJ}_{\Lambda, \mathbf{r}\cdot\Lambda^{\prime}}$ and $\leftidx{}{^\mathbf{r}}{\tilde{\JJ}_{\Lambda, \Lambda^{''}}}= \tilde{\JJ}_{\Lambda, \mathbf{r}\cdot\Lambda^{''}}$.
By~\cite[Lemma 6.4]{daniel-3}, we have
$$
\ind_{\tilde{\JJ}_{\Lambda, \Lambda'}}^{\tilde{\PP}_{\Lambda, \Lambda}} \tilde{\kappa}' \cong \ind_{\tilde{\JJ}_{\Lambda, \Lambda^{\prime \prime}}}^{\tilde{\PP}_{\Lambda, \Lambda}} \Psi_{\Lambda,\Lambda',\Lambda''}(\tilde{\kappa}').
$$
Conjugating both sides by $\mathbf{r}$ yields
\[
	\ind_{\tilde{\JJ}_{\Lambda, \mathbf{r}\cdot\Lambda'}}^{\tilde{\PP}_{\Lambda,\Lambda}} \leftidx{}{^\mathbf{r}}{\tilde{\kappa}'} \cong
\leftidx{}{^\mathbf{r}}{\ind_{\tilde{\JJ}_{\Lambda, \Lambda^{\prime}}}^{\tilde{\PP}_{\Lambda, \Lambda}} \tilde{\kappa}^{\prime}} \cong
\leftidx{}{^\mathbf{r}}{\ind_{\tilde{\JJ}_{\Lambda, \Lambda^{\prime \prime}}}^{\tilde{\PP}_{\Lambda, \Lambda}} \Psi_{\Lambda,\Lambda',\Lambda''}(\tilde{\kappa}')} \cong
\ind_{\tilde{\JJ}_{\Lambda, \mathbf{r}\cdot\Lambda''}}^{\tilde{\PP}_{\Lambda, \Lambda}} \leftidx{}{^\mathbf{r}}{\Psi_{\Lambda,\Lambda',\Lambda''}(\tilde{\kappa}')}.
\]
The lemma then follows again from~\cite[Lemma 6.4]{daniel-3}.
\end{proof}

\begin{nota}
Suppose that $\{\sx_i\}_{i=0}^{m-1}$ is a set of points in $\mathscr{B}(\tilde{G}_\beta)$ such that their reductions $[\sx_0], \cdots, [\sx_{m-1}]$  are the vertices of a chamber $\mathcal{C}$ in $\mathscr{B}_{\operatorname{red}}(\tilde{G}_\beta)$.
For  $i=0, \cdots, m-1$, we denote by $\Lambda_i$ the $\mathfrak{o}_E$-$\mathfrak{o}_D$-lattice sequence corresponding to $j_\beta(\sx_i)$,
and let $\Delta_i = [\Lambda_i, n_i, 0, \beta]$ be a self-duel semisimple stratum.
By~\cite[Lemma 9.5]{daniel-3}, there exists a unique $\tilde{\kappa}_i \in \beta$-$\tdext(\Lambda_i)$ satisfying property \textbf{(ORD)}-$p$ .
Let $x \in \mathscr{B}(\tilde{G}_\beta)$ such that $[x]$ is the barycentre of $\mathcal{C}$ and let $\Lambda^{\operatorname{bar}}$ be the $\mathfrak{o}_E$-$\mathfrak{o}_D$-lattice sequence corresponding to $j_\beta(x)$.
\end{nota}

\subsubsection{Case $\beta$ is simple}
We first assume that $\beta$ is simple, \shu{then $E= F[\beta]$ is a field and $\tilde{G}_\beta \cong \GL_{D_\beta}(V_\beta) \cong \GL_m(D_\beta)$ for some skewfield $D_\beta$ with centre $E$, uniformizer~$\varpi_\beta$, valuation ring~$\mathfrak{o}_\beta$ and some right-$D_\beta$-vector space $V_\beta$ of dimension $m$.}

\begin{lem}\label{L:simple-rotation}
There exist $\sx_0,\cdots,\sx_{m-1}\in \mathscr{B}(\tilde{G}_\beta)$ such that, for any $0 \leq i,j \leq m-1$, there exists $\mathbf{r}\in \mathfrak{n}(\Lambda^{\operatorname{bar}}_{\beta})$ such that $\mathbf{r}\Lambda_i =\Lambda_j$.
\end{lem}
\begin{proof}
We choose an $D_\beta$-basis $e_1,\cdots, e_m$ for $V_\beta$.
We first assume that $\mathcal{C}= \mathcal{C}_0$ is the fundamental chamber in $\mathscr{B}_{\operatorname{red}}(\tilde{G}_\beta)$ 
such that
$\Lambda_{i,\beta}(k) = \varpi_\beta^k \mathcal{L}_i$ for all $k\in \ZZ$, where
\[
	\mathcal{L}_i = e_1 \varpi_\beta\mathfrak{o}_\beta  \oplus \cdots \oplus  e_i \varpi_\beta\mathfrak{o}_\beta \oplus e_{i+1}\mathfrak{o}_\beta \oplus  \cdots \oplus e_m \mathfrak{o}_\beta,   \quad \quad \quad \quad   (0\leq i \leq m-1).
\]
Consider the $m$ by  $m$ matrix
$
\Pi_m \coloneqq 
\begin{psmallmatrix}
	0 & I_{m-1} \\
\varpi_E & 0 
\end{psmallmatrix}\in \tilde{\PP}(\Lambda^{\operatorname{bar}}_{\beta})$, where $I_{m-1}$ denotes the $(m-1)$ by $(m-1)$ identity matrix.
Put $\mathbf{r}_0 \coloneqq (\Pi_m)^{j-i}$.
It is straightforward to compute that $\mathbf{r}_0 \cdot \Lambda_{i,\beta} = \Lambda_{j,\beta}$, which implies  $\mathbf{r}_0 \cdot \Lambda_{i} = \Lambda_{j}$ since $j_\beta$ is $\tilde{G}_\beta$-equivariant.

We now assume that $\mathcal{C}$ is arbitrary.
Notice that $\mathcal{C}$ and $\mathcal{C}_0$ are contained in a common apartment, then
by the strong transitivity, we find and element $g\in \tilde{G}_\beta$ such that $g \mathcal{C} =\mathcal{C}_0$.
Then $\mathbf{r} \coloneqq g^{-1} \mathbf{r}_0 g$ satisfies the desired properties.
\end{proof}

\begin{lem}\label{L:simple-compatibility}
	Let $\sx_0,\cdots, \sx_{m-1}$ be points in $\mathscr{B}(\tilde{G}_\beta)$ provided by Lemma~\ref{L:simple-rotation}.
We have $$\Psi_{\Lambda^{\operatorname{bar}},\Lambda_i, \Lambda^{\operatorname{bar}}}\left(\Res^{\tilde{\JJ}_{\Lambda_i}}_{\tilde{\JJ}_{\Lambda^{\operatorname{bar}},\Lambda_i}} \tilde{\kappa}_i \right) \cong \Psi_{\Lambda^{\operatorname{bar}},\Lambda_j, \Lambda^{\operatorname{bar}}}\left(\Res^{\tilde{\JJ}_{\Lambda_j}}_{\tilde{\JJ}_{\Lambda^{\operatorname{bar}},\Lambda_j}} \tilde{\kappa}_j \right).$$
\end{lem}

\begin{proof}
Put $\tilde{\kappa} \coloneqq \Psi_{\Lambda^{\operatorname{bar}},\Lambda_i, \Lambda^{\operatorname{bar}}}\left(\Res^{\tilde{\JJ}_{\Lambda_i}}_{\tilde{\JJ}_{\Lambda^{\operatorname{bar}},\Lambda_i}} \tilde{\kappa}_i \right)$.
We first notices that $\leftidx{}{^\mathbf{r}}{\tilde{\kappa}_i}$ is again a $\beta$-extension in $\beta$-$\tdext(\Lambda_j)$ satisfying property \textbf{(ORD-$p$)}, where $\mathbf{r}$ is given by Lemma~\ref{L:simple-rotation}. 
By uniqueness we deduce that $\leftidx{}{^\mathbf{r}}{\tilde{\kappa}_i} \cong \tilde{\kappa}_j$.
We choose a segment $\mathcal{S}$ from $\Lambda_i$ to $\Lambda$ in the building $\mathscr{B}(\tilde{G}_\beta)$ of $\tilde{G}_\beta$ and pairwise different points
$$
\Lambda_i=\Lambda_{(0)}, \Lambda_{(1)},  \ldots, \Lambda_{(u-1)}, \Lambda_{(u)}=\Lambda^{\operatorname{bar}}
$$
on the segment such that for all indexes $s \in\{1, \ldots, u\}$ the condition
$$
\tilde{\mathfrak{a}}\left(\Lambda_{(s-1)}\right) \cap \tilde{\mathfrak{a}}\left(\Lambda_{(s)}\right) \in\left\{\tilde{\mathfrak{a}}\left(\Lambda_{(s-1)}\right), \tilde{\mathfrak{a}}\left(\Lambda_{(s)}\right)\right\}
$$
is satisfied. 
Now $\mathbf{r}\cdot \mathcal{S}$ is a segment from $\Lambda_j=\mathbf{r}\cdot \Lambda_i$ to $\mathbf{r}\cdot\Lambda^{\operatorname{bar}}$.
We have that $\mathbf{r}\cdot\Lambda^{\operatorname{bar}}$ is a translate of $\Lambda^{\operatorname{bar}}$ since $\mathbf{r}\in \mathfrak{n}(\Lambda^{\operatorname{bar}})$, and
$$
\Lambda_j=\mathbf{r}\cdot\Lambda_{(0)}, \mathbf{r}\cdot\Lambda_{(1)}, \ldots, \mathbf{r}\cdot\Lambda_{(u-1)}, \mathbf{r}\cdot\Lambda_{(u)}=\mathbf{r}\cdot\Lambda^{\operatorname{bar}}
$$
are pairwise different points satisfying the desired condition.
For any $s \in\{1, \ldots, u\}$ and any $\tilde{\kappa}_{(s-1)}\in \beta$-$\tdext(\Lambda^{\operatorname{bar}}, \Lambda_{(s-1)})$, we deduce from Lemma~\ref{L:beta-conjugation} that
\[
	\leftidx{}{^\mathbf{r}}{\left(\Psi_{\Lambda^{\operatorname{bar}},\Lambda_{(s-1)},\Lambda_{(s)}}} \left(\tilde{\kappa}_{(s-1)}\right)\right) \cong \Psi_{\Lambda^{\operatorname{bar}},\mathbf{r}\cdot \Lambda_{(s-1)},\mathbf{r}\cdot \Lambda_{(s)}}  \left(\leftidx{}{^\mathbf{r}}{\tilde{\kappa}_{(s-1)}}\right).
\] 
It then follows that $\leftidx{}{^\mathbf{r}}{\left(\Psi_{\Lambda^{\operatorname{bar}},\Lambda_i,\Lambda^{\operatorname{bar}}}} \left(\Res^{\tilde{\JJ}_{\Lambda_i}}_{\tilde{\JJ}_{\Lambda^{\operatorname{bar}},\Lambda_{i}}}\tilde{\kappa}_i\right)\right) \cong \Psi_{\Lambda^{\operatorname{bar}}, \mathbf{r}\cdot \Lambda_i, \Lambda^{\operatorname{bar}}} \left( \Res^{\tilde{\JJ}_{\mathbf{r}\cdot\Lambda_i}}_{\tilde{\JJ}_{\Lambda^{\operatorname{bar}},\mathbf{r}\cdot\Lambda_{i}}}\leftidx{}{^\mathbf{r}}{\tilde{\kappa}_i}\right)$; and the latter is isomorphic to $\Psi_{\Lambda^{\operatorname{bar}}, \Lambda_j, \Lambda^{\operatorname{bar}}}\left( \Res^{\tilde{\JJ}_{\Lambda_j}}_{\tilde{\JJ}_{\Lambda^{\operatorname{bar}},\Lambda_{j}}} \tilde{\kappa}_j\right)$ since $\mathbf{r}\cdot \Lambda_i = \Lambda_j$ and $\leftidx{}{^\mathbf{r}}{\tilde{\kappa}_i} \cong \tilde{\kappa}_j$. 
We thus have
\[
\tilde{\kappa} \cong \leftidx{}{^\mathbf{r}}{\tilde{\kappa}} =  \leftidx{}{^\mathbf{r}}{\left(\Psi_{\Lambda^{\operatorname{bar}},\Lambda_i,\Lambda^{\operatorname{bar}}}} \left(\Res^{\tilde{\JJ}_{\Lambda_i}}_{\tilde{\JJ}_{\Lambda^{\operatorname{bar}},\Lambda_i}} \tilde{\kappa}_i \right)\right) \cong \Psi_{\Lambda^{\operatorname{bar}}, \Lambda_j, \Lambda^{\operatorname{bar}}}\left(\Res^{\tilde{\JJ}_{\Lambda_j}}_{\tilde{\JJ}_{\Lambda^{\operatorname{bar}},\Lambda_j}} \tilde{\kappa}_j \right),
\] 
as desired.
\end{proof}

\begin{prop}\label{P:simple-compatibility}
Under the conditions of Lemma~\ref{L:simple-compatibility} and suppose that $x \in \mathscr{B}(\tilde{G}_\beta)$ such that $[x]$ lies in a facet whose closure contains $[\sx_i]$ and $[\sx_j]$. 
Let $\Lambda_x$ be the $\mathfrak{o}_E$-$\mathfrak{o}_D$-lattice sequence corresponding to $j_\beta(x)$.
Then 
\[
\Psi_{\Lambda_x,\Lambda_i, \Lambda_x} \left(\Res^{\tilde{\JJ}_{\Lambda_i}}_{\tilde{\JJ}_{\Lambda_x,\Lambda_i}} \tilde{\kappa}_i \right) \cong \Psi_{\Lambda_x,\Lambda_j, \Lambda}\left(\Res^{\tilde{\JJ}_{\Lambda_j}}_{\tilde{\JJ}_{\Lambda_x,\Lambda_j}} \tilde{\kappa}_j \right).
\] 
\end{prop}

\begin{proof}
	Notice that $\tilde{\PP}(\Lambda_{x,\beta}) \cup \tilde{\PP}( \Lambda^{\operatorname{bar}}_{\beta}) \subseteq \tilde{\PP}(\Lambda_{k,\beta})$ for $k=i,j$.
By replacing $\Lambda'$ with $\Lambda_x$ and  $\Lambda''$ by $\Lambda^{\operatorname{bar}}$ in~\cite[Theorem 6.9]{daniel-3}, we obtain a bijective map $\Psi_{\Lambda_x,\Lambda^{\operatorname{bar}}}$ from $\beta$-$\tdext_{\Lambda_k}(\Lambda_x)$ to  $\beta$-$\tdext_{\Lambda_k}(\Lambda^{\operatorname{bar}})$ satisfying
\[
	\Psi_{\Lambda_x,\Lambda^{\operatorname{bar}}} \circ \Psi_{\Lambda_x, \Lambda_k, \Lambda_x} \circ \Res^{\tilde{\JJ}_{\Lambda_k}}_{\tilde{\JJ}_{\Lambda_x,\Lambda_k}} = \Psi_{\Lambda^{\operatorname{bar}}, \Lambda_k, \Lambda^{\operatorname{bar}}} \circ \Res^{\tilde{\JJ}_{\Lambda_k}}_{\tilde{\JJ}_{\Lambda^{\operatorname{bar}},\Lambda_k}}  \quad\quad\quad\quad  (k=i,j).
\]
By Lemma~\ref{L:simple-compatibility}, we have $\Psi_{\Lambda^{\operatorname{bar}}, \Lambda_i, \Lambda^{\operatorname{bar}}} \circ \Res^{\tilde{\JJ}_{\Lambda_i}}_{\tilde{\JJ}_{\Lambda^{\operatorname{bar}},\Lambda_i}} (\tilde{\kappa}_i) = \Psi_{\Lambda^{\operatorname{bar}}, \Lambda_j, \Lambda^{\operatorname{bar}}} \circ \Res^{\tilde{\JJ}_{\Lambda_j}}_{\tilde{\JJ}_{\Lambda^{\operatorname{bar}},\Lambda_j}}(\tilde{\kappa}_j)$, and therefore

$$\Psi_{\Lambda_x,\Lambda^{\operatorname{bar}}} \circ \Psi_{\Lambda_x, \Lambda_i, \Lambda_x} \circ \Res^{\tilde{\JJ}_{\Lambda_i}}_{\tilde{\JJ}_{\Lambda_x,\Lambda_i}} (\tilde{\kappa}_i) =
\Psi_{\Lambda_x,\Lambda^{\operatorname{bar}}} \circ \Psi_{\Lambda_x, \Lambda_j, \Lambda_x} \circ \Res^{\tilde{\JJ}_{\Lambda_j}}_{\tilde{\JJ}_{\Lambda_x,\Lambda_j}} (\tilde{\kappa}_j).$$
The proposition follows from the injectivity of $\Psi_{\Lambda_x,\Lambda^{\operatorname{bar}}}$.

\end{proof}

\subsubsection{Case $\beta$ is semisimple}

Suppose that $\beta$ is semisimple. Then $E = F[\beta]$ decomposes as a product of fields $E = \prod\limits_{i=1}^k E_i$ and $\tilde{G}_\beta$ decomposes as
\shu{$\tilde{G}_\beta \cong \prod\limits_{i=1}^k \GL_{m_i}(D^i_{\beta})$  for some $m_i\in\ZZ$ and some skewfields $D^i_{\beta}$ with centre $E_i$.}
Moreover, \shu{for any chamber $\mathcal{C}$ in $\mathscr{B}_{\operatorname{red}}(\tilde{G}_\beta)$, we have a decomposition $\overline{\mathcal{C}} = \prod\limits_{i=1}^k \overline{\mathcal{C}^i}$.}
Let $V= \bigoplus_{i=1}^k V^i$ be the associated splitting and let $L = \prod\limits_{i=1}^k \tilde{G}^i$ be the associated Levi subgroup of $\tilde{G}$ where $\tilde{G}^i =\Aut_D(V^i)$.
Let $Q$ be a parabolic subgroup of $\tilde{G}$ containing $L$ with unipotent radical $U$ and let $U^-$ be the opposite unipotent radical. 

For any $1 \leq i \leq k$, let $\mathcal{V}^i \coloneqq \{\sx^i_0, \sx^i_1,\cdots \sx^i_{m_i-1}\}  \subseteq \mathscr{B}(\tilde{G}^i_\beta)$ be a set of representatives of vertices of $\mathcal{C}^i$ given by Lemma~\ref{L:simple-rotation}.

\begin{cons}\label{C:kappa-gl}
Fix an arbitrary $x\in \mathscr{B}(\tilde{G}_\beta)$ such that $[x] \in \overline{\mathcal{C}}$ and let $\Lambda_x$ be the  $\mathfrak{o}_E$-$\mathfrak{o}_D$-lattice sequence corresponding to $j_\beta(x)$.
 We then have decompositions $\Lambda_x = \bigoplus\limits_{i=1}^k \Lambda^i_x$ and  $\tJJ_{\Lambda_x,L} \coloneqq \tJJ_{\Lambda_x} \cap L \cong \prod\limits_{j=1}^k \tJJ_{\Lambda_x^i}$.
 For any $1 \leq i \leq k$, there exists $\sx^i_{a}\in \mathcal{V}^i$ such that $\tilde{\mathfrak{b}}(\Lambda_x^i) \subseteq \tilde{\mathfrak{b}}(\Lambda_{x^i_{a}})$, for some $0 \leq a \leq m_i-1$.
Let $\tilde{\kappa}_{x^i_{a}}\in \beta\text{-}\tdext(\Lambda_{x^i_{a}})$ be the unique $\beta$-extension satisfying property \textbf{(ORD)-$p$} and put
 \[
\tilde{\kappa}_{\Lambda^i_x}  \coloneqq \Psi_{\Lambda_x^i,\Lambda_{\sx^i_{a}}, \Lambda_x^i} \left(\Res^{\tilde{\JJ}_{\Lambda_{\sx^i_{a}}}}_{\tilde{\JJ}_{\Lambda_x^i,\Lambda_{\sx^i_{a}}}} \tilde{\kappa}_{\sx^i_{a}} \right).
\] 
By Proposition~\ref{P:simple-compatibility}, $\tilde{\kappa}_{\Lambda^i_x}$ is independent of the choice of  $x^i_{a}\in \mathcal{V}^i$.

We define the representation $\tilde{\kappa}_{\Lambda_x,L} \coloneqq \bigboxtimes\limits_{i=1}^k \tilde{\kappa}_{\Lambda^i_x}$ of $\tJJ_{\Lambda_x,L}$, and then define the representation $\tilde{\kappa}_{\Lambda_x,Q}$ of 
$$\tJJ_{\Lambda_x,Q} =\tHH^1_{\Lambda_x} \left(\tJJ_{\Lambda_x}\cap L \right) \left(\tJJ_{\Lambda_x}\cap U \right)= \left(\tHH^1_{\Lambda_x} \cap U^- \right) \left(\tJJ_{\Lambda_x}\cap L \right) \left(\tJJ^1_{\Lambda_x}\cap U \right)$$
via $\tilde{\kappa}_{\Lambda_x,Q}(hlu) \coloneqq \tilde{\kappa}_{\Lambda_x,L}(l)$ for $h\in  \tHH^1_{\Lambda_x}\cap U^-, l\in \tJJ_{\Lambda_x}\cap L$ and $u\in \tJJ^1_{\Lambda_x}\cap U$.
Put
\[
	\tilde{\kappa}_{\Lambda_x} \coloneqq \ind_{\tJJ_{\Lambda_x,Q}}^{\tJJ_{\Lambda_x}} \tilde{\kappa}_{\Lambda_x,Q}.
\] 

\end{cons}

For an $\mathfrak{o}_E$-$\mathfrak{o}_D$-lattice sequences $\Lambda$ with $\tilde{\mathfrak{b}}(\Lambda) \subseteq \tilde{\mathfrak{b}}(\Lambda_x)$
we similarly define
\[
	\tJJ_{\Lambda,\Lambda_x,Q} \coloneqq \tHH^1_{\Lambda_x}(\tJJ_{\Lambda,\Lambda_x} \cap Q) = \left( \tHH^1_{\Lambda_x}\cap U^- \right) \left(\tJJ_{\Lambda,\Lambda_x} \cap L \right)\left(\tJJ^1_{\Lambda_x} \cap U  \right)
\] 
Put $\tilde{\kappa}_{\Lambda^i,\Lambda^i_x} \coloneqq \tilde{\kappa}_{\Lambda_x^i}|_{\tJJ_{\Lambda^i,\Lambda^i_x}}$.
We define the representation $\tilde{\kappa}_{\Lambda,\Lambda_x,L} \coloneqq \bigboxtimes\limits_{i=1}^k \tilde{\kappa}_{\Lambda^i,\Lambda^i_x}$ of $\tJJ_{\Lambda,\Lambda_x} \cap L$, then extend it to a representation  $\tilde{\kappa}_{\Lambda,\Lambda_x,Q}$ of $\tJJ_{\Lambda,\Lambda_x,Q}$ by declaring $\tilde{\kappa}_{\Lambda,\Lambda_x,Q}$ to be trivial on $\tHH^1_{\Lambda_x}\cap U^-$ and $(\tJJ^1_{\Lambda_x} \cap U)$.
We then put $\tilde{\kappa}_{\Lambda,\Lambda_x} \coloneqq \ind^{\tJJ_{\Lambda,\Lambda_x}}_{\tJJ_{\Lambda,\Lambda_x,Q}} \tilde{\kappa}_{\Lambda,\Lambda_x,Q}$.

The following proposition shows that $\tilde{\kappa}_{\Lambda_x}$ is indeed a $\beta$-extension.

\begin{prop}
For an $\mathfrak{o}_E$-$\mathfrak{o}_D$-lattice sequences $\Lambda$ with $\tilde{\mathfrak{b}}(\Lambda) \subseteq \tilde{\mathfrak{b}}(\Lambda_x)$, we have
\begin{itemize}
\item[(i)] $\Res^{\tJJ_{\Lambda_x}}_{\tJJ_{\Lambda,\Lambda_x}} \tilde{\kappa}_{\Lambda_x} \cong \tilde{\kappa}_{\Lambda,\Lambda_x} \in \tdext(\Lambda,\Lambda_x)$; in particular, $\tilde{\kappa}_{\Lambda_x}$ is an extension of the Heisenberg representation $\eta_{\Lambda_x}$;
\item[(ii)] $\operatorname{I}_{\tilde{G}} (\tilde{\kappa}_{\Lambda_x}) = \tJJ_{\Lambda_x} \tilde{G}_\beta \tJJ_{\Lambda_x}$.
\end{itemize}

\end{prop}

\begin{proof}
(i)  
We have $\tJJ_{\Lambda,\Lambda_x} = \tJJ^1_{\Lambda,\Lambda_x} \tJJ_{\Lambda,\Lambda_x,Q}$.
Mackey's restriction-induction formula then implies
\begin{align*}
	\Res^{\tJJ_{\Lambda,\Lambda_x}}_{\tJJ^1_{\Lambda,\Lambda_x}} \tilde{\kappa}_{\Lambda,\Lambda_x} = & \Res^{\tJJ_{\Lambda,\Lambda_x}}_{\tJJ^1_{\Lambda,\Lambda_x}}  \ind_{\tJJ_{\Lambda,\Lambda_x,Q}}^{\tJJ_{\Lambda,\Lambda_x}} \tilde{\kappa}_{\Lambda,\Lambda_x,Q}
	\cong \ind_{\tJJ^1_{\Lambda,\Lambda_x,Q}}^{\tJJ^1_{\Lambda,\Lambda_x}} \Res^{\tJJ_{\Lambda,\Lambda_x,Q}}_{\tJJ^1_{\Lambda,\Lambda_x,Q}} \tilde{\kappa}_{\Lambda,\Lambda_x,Q} \\
	\cong & \ind_{\tJJ^1_{\Lambda,\Lambda_x,Q}}^{\tJJ^1_{\Lambda,\Lambda_x}} \tilde{\eta}_{\Lambda,\Lambda_x,Q}
	\cong \tilde{\eta}_{\Lambda,\Lambda_x},
\end{align*}
proving that $\tilde{\kappa}_{\Lambda,\Lambda_x} \in \tdext(\Lambda,\Lambda_x)$.
For the claimed isomorphism, we first notice that $\tJJ_{\Lambda_x} = \tJJ_{\Lambda,\Lambda_x} \tJJ_{\Lambda_x,Q}$,
then apply Mackey's restriction-induction to obtain
\[\Res^{\tJJ_{\Lambda_x}}_{\tJJ_{\Lambda,\Lambda_x}} \tilde{\kappa}_{\Lambda_x} = \Res^{\tJJ_{\Lambda_x}}_{\tJJ_{\Lambda,\Lambda_x}}  \ind_{\tJJ_{\Lambda_x,Q}}^{\tJJ_{\Lambda_x}} \tilde{\kappa}_{\Lambda_x,Q}
	\cong \ind_{\tJJ_{\Lambda,\Lambda_x,Q}}^{\tJJ_{\Lambda, \Lambda_x}} \tilde{\kappa}_{\Lambda,\Lambda_x,Q}
	=\tilde{\kappa}_{\Lambda,\Lambda_x}.
\]

(ii) 
It follows from assertion (i) that
$$
\operatorname{I}_{\tilde{G}} (\tilde{\kappa}_{\Lambda_x}) \subseteq \operatorname{I}_{\tilde{G}} (\tilde{\eta}_{\Lambda_x}) = \tJJ^1_{\Lambda_x} \tilde{G}_\beta \tJJ^1_{\Lambda_x}
\subseteq  \tJJ_{\Lambda_x} \tilde{G}_\beta \tJJ_{\Lambda_x}.
$$
For the inverse inclusion, it suffices to show that $\tilde{G}_\beta \subseteq \operatorname{I}_{\tilde{G}} (\tilde{\kappa}_{\Lambda_x})$.
Let $g = \prod_{i=1}^k g_i \in \tilde{G}_\beta$ with the latter identified with 
$\prod\limits_{i=1}^k \GL_{m_i}(D_{E_i})$.
Then $g$ intertwines $\tilde{\kappa}_{\Lambda_x,L}$ since each $g_i$ intertwines $\tilde{\kappa}_{\Lambda^i_x}$.
It then follows that $g$ intertwines $\tilde{\kappa}_{\Lambda_x,Q}$ (noting $g\in \tilde{G}_\beta \subseteq L$).
By definition, $\tilde{\kappa}_{\Lambda_x} = \ind_{\tJJ_{\Lambda_x,Q}}^{\tJJ_{\Lambda_x}} \tilde{\kappa}_{\Lambda_x,Q}$,
then Frobenius reciprocity and Mackey's restriction-induction formula shows that
$\Hom_{\tJJ_{\Lambda_x} \cap \leftidx{}{^g}{\!\tJJ_{\Lambda_x}}} \left(\tilde{\kappa}_{\Lambda_x}, \leftidx{}{^g}{\tilde{\kappa}_{\Lambda_x}} \right)$ contains a summand $\Hom_{\tJJ_{\Lambda_x,Q} \cap \leftidx{}{^g}{\!\tJJ_{\Lambda_x,Q}}} \left(\tilde{\kappa}_{\Lambda_x,Q}, \leftidx{}{^g}{\tilde{\kappa}_{\Lambda_x,Q}} \right)$ which is non-zero.
Hence, $$\Hom_{\tJJ_{\Lambda_x} \cap \leftidx{}{^g}{\!\tJJ_{\Lambda_x}}} \left(\tilde{\kappa}_{\Lambda_x}, \leftidx{}{^g}{\tilde{\kappa}_{\Lambda_x}} \right) \neq 0,$$ 
i.e., $g$ intertwines $\tilde{\kappa}_{\Lambda_x}$.

\end{proof}

We next show that $\{\tilde{\kappa}_{\Lambda_x} | x\in \overline{\mathcal{C}}\}$ is a compatible family of $\beta$-extensions.

\begin{prop}\label{P:compatible-beta-extension-gl}
For all $\mathfrak{o}_E$-$\mathfrak{o}_D$-lattice sequences $\Lambda, \Lambda_x$ and $\Lambda_y$ satisfying 
$\tilde{\mathfrak{b}}(\Lambda) \subseteq \tilde{\mathfrak{b}}(\Lambda_x) \cap \tilde{\mathfrak{b}}(\Lambda_y)$, we have 
\[
	\Psi_{\Lambda,\Lambda_x,\Lambda_y} \left(\Res^{\tJJ_{\Lambda_x}}_{\tJJ_{\Lambda, \Lambda_x}} \tilde{\kappa}_{\Lambda_x} \right) \cong   \Res^{\tJJ_{\Lambda_y}}_{\tJJ_{\Lambda, \Lambda_y}} \tilde{\kappa}_{\Lambda_y}.
\] 
\end{prop}

\begin{proof}
For any $1 \leq i \leq k$, we choose a path of $\mathfrak{o}_E$-$\mathfrak{o}_D$-lattice sequences
\[
	\Lambda^i_x = \Lambda^i_{(0)}, \Lambda^i_{(1)}, \cdots, \Lambda^i_{(u_i)}= \Lambda^i_y,
\]
such that for all $s \in\{1, \ldots, u_i\}$ one has
\[
\tilde{\mathfrak{a}}\left(\Lambda^i_{(s-1)}\right) \cap \tilde{\mathfrak{a}}\left(\Lambda_{(s)}\right)
\in\left\{\tilde{\mathfrak{a}}\left(\Lambda^i_{(s-1)}\right), \tilde{\mathfrak{a}}\left(\Lambda^i_{(s)}\right)\right\} \quad \text{and} \quad 
\tilde{\mathfrak{b}}\left(\Lambda^i\right)\subseteq \tilde{\mathfrak{b}} \left(\Lambda^i_{(s-1)}\right) \cap \tilde{\mathfrak{b}}\left(\Lambda^i_{(s)}\right).
\]
We form the path
\begin{align*}
&	\Lambda_x= \left(\Lambda^1_x, \Lambda^2_x \cdots,\Lambda^k_x \right) \to \left(\Lambda^1_{(1)}, \Lambda^2_x \cdots,\Lambda^k_x \right) \to \cdots \to \left(\Lambda^1_{(u_1)} =\Lambda^1_y, \Lambda^2_x \cdots,\Lambda^k_x \right) \to \\
& \left(\Lambda^1_y, \Lambda^2_{(1)} \cdots,\Lambda^k_x \right) \to \cdots \to  \left(\Lambda^1_y, \Lambda^2_{(u_2)} = \Lambda^2_y \cdots,\Lambda^k_x \right) \to \cdots \to \left(\Lambda^1_x, \Lambda^2_x \cdots,\Lambda^k_x \right)  =\Lambda_y
\end{align*}
We may thus reduce to the case where
\[
\tilde{\mathfrak{a}}(\Lambda)\subseteq \tilde{\mathfrak{a}}(\Lambda_x)\cap \tilde{\mathfrak{a}}(\Lambda_y)
\in \{\tilde{\mathfrak{a}}(\Lambda_x), \tilde{\mathfrak{a}}(\Lambda_y)\},
\]
as in~\cite[\S 6.3]{daniel-3}, and moreover
\begin{equation}\label{H:path-gl}
\tilde{\mathfrak{a}}(\Lambda^i)\subseteq \tilde{\mathfrak{a}}(\Lambda^i_x)\cap \tilde{\mathfrak{a}}(\Lambda^i_y)
\in \{ \tilde{\mathfrak{a}}(\Lambda^i_x), \tilde{\mathfrak{a}}(\Lambda^i_y)\}, \qquad ( 1 \leq i \leq k).
\end{equation}

For $z\in \{x,y\}$, we write  $\tilde{\kappa}_{\Lambda,\Lambda_z} \coloneqq \tilde{\kappa}_{\Lambda_z}|_{\tJJ_{\Lambda,\Lambda_z}}$.
Put
\[
	\tilde{\rho}_{\Lambda,y} \coloneqq 	\Psi_{\Lambda,\Lambda_x,\Lambda_y} \left(\Res^{\tJJ_{\Lambda_x}}_{\tJJ_{\Lambda, \Lambda_x}}\tilde{\kappa}_{\Lambda_x} \right).
\] 
We define $\tilde{\rho}_{\Lambda,y,Q}$ to be the representation of $\tJJ_{\Lambda,\Lambda_y,Q}$ on the $\tJJ_{\Lambda,\Lambda_y}\cap U$-fixed vectors of $\tilde{\rho}_{\Lambda,\Lambda_y}$.
We then have
$\tilde{\rho}_{\Lambda,\Lambda_y} \cong \ind^{\tJJ_{\Lambda,\Lambda_y}}_{\tJJ_{\Lambda,\Lambda_y,Q}} \tilde{\rho}_{\Lambda,\Lambda_y,Q}$.
We define $\tilde{\rho}_{\Lambda,y,L} \coloneqq \tilde{\rho}_{\Lambda,\Lambda_y,Q}|_{\tJJ_{\Lambda,\Lambda_y,L}}$, then
$\tilde{\rho}_{\Lambda,\Lambda_y,L} \cong \bigboxtimes\limits_{i=1}^k \tilde{\rho}_{\Lambda^i,\Lambda^i_y}$ for some irreducible representations $\tilde{\rho}_{\Lambda^i,\Lambda^i_y}$ of $\tJJ_{\Lambda^i,\Lambda^i_y}$.
We claim that
\[
\tilde{\kappa}_{\Lambda^i,\Lambda^i_y} \cong \tilde{\rho}_{\Lambda^i,\Lambda^i_y}, \qquad \forall 1 \leq i \leq k.
\]
Granted this claim, we have $\tilde{\kappa}_{\Lambda,\Lambda_y,L} \cong \tilde{\rho}_{\Lambda,\Lambda_y,L}$, 
inducing to $\tJJ_{\Lambda,\Lambda_y}$ then implies $\tilde{\kappa}_{\Lambda,\Lambda_y,Q} \cong \tilde{\rho}_{\Lambda,\Lambda_y,Q}$, and hence $\tilde{\kappa}_{\Lambda,\Lambda_y} \cong \tilde{\rho}_{\Lambda,\Lambda_y}$.
Then proposition then follows.

It remains to prove the claim.
Indeed, by the transitivity of induced representations, we have
\[
	\ind^{\tPP_{\Lambda,\Lambda}}_{\tJJ_{\Lambda, \Lambda_x, Q}} \tilde{\kappa}_{\Lambda,\Lambda_x,Q}
	\cong \ind^{\tPP_{\Lambda,\Lambda}}_{\tJJ_{\Lambda, \Lambda_x}} \tilde{\kappa}_{\Lambda,\Lambda_x}
	\cong \ind^{\tPP_{\Lambda,\Lambda}}_{\tJJ_{\Lambda, \Lambda_y}}  \tilde{\rho}_{\Lambda,\Lambda_y}
	\cong \ind^{\tPP_{\Lambda,\Lambda}}_{\tJJ_{\Lambda, \Lambda_y, Q}}  \tilde{\rho}_{\Lambda,\Lambda_y,Q},
\] 
and hence $\Hom_{\tPP_{\Lambda,\Lambda}} \left(	\ind^{\tPP_{\Lambda,\Lambda}}_{\tJJ_{\Lambda, \Lambda_x, Q}} \tilde{\kappa}_{\Lambda,\Lambda_x,Q}, \ind^{\tPP_{\Lambda,\Lambda}}_{\tJJ_{\Lambda, \Lambda_y, Q}}  \tilde{\rho}_{\Lambda,\Lambda_y,Q} \right) \neq 0$.
By Frobenius reciprocity and Mackey's restriction-induction formula, we have
\[
	\bigoplus_{\bar{g} \in \tJJ_{\Lambda, \Lambda_x, Q} \backslash \tPP_{\Lambda,\Lambda} / \tJJ_{\Lambda, \Lambda_y, Q}}
		\Hom_{\tJJ_{\Lambda, \Lambda_x, Q} \cap \leftidx{}{^g}{\!\tJJ_{\Lambda, \Lambda_y, Q}}} \left(\tilde{\kappa}_{\Lambda,\Lambda_x,Q}, \leftidx{}{^g}{\tilde{\rho}_{\Lambda,\Lambda_y,Q}} \right) \neq 0.
\] 
Then there exists $0 \neq \phi \in \Hom_{\tJJ_{\Lambda, \Lambda_x, Q} \cap \leftidx{}{^g}{\!\tJJ_{\Lambda, \Lambda_y, Q}}} \left(\tilde{\kappa}_{\Lambda,\Lambda_x,Q}, \leftidx{}{^g}{\tilde{\rho}_{\Lambda,\Lambda_y,Q}} \right)$ for some $g\in \tPP_{\Lambda,\Lambda}$. 
We first notice that
$
g\in \operatorname{I}_{\tilde{G}} (\tilde{\eta}_{\Lambda,\Lambda_x,Q}, \tilde{\eta}_{\Lambda,\Lambda_y,Q}) = \tJJ^1_{\Lambda,\Lambda_x,Q} \tilde{G}_\beta \tJJ^1_{\Lambda,\Lambda_y,Q}.
$
Because $\tilde{G}_\beta \subseteq L$, we conclude that the representative of $\bar g$ can be chosen in $\tPP_{\Lambda,\Lambda} \cap L$.
We then restrict $\phi$ to $L$ and obtain
$
	\Hom_{\tJJ_{\Lambda, \Lambda_x, L} \cap \leftidx{}{^g}{\tJJ_{\Lambda, \Lambda_y, L}}} \left(\tilde{\kappa}_{\Lambda,\Lambda_x,L}, \leftidx{}{^g}{\tilde{\rho}_{\Lambda,\Lambda_y,L}} \right) \neq 0.
$ 
By Frobenius reciprocity and Mackey's restriction-induction formula again, we obtain
\[
	\Hom_{\tPP_{\Lambda,\Lambda}\cap L} \left(\ind^{\tPP_{\Lambda,\Lambda} \cap L}_{\tJJ_{\Lambda, \Lambda_x, L}} \tilde{\kappa}_{\Lambda,\Lambda_x,L}, \ind^{\tPP_{\Lambda,\Lambda} \cap L}_{\tJJ_{\Lambda, \Lambda_y, L}}  \tilde{\rho}_{\Lambda,\Lambda_y,L} \right) \neq 0.
\]
Notice that
\begin{align*}
	&\Hom_{\tPP_{\Lambda,\Lambda}\cap L} \left(\ind^{\tPP_{\Lambda,\Lambda} \cap L}_{\tJJ_{\Lambda, \Lambda_x, L}} \tilde{\kappa}_{\Lambda,\Lambda_x,L}, \ind^{\tPP_{\Lambda,\Lambda} \cap L}_{\tJJ_{\Lambda, \Lambda_y, L}}  \tilde{\rho}_{\Lambda,\Lambda_y,L} \right)  \\
	\cong & 
	\Hom_{\prod_{i=1}^k \tPP_{\Lambda^i,\Lambda^i}} \left( \ind^{\prod_{i=1}^k \tPP_{\Lambda^i,\Lambda^i}}_{\prod_{i=1}^k \tJJ_{\Lambda^i,\Lambda^i_x}} \bigboxtimes_{i=1}^k \tilde{\kappa}_{\Lambda^i,\Lambda^i_x},  \ind^{\prod_{i=1}^k \tPP_{\Lambda^i,\Lambda^i}}_{\prod_{i=1}^k \tJJ_{\Lambda^i,\Lambda^i_y}} \bigboxtimes\limits_{i=1}^k \tilde{\rho}_{\Lambda^i,\Lambda^i_y} \right) \\
	\cong & \Hom_{\prod_{i=1}^k\tPP_{\Lambda^i,\Lambda^i}} \left( \bigboxtimes\limits_{i=1}^k \ind^{ \tPP_{\Lambda^i,\Lambda^i}}_{\tJJ_{\Lambda^i,\Lambda^i_x}}  \tilde{\kappa}_{\Lambda^i,\Lambda^i_x},  \bigboxtimes\limits_{i=1}^k \ind^{\tPP_{\Lambda^i,\Lambda^i}}_{\tJJ_{\Lambda^i,\Lambda^i_y}} \tilde{\rho}_{\Lambda^i,\Lambda^i_y} \right) \\
	\cong & \bigotimes\limits_{i=1}^k \Hom_{\tPP_{\Lambda^i,\Lambda^i}} \left(  \ind^{ \tPP_{\Lambda^i,\Lambda^i}}_{\tJJ_{\Lambda^i,\Lambda^i_x}}  \tilde{\kappa}_{\Lambda^i,\Lambda^i_x},  \ind^{\tPP_{\Lambda^i,\Lambda^i}}_{\tJJ_{\Lambda^i,\Lambda^i_y}} \tilde{\rho}_{\Lambda^i,\Lambda^i_y} \right).
\end{align*}
Thus, for any $1 \leq i \leq k$, we have $\Hom_{\tPP_{\Lambda^i,\Lambda^i}} \left(  \ind^{ \tPP_{\Lambda^i,\Lambda^i}}_{\tJJ_{\Lambda^i,\Lambda^i_x}}  \tilde{\kappa}_{\Lambda^i,\Lambda^i_x},  \ind^{\tPP_{\Lambda^i,\Lambda^i}}_{\tJJ_{\Lambda^i,\Lambda^i_y}} \tilde{\rho}_{\Lambda^i,\Lambda^i_y} \right) \neq 0$, whence an isomorphism $\ind^{ \tPP_{\Lambda^i,\Lambda^i}}_{\tJJ_{\Lambda^i,\Lambda^i_x}}  \tilde{\kappa}_{\Lambda^i,\Lambda^i_x} \cong \ind^{\tPP_{\Lambda^i,\Lambda^i}}_{\tJJ_{\Lambda^i,\Lambda^i_y}} \tilde{\rho}_{\Lambda^i,\Lambda^i_y}$ by Schur's lemma.
It follows that 
\begin{equation*}
\tilde{\rho}_{\Lambda^i,\Lambda^i_y} \cong \Psi_{\Lambda^i,\Lambda^i_x,\Lambda^i_y}(\tilde{\kappa}_{\Lambda^i,\Lambda^i_x}), \qquad \forall 1 \leq i \leq k.
\end{equation*}
To prove the claim that
$\tilde{\kappa}_{\Lambda^i,\Lambda^i_y} \cong \tilde{\rho}_{\Lambda^i,\Lambda^i_y} ~ (1 \leq i \leq k)$,
it is then equivalent to showing that $\Psi_{\Lambda^i,\Lambda^i_x,\Lambda^i_y}(\tilde{\kappa}_{\Lambda^i,\Lambda^i_x}) \cong \tilde{\kappa}_{\Lambda^i,\Lambda^i_y} ~ (1 \leq i \leq k)$.
Indeed, by our hypothesis~\eqref{H:path-gl} on the path, we may assume that $\tilde{\mathfrak{b}}(\Lambda^i_x) \subseteq \tilde{\mathfrak{b}}(\Lambda^i_y) \subseteq \tilde{\mathfrak{b}}(\Lambda^i_{\mathsf{M}})$ for a vertex of the facet  containing both $x$ and $y$.
By Construction~\ref{C:kappa-gl}, we have 
$\tilde{\kappa}_{\Lambda^i_z}  = \Psi_{\Lambda_z^i,\Lambda_{\Lambda^i_{\mathsf{M}}}, \Lambda_z^i} \left( \tilde{\kappa}_{\Lambda^i_{\mathsf{M}}} |_{\tilde{\JJ}_{\Lambda_z^i,\Lambda^i_{\mathsf{M}}}}\right)$.
For $z\in \{x,y\}$, apply~\cite[Lemma 6.5]{daniel-3}, upon replacing $\Lambda$ (resp.\,$\tilde{\Lambda},\Lambda'$ and $\Lambda''$)
by $\Lambda^i$ (resp.\,$\Lambda^i_z,\Lambda^i_{\mathsf{M}}$ and $\Lambda^i_z$), 
to $\tilde{\kappa}_{\Lambda^i_{\mathsf{M}}}|_{\tilde{\JJ}_{\Lambda_z^i,\Lambda^i_{\mathsf{M}}}}$, we obtain $\tilde{\kappa}_{\Lambda^i,\Lambda^i_z} = \Psi_{\Lambda^i,\Lambda^i_{\mathsf{M}},\Lambda^i_z} \left(\tilde{\kappa}_{\Lambda^i_{\mathsf{M}}}|_{\tilde{\JJ}_{\Lambda^i,\Lambda^i_{\mathsf{M}}}} \right)$,
and hence
\begin{align*}
	\Psi_{\Lambda^i,\Lambda^i_x,\Lambda^i_y} (\tilde{\kappa}_{\Lambda^i,\Lambda^i_x}) & =  \Psi_{\Lambda^i,\Lambda^i_x,\Lambda^i_y} \circ \Psi_{\Lambda^i,\Lambda^i_{\mathsf{M}},\Lambda^i_x}  \left(\tilde{\kappa}_{\Lambda^i_{\mathsf{M}}}|_{\tilde{\JJ}_{\Lambda^i,\Lambda^i_{\mathsf{M}}}} \right) \\
  & = \Psi_{\Lambda^i,\Lambda^i_{\mathsf{M}},\Lambda^i_y} \left(\tilde{\kappa}_{\Lambda^i_{\mathsf{M}}}|_{\tilde{\JJ}_{\Lambda^i,\Lambda^i_{\mathsf{M}}}} \right)
  = \tilde{\kappa}_{\Lambda^i,\Lambda^i_y},
\end{align*}
as desired.

\end{proof}

\shu{
We deduce the following theorem from Proposition~\ref{P:compatible-beta-extension-gl} by imitating the proof of Theorem~\ref{T:compatible-beta-extension}.
}

\begin{thm}\label{T:compatible-beta-extension-gl}
	The family $\{ \tilde{\kappa}^\circ_x \mid x\in\overline{\mathcal{C}}\}$ is a \emph{compatible family of $\beta$-extensions} in the following sense. 

\begin{itemize}
	\item[(i)]  $\tilde{\kappa}_x \in \beta\text{-}\tdext_{\Lambda_{\mathsf{M}}}(\Lambda_x)$ for some vertex $\Lambda_{\mathsf{M}}$ of $\overline{\mathcal{C}}$ with $\tilde{\mathfrak{b}}(\Lambda_x) \subseteq \tilde{\mathfrak{b}}(\Lambda_{\mathsf{M}})$;
\item[(ii)] for any vertices $\Lambda_1, \Lambda_2$ of $\overline{\mathcal{C}}$ such that $\tilde{\mathfrak{b}}(\Lambda_x) \subseteq \tilde{\mathfrak{b}}(\Lambda_{1}) \cap \tilde{\mathfrak{b}}(\Lambda_{2})$, we have
\[
	\tilde{\kappa}_x \cong 
	\Psi_{\Lambda_x,\Lambda_1, \Lambda_x} \left(\Res^{\tJJ_{\Lambda_1}}_{\tJJ_{\Lambda_x, \Lambda_1}}\tilde{\kappa}_1 \right) \cong \Psi_{\Lambda_x,\Lambda_2, \Lambda_x} \left(\Res^{\tJJ_{\Lambda_2}}_{\tJJ_{\Lambda_x, \Lambda_2}}\tilde{\kappa}_2 \right).
\] 
\end{itemize}
\end{thm}

\printbibliography

\end{document}